\documentclass[11pt]{article}

\usepackage{amsmath}
\usepackage{amsthm}
\usepackage{amsfonts}
\usepackage{graphicx}
\usepackage{subfigure}
\usepackage{pinlabel}
\usepackage{enumerate}
\usepackage{sectsty}
\usepackage{fullpage}

\theoremstyle{plain} \newtheorem{thm}{Theorem}[subsection]
\theoremstyle{plain} \newtheorem{cor}[thm]{Corollary}
\theoremstyle{plain} \newtheorem{prop}[thm]{Proposition}
\theoremstyle{plain} \newtheorem{conj}[thm]{Conjecture}
\theoremstyle{plain} \newtheorem{lem}[thm]{Lemma}
\theoremstyle{definition} \newtheorem{df}[thm]{Definition}
\theoremstyle{remark} \newtheorem{rmk}[thm]{Remark}
\theoremstyle{remark} 
\newcommand{\Bn}{B_{2n}}

\newcommand{\B}[1]{B_{#1} }

\newcommand{\DBCs}[1]{\Sigma(#1)\#(S^{2}\times S^{1})}
\newcommand{\SxS}{S^{2}\times S^{1}}
\newcommand{\sfr}{\mathfrak{s}}
\newcommand{\sO}{\mathfrak{s}_{0}}
\newcommand{\DBC}[1]{\Sigma(#1)}

\newcommand{\Ztil}{\tld{\mathcal{Z}}}

\newcommand{\Zcal}{\mathcal{Z}}

\newcommand{\G}{\mathcal{G}}
\newcommand{\Gtr}{\underline{G}}

\newcommand{\Fcal}{\mathcal{F}}

\newcommand{\DD}{\mathcal{D}}

\newcommand{\h}{\mathcal{H}}
\newcommand{\del}{\partial}
\newcommand{\tld}[1]{\widetilde{#1}}

\newcommand{\Gtil}{\tld{\mathcal{G}}}

\newcommand{\conf}[1]{\text{Conf}\,^{#1}(\mathbb{C})}
\newcommand{\red}[1]{\underline{#1}}
\newcommand{\OS}{Ozsv\'ath and Szab\'o }
\newcommand{\Symg}{\text{Sym}^{g}(\Sigma)}

\newcommand{\HFM}{\widehat{HF}(M)}
\newcommand{\Us}{\mathfrak{U}_{\mathfrak{s}}}

\newcommand{\HFxs}[1]{\widehat{HF}(\DBCs{#1}, \mathfrak{s})}

\newcommand{\CF}[1]{\widehat{CF}(\DBC{K})}
\newcommand{\HF}[1]{\widehat{HF}(\DBC{K})}

\newcommand{\Sc}[1]{\text{Spin}^{c}(\DBC{#1})}

\newcommand{\ah}{\widehat{\alpha}}
\newcommand{\bh}{\widehat{\beta}}

\newcommand{\ba}{\boldsymbol{\alpha}}
\newcommand{\bb}{\boldsymbol{\beta}}
\newcommand{\bg}{\boldsymbol{\gamma}}

\newcommand{\bah}{\widehat{\ba}}
\newcommand{\bbh}{\widehat{\bb}}

\newcommand{\bat}{\boldsymbol{\tld{\alpha}}}
\newcommand{\bbt}{\boldsymbol{\tld{\beta}}}

\newcommand{\Tah}{\mathbb{T}_{\bah}}
\newcommand{\Tbh}{\mathbb{T}_{\bbh}}

\newcommand{\Tahr}{\red{\mathbb{T}}_{\bah}}
\newcommand{\Tbhr}{\red{\mathbb{T}}_{\bbh}}

\newcommand{\Ta}{\mathbb{T}_{\ba}}

\newcommand{\Tb}{\mathbb{T}_{\bb}}

\newcommand{\Tg}{\mathbb{T}_{\bg}}

\newcommand{\Tap}{\tor{\ba'}}
\newcommand{\Tbp}{\tor{\bb'}}
\newcommand{\bx}{\mathbf{x}}
\newcommand{\be}{\mathbf{e}}
\newcommand{\by}{\mathbf{y}}

\newcommand{\bw}{\mathbf{w}}

\newcommand{\bz}{\mathbf{z}}

\newcommand{\thetabb}{\boldsymbol{\theta}_{\bb \bb'}}

\newcommand{\thet}[1]{\boldsymbol{\theta}_{#1}}

\newcommand{\tor}[1]{\mathbb{T}_{#1}}

\newcommand{\AD}{\nabla}
\newcommand{\Zcaltwo}{\mathbb{Z}/2\mathbb{Z}}

\newcommand{\F}[1]{\mathbb{F}\left< #1 \right>}
\newcommand{\bda}{\boldsymbol{a}}

\numberwithin{equation}{section}
\begin{document}
\allsectionsfont{\mdseries\itshape}
\sectionfont{\centering \mdseries \itshape}

\title{The anti-diagonal filtration: reduced theory and applications} 
\author{Eamonn Tweedy} 
\maketitle

\begin{abstract} 
Given a knot $K \subset S^3$, Seidel and Smith described a graded cohomology group $Kh_{symp,inv}(K)$, a variant of their symplectic Khovanov cohomology group \cite{ss:R2}.  They also constructed a spectral sequence converging to the Heegaard Floer homology group $\widehat{HF}(\DBC{K}\#(S^{2}\times S^{1}))$ with $E^{1}$-page isomorphic to a summand of $Kh_{symp,inv}(K)$.  A previous paper \cite{et:R} showed that the higher pages of this spectral sequence are knot invariants.  Here we discuss a reduced version of the spectral sequence which directly computes $\widehat{HF}(\DBC{K})$.  Under some degeneration conditions, one obtains a new absolute Maslov grading on the group $\widehat{HF}(\DBC{K})$.  This occurs when $K$ is a two-bridge knot, and we compute the grading in this case.  We also extract some $\mathbb{Q}$-valued knot invariants from this construction.
\end{abstract} 

\clearpage

\section{INTRODUCTION}\label{sec:intro}

Let $K \subset S^{3}$ be a knot, and let $\DBC{K}$ denote the double cover of $S^{3}$ branched along $K$.  This paper is a continuation of a previous one \cite{et:R} in which we studied the invariance properties of the spectral sequence whose $E^{1}$ page is isomorphic to a direct summand of a particular variant of Seidel and Smith's symplectic Khovanov cohomology and which converges to the Heegaard Floer homology group $\widehat{HF}(\DBC{K}\#(S^{2}\times S^{1}))$.  In \cite{et:R}, we proved that the filtered chain homotopy type of the filtered chain complex inducing this spectral sequence is a knot invariant; this implies that the higher pages of the spectral sequence are knot invariants also.  The present paper will give a definition for a reduced version of the theory in the form of a filtration on the Heegaard Floer chain complex $\widehat{CF}(\DBC{K})$.

Let $b \in \B{2n}$ be a braid whose plat closure is a diagram of the knot $K$.  Reviewing Manolescu's construction in \cite{cm:R} and following the work of Bigelow \cite{big:jones}, we described in \cite{et:R} how to define a fork diagram for $b$ and compute a function $R: \G \rightarrow \mathbb{Z}+\frac{1}{2}$, where $\G$ is a set of Bigelow generators in the diagram.  The notion of a fork diagram, as well as the definitions of the gradings $T$ and $Q$ are originally due to Bigelow \cite{big:jones}.  As suggested in \cite{cm:R}, one can define a reduced version of $R$ (denoted by $\red{R}$) by considering a \textit{reducible fork diagram};  a set $\red{\G}$ of reduced Bigelow generators are determined by omitting a pair of arcs from the fork diagram in a prescribed way.  We'll review the notion of a fork diagram in more detail in Section \ref{sec:redR} below, as well as describe how to reduce them and how to define reduced versions of the gradings $P$, $Q$, and $T$ that appear in \cite{cm:R}.  A certain holomorphic volume form is used in \cite{et:R} to define a grading $\tld{R}$ on the unreduced Bigelow generators, and here we'll use an analogous form to define a reduced version $\red{\tld{R}}:\red{\G} \rightarrow \mathbb{Z}$.  One can in fact compute this grading via the formula
\begin{equation*}
\red{\tld{R}} = \red{T} - \red{Q} + \red{ P},
\end{equation*}
where the functions $\red{Q},\red{T}, \red{ P}:\red{\G} \rightarrow \mathbb{Z}$ are analogous to their unreduced counterparts $Q$, $T$, and $P$ found in \cite{et:R}.

We acquire $\red{R}$ from $\red{\tld{R}}$ via a rational shift $s_{\red{R}}$; let
\begin{equation*}
\red{R} = \red{\tld{R}} + s_{\red{R}}(b,D),\quad \text{where} \quad s_{\red{R}}(b,D) = \frac{e(b) - w(D) -2(n-1)}{4}.
\end{equation*}

These reduced Bigelow generators are in one-to-one correspondence with a set of generators for $\widehat{CF}(\DBC{K})$ (notice that the $\SxS$ summand found in the unreduced theory has been removed).  Since $\DBC{K}$ is a rational homology sphere, every $\mathfrak{s} \in \text{Spin}^{c}(\DBC{K})$ is torsion and so the entire complex $\widehat{CF}(\DBC{K})$ carries the $\mathbb{Q}$-valued absolute grading $\tld{gr}$ defined by \OS in \cite{os:tri}.  Thus we can define a filtration grading $\red{\rho}$ on $\widehat{CF}(\DBC{K})$ via
\begin{equation*}
\red{\rho} = \red{R} - \tld{gr}
\end{equation*}

Let $P$ stand for either the ring of integers $\mathbb{Z}$ or for a field $\mathbb{F}$.  Now consider two distinguished filtered chain complexes $V$ and $W$ of free modules over $P$ defined by
$$ V^P_{*} := H_{*}\left(S^1; P\right) \quad \text{and} \quad W^P_{*}:= H_{*}\left( S^0; P \right)$$
Now promote $V$ and $W$ to $\mathbb{Z}$-filtered complexes by declaring each to supported in filtration level $0$.

One would hope to relate the reduced and unreduced theories.  Observe that the functions $(R+1/2)$ and $\red{R}$ also provide filtrations on the Heegaard Floer complexes; Proposition \ref{prop:RredR} is proved in Section \ref{sec:RredR}, and provides a correspondence between the reduced and unreduced complexes.  Note that we use $(R+1/2)$ rather than $R$ because $R$ is $(\mathbb{Z}+1/2)$-valued for knots.

\begin{prop}\label{prop:RredR}
Let $P$ either stand for $\mathbb{Z}$ or a field $\mathbb{F}$.  Let $b \in \B{2n}$ be a braid which induces a reducible fork diagram and whose closure is a diagram for the knot $K$.  Let $\red{\h}$ (respectively $\h$) be the Heegaard diagram for $\DBC{K}$ (respectively $\DBCs{K}$) provided by Proposition \ref{prop:DBCred} below (respectively Proposition 4.2.1 from \cite{et:R}).  Let $\mathfrak{s} \in \text{Spin}^{c}(\DBC{K})$ and let $\mathfrak{s}_{0} \in \text{Spin}^{c}(\SxS)$ denote the torsion element.  Equip $\widehat{CF}(\h, \mathfrak{s}\#\mathfrak{s}_{0}; P)$ with the $(R+1/2)$-filtration and equip $\widehat{CF}(\red{\h}, \mathfrak{s}; P)$ with the $\red{R}$-filtration.  Let $V^P$ and $W^P$ be the filtered complexes defined above.  Then the filtered complexes
\begin{equation*}
\widehat{CF}(\h, \mathfrak{s} \# \mathfrak{s}_{0}; P) \quad \text{and} \quad \widehat{CF}(\red{\h}, \mathfrak{s}; P) \otimes_{P} V^P
\end{equation*}
have the same filtered chain homotopy type.  Furthermore, equipping  $\widehat{CF}(\h, \mathfrak{s}\#\mathfrak{s}_{0};P)$ with the $\rho$-filtration and $\widehat{CF}(\red{\h}, \mathfrak{s};P)$ with the $\red{\rho}$-filtration,
\begin{equation*}
\widehat{CF}(\h, \mathfrak{s} \# \mathfrak{s}_{0};P) \quad \text{and} \quad \widehat{CF}(\red{\h}, \mathfrak{s};P) \otimes_{P} W^P
\end{equation*}
have the same filtered chain homotopy type.
\end{prop}

One would like to use Proposition \ref{prop:RredR} in conjunction with Theorem 1.0.1 from \cite{et:R} (the invariance result for the unreduced theory) to obtain an invariance result for the reduced theory.  However, this requires recovering the filtered chain homotopy type of $\widehat{CF}(\red{\h}, \mathfrak{s}; P)$ from the filtered chain homotopy type of $\widehat{CF}(\red{\h}, \mathfrak{s}; P) \otimes_{P} V^P$ (and from that of $\widehat{CF}(\red{\h}, \mathfrak{s};P) \otimes_{P} W^P$).  The following fact indicates that this is possible when $P = \mathbb{F}$, a field.

\begin{prop}\label{prop:cancel}
Let $\mathbb{F}$ be a field, and let $C_1$ and $C_2$ be finite-dimensional $\mathbb{Z}$-graded, $\mathbb{Z}$-filtered chain complexes of vector spaces over $\mathbb{F}$, and let $X$ denote either one of the filtered complexes $V^{\mathbb{F}}$ or $W^{\mathbb{F}}$ defined above.  If $C_1 \otimes_{\mathbb{F}} X$ is filtered chain homotopy equivalent to $C_2 \otimes_{\mathbb{F}} X$, then $C_1$ is filtered chain homotopy equivalent to $C_2$.
\end{prop}

Along with Theorem 1.0.1 from \cite{et:R} and Proposition \ref{prop:RredR}, Proposition \ref{prop:cancel} implies the following:

\begin{thm}\label{thm:redRthm}
Fix $F$ to either stand for $\red{R}$ or $\red{\rho}$, and let $\mathbb{F}$ be a field.  Let the braids $b \in B_{2n}$ and $b' \in B_{2m}$ have plat closures which are both diagrams for the knot $K$, and assume that both induce reducible fork diagrams.  Let $\h$ and $\h'$ be the pointed Heegaard diagrams for $\DBC{K}$ induced by $b$ and $b'$,respectively, in the sense of Proposition \ref{prop:DBCred} below.  Then the $F$-filtered chain complexes
\begin{equation*}
\widehat{CF}(\h; \mathbb{F}) \quad \text{and} \quad \widehat{CF}(\h'; \mathbb{F})
\end{equation*}
have the same filtered chain homotopy type.
\end{thm}


\begin{rmk}\label{rmk:Rinv}
Theorem 1.0.1 of \cite{et:R} was in fact only stated with respect to the filtration $\rho$.  However, it is implicit in the proof of that result that an analogous invariance statement holds for the filtration $R$ as well.
\end{rmk}

The filtration $\red{\rho}$ induces a reduced version of the spectral sequence (over $\mathbb{F}$) with pages $\red{E}^{k}$, and Theorem \ref{thm:redRthm} implies the following.

\begin{cor}\label{cor:redRss}
For $k \geq 1$, the page $\red{E}^{k}$ is a knot invariant.
\end{cor}

One therefore obtains the following relationship between the pages of the spectral sequences induced by $\rho$ and $\red{\rho}$, indicating that the reduced spectral sequence determines the unreduced one:
\begin{cor}\label{cor:ss}
For $k \geq 1$, $E^{k} \cong \red{E}^{k} \oplus \red{E}^{k}$ as $\mathbb{Z}^2$-graded $\mathbb{F}$-vector spaces.
\end{cor}

In Section \ref{sec:sum}, we show that the reduced theory enjoys a K\"unneth-type theorem with respect to connected sums of knots; this result provides a computational tool for composite knots.

\begin{thm}\label{thm:sum}
Fix $F$ to either stand for $\red{R}$ or $\red{\rho}$, and let $\mathbb{F}$ be a field.  Let $K_{1}, K_{2} \subset S^{3}$ be knots, let $\mathfrak{s}_{i} \in \text{Spin}^{c}(\DBC{K_{i}}), i= 1, 2$.  Then the filtered chain complexes
\begin{equation*}
\widehat{CF}(\DBC{K_{1} \# K_{2}}, \mathfrak{s_{1}} \# \mathfrak{s_{2}}; \mathbb{F}) \quad \text{and} \quad 
	\widehat{CF}(\DBC{K_{1}}, \mathfrak{s_{1}}; \mathbb{F})
	\otimes_{\mathbb{F}}
	\widehat{CF}(\DBC{K_{2}}, \mathfrak{s_{2}};\mathbb{F})
\end{equation*}
have the same filtered chain homotopy type, where $\widehat{CF}(\DBC{K_{1} \# K_{2}}, \mathfrak{s_{1}} \# \mathfrak{s_{2}};\mathbb{F})$ is equipped with the $F$-filtration and $\widehat{CF}(\DBC{K_{1}}, \mathfrak{s_{1}};\mathbb{F}) \otimes_{\mathbb{F}} \widehat{CF}(\DBC{K_{2}}, \mathfrak{s_{2}};\mathbb{F})$ is equipped with the tensor product filtration induced by the $F$-filtrations on the factors.
\end{thm}

Theorem \ref{thm:sum} implies the following fact regarding the reduced spectral sequence:

\begin{cor}\label{cor:sssum}
For $k \geq 1$, $\red{E}^{k}(K_{1} \# K_{2}) \cong \red{E}^{k}(K_{1}) \otimes_{\mathbb{F}}\red{E}^{k}(K_{2})$.
\end{cor}

Given a knot $K$ and some $\mathfrak{s} \in \text{Spin}^{c}(\DBC{K})$, we denote by $\widehat{CF}^{*}(\DBC{K},\mathfrak{s})$ the dual complex to $\widehat{CF}_{*}(\DBC{K},\mathfrak{s})$.  Then $\widehat{CF}^{-*}(\DBC{K}, \mathfrak{s})$ is a chain complex, and we define a filtration $\red{\rho}^{*}$ via
\begin{equation*}
\red{\rho}^{*}(x^{*}) = -\red{\rho}(x) \quad \text{for each } x\in \Ta\cap \Tb.
\end{equation*}
\begin{thm}\label{thm:mirror}
Fix $F$ to either stand for $\red{R}$ or $\red{\rho}$, and let $\mathbb{F}$ be a field.  Let $-K$ denote the mirror image of the knot $K\subset S^{3}$, and let $\mathfrak{s} \in \text{Spin}^{c}(\DBC{K})$.  Then
\begin{equation*}
\widehat{CF}_{*}(\DBC{-K},\mathfrak{s}; \mathbb{F}) \quad \text{and} \quad \widehat{CF}^{-*}(\DBC{K},\mathfrak{s}; \mathbb{F})
\end{equation*}
are filtered chain isomorphic, where the complexes carry the filtrations $\red{\rho}$ and $\red{\rho}^{*}$, respectively.
\end{thm}

Section \ref{sec:r} will discuss how one can distill a family of knot invariants from the reduced filtration in the form of a function $r_{K}: \text{Spin}^{c}(\DBC{K}; \mathbb{F}) \rightarrow \mathbb{Q}$.  Restricting $r_{K}$ to the unique Spin structure $\mathfrak{s}_{0}$ on $\DBC{K}$, one obtains the $\mathbb{Q}$-valued knot invariant $r(K) = r_{K}(\mathfrak{s}_{0})$.  Theorems \ref{thm:sum} and \ref{thm:mirror} imply the following:

\begin{cor}\label{cor:sum}
Let $K_{1},K_{2} \subset S^{3}$ be knots, and let $\mathfrak{s}_{i} \in \text{Spin}^{c}(\DBC{K_{i}})$ for $i = 1,2.$  Then
$$r_{K_{1} \# K_{2}}(\mathfrak{s}_{1} \# \mathfrak{s}_{2})  = r_{K_{1}}(\mathfrak{s}_{1}) + r_{K_{2}}(\mathfrak{s}_{2}) \quad
\text{and} \quad r(K_{1} \# K_{2}) = r(K_{1}) + r(K_{2}).$$
\end{cor}

\begin{cor}\label{cor:mirror}
Let $K\subset S^{3}$ be a knot and let $\mathfrak{s} \in \text{Spin}^{c}(\DBC{K})$.  Then
$$r_{-K}(\mathfrak{s})  = - r_{K}(\mathfrak{s}) \quad
\text{and} \quad r(-K) = -r(K).$$
\end{cor}

In \cite{ss:R2}, Seidel and Smith conjectured the existence of a concordance invariant arising from this theory.  Motivated by their suggestion and by Corollaries \ref{cor:sum} and \ref{cor:mirror}, we make the following speculation:

\begin{conj}\label{conj:conc}
Let $\mathcal{C}$ denote the smooth knot concordance group.  The knot invariant $r(K)$ provides a well-defined group homomorphism
$$ r: \mathcal{C} \rightarrow \mathbb{Q}.$$
\end{conj}

We defined in \cite{et:R} the notion of $\rho$-degeneracy of a knot (see Definition \ref{def:deg} below).  The following is a consequence of Proposition \ref{prop:RredR}:
\begin{prop}\label{prop:Rdeg}
Let $K \subset S^{3}$ be a $\rho$-degenerate knot.  Then the following hold:
\begin{enumerate}[(i)]
\item The filtration on $\widehat{HF}(\DBC{K}; \mathbb{F})$ induced by $\red{R}$ lifts the relative Maslov $\mathbb{Z}$-grading on each nontrivial factor $\widehat{HF}(\DBC{K}, \mathfrak{s};\mathbb{F})$  \label{Rdeg:2}
\item The grading $\red{R}$ is an invariant of $K$.
\end{enumerate}
\end{prop}

One sees this behavior when $K$ is a two-bridge knot.

\begin{thm}\label{thm:Rsignthm}
Let $\mathbb{F}$ be a field.  Let $K \subset S^{3}$ be a two-bridge knot.  Then $K$ is $\rho$-degenerate, and
\begin{equation*}
\text{dim}_{\mathbb{F}}\left( \widehat{HF}_{\red{R} = k} (\DBC{K}; \mathbb{F}) \right)= \begin{cases}
\text{det}(K)& \text{if } k = \frac{\sigma(K)}{2}\\
0 & \text{otherwise},
\end{cases}
\end{equation*}
where $\sigma(K)$ denotes the classical signature of $K$ and $\text{det}(K)$ denotes the determinant of $K$.
\end{thm}

We speculate the following, as suggested by Seidel and Smith in \cite{ss:R2}:

\begin{conj}\label{conj:ss}
Every knot $K \subset S^{3}$ is $\rho$-degenerate.
\end{conj}

The constructions in this paper and in \cite{et:R} can be expanded to links, with the limitation that one can only define the filtrations $\rho$ and $\red{\rho}$ on the summands of the $\widehat{CF}$ complexes corresponding the torsion $\text{Spin}^{c}$-structures.  We'll make use of this extension in the in Section \ref{sec:RredR} - see Remark \ref{rmk:links}.  The situation for links is discussed further in Section \ref{sec:links}.

\begin{rmk}
From now on, we'll only explicitly include the coefficient ring in the notation when the distinction is necessary.  The background constructions of Section \ref{sec:HF} and constructions of Section \ref{sec:redR} through Section \ref{sec:RredR}  work over either of $\mathbb{Z}$ or over $\mathbb{F}$.
\end{rmk}




\section{HEEGAARD FLOER THEORY}\label{sec:HF}

In \cite{os:disk}, \OS define the Heegaard Floer homology group $\widehat{HF}(M)$ associated to a connected, closed, oriented 3-manifold $M$.  A genus-g Heegaard splitting for such a manifold can be described via a \textit{pointed Heegaard diagram} $\h_{\ba\bb} = \left(\Sigma; \ba; \bb; z\right)$, where $\Sigma$ is the splitting surface, $\ba$ and $\bb$ are g-tuples of attaching curves for the handlebodies, and $z \in (\Sigma \setminus \cup \alpha_{i} \setminus \cup \beta_{i})$.  Recall the following definitions:

\begin{df}\label{df:pd}
Let $\left(\Sigma; \ba; \bb; z\right)$ be a pointed Heegaard diagram, and let $D_1, \ldots, D_m$ be the connected components of $\Sigma \setminus \left( \cup \alpha_i \right) \setminus \left( \cup \beta_i \right)$, where $z \in D_m$.  Then a two-chain
$$ \mathcal{P}:= \sum_{i =1}^{m-1}n_{i} D_i \quad \text{with} \quad n_i \in \mathbb{Z}$$
is called a \textit{periodic domain} if its boundary is a sum of $\alpha$ and $\beta$ circles.
\end{df}

\begin{df}\label{df:ad}
A Heegaard diagram $\left(\Sigma; \ba; \bb; z\right)$ is called \textit{admissible} if every periodic domain has both positive and negative coefficients.
\end{df}

Given an admissible pointed Heegaard diagram one can compute the group $\widehat{HF}(M)$, which is the Lagrangian Floer homology of the tori $\Ta := \alpha_1 \times \ldots \times \alpha_g$ and $\Tb:= \beta_1 \times \ldots \times \beta_g$ lying inside of the symplectic manifold $\text{Sym}^g(\Sigma \setminus z)$.

More precisely, the group $\widehat{CF}(\h_{\ba\bb})$ is generated by the set of intersections $\Ta \cap \Tb \subset \Symg$, and the differential is given by
$$\widehat{\partial}(\bx) =
\displaystyle\sum_{\by \in \Ta \cap \Tb} \left( \displaystyle\sum_{\{\phi \in \pi_{2}(\bx, \by)| \mu(\phi) = 1, n_{z}(\phi) = 0\}} \left(\# \widehat{\mathcal{M}}\left(\phi\right)\right)\right)\by,
$$
where $\pi_2(\bx,\by)$ denotes the group of homotopy classes of 2-gons connecting $\bx$ and $\by$, $\widehat{\mathcal{M}}(\phi)$ denotes the reduced moduli space of pseudo-holomorphic representatives for the class $\phi$, $\mu(\phi)$ denotes the Maslov index of $\phi$, and $n_z(\phi):=\text{Im}(\phi) \cap \left( \{ z \} \cap \text{Sym}^{g-1}(\Sigma)\right).$

Classes of such 2-gons are typically studied by analyzing their ``shadows'' in the surface $\Sigma$.

\begin{df}
Let $\left( \Sigma; \ba ; \bb; z \right)$ be a pointed Heegaard diagram, and denote by $\mathcal{D}_0, \mathcal{D}_1, \ldots, \mathcal{D}_N$ the connected components of $\Sigma \setminus \left( \cup_i \alpha_i\right) \setminus \left( \cup_i \beta_i \right),$ where $\mathcal{D}_0$ is the component containing the basepoint $z$.  Then for $0 \leq j \leq N$, choose a point $z_j$ in the interior of $\mathcal{D}_j$.  For some class $\phi \in \pi_2(\bx, \by)$ for $\bx, \by \in \Ta \cap \Tb$, the \textit{domain of $\phi$}  is the 2-chain
$$\mathcal{D}(\phi):= \sum_{j = 0}^{N} n_j \mathcal{D}_j \quad \text{where} \quad n_j:= \text{Im}(\phi) \cap \left(\left \{ z_j \right\} \times \text{Sym}^{g-1}(\Sigma) \right).$$
We'll say that $\phi$ \textit{avoids the basepoint} if $n_0 = 0$ (equivalently, $n_z(\phi) = 0$).
\end{df}

Recall that there is a function
\begin{equation*}
\mathfrak{s}_{z}: \Ta \cap \Tb \longrightarrow \text{Spin}^{c}(M)
\end{equation*}
partitioning $\Ta \cap \Tb$ into equivalence classes $\Us$.  This function induces decompositions
\begin{equation*}
\widehat{CF}(\h)  = \displaystyle \bigoplus_{\mathfrak{s} \in \text{Spin}^{c}(M)} \widehat{CF}(\h, \mathfrak{s}) \quad \text{and} \quad
\HFM = \displaystyle \bigoplus_{\mathfrak{s} \in \text{Spin}^{c}(M)} \widehat{HF}(M, \mathfrak{s}).
\end{equation*}

For each $\mathfrak{s} \in \text{Spin}^{c}(M)$ the chain complex $\widehat{CF}(M,\mathfrak{s})$ carries a relative grading $gr$ defined via the Maslov index.  For $\mathfrak{s} \in \text{Spin}^{c}(M)$ torsion, \OS use surgery cobordisms to construct in \cite{os:tri} an absolute $\mathbb{Q}$-valued grading $\tld{gr}$ on $\Us$ which lifts the relative grading in the following sense: if $\bx, \by \in \Us$, then
\begin{equation*}
\tld{gr}(\bx) - \tld{gr}(\by) = gr(\bx, \by). 
\end{equation*}
Whenever $b_{1}(M) = 0$, all $\text{Spin}^{c}$ structures on $M$ are torsion and so the group $\widehat{HF}(M)$ can be absolutely graded via $\tld{gr}$.  In particular, this holds for $M =\DBC{K}$ for a knot $K \subset S^{3}$.  Although $\text{Spin}^{c}(\DBCs{K})$ contains non-torsion elements, the group $\HFxs{K}$ is nontrivial only if $\mathfrak{s}$ is torsion.  

There is an analogous notion of a \textit{pointed Heegaard triple-diagram} (resp. \textit{quadruple-diagram}), and one can study \textit{triply-periodic domains} (resp. \textit{quadruply-periodic domains}) in such a diagram; a pointed triple-diagram (resp. quadruple-diagram) is called \textit{admissible} if every triply-periodic (resp. quadruply-periodic) domain has both positive and negative coefficients.

A cobordism between closed 3-manifolds can be described be an admissible pointed triple-diagram $\left(\Sigma; \ba; \bb; \bg; z\right)$, which in turn induces a chain map $\widehat{f}_{\alpha\beta\gamma}: \widehat{CF}(\h_{\ba\bb}) \otimes  \widehat{CF}(\h_{\bb\bg}) \rightarrow  \widehat{CF}(\h_{\ba\bg})$ given by
\begin{equation*}
\widehat{f}_{\alpha\beta\gamma}(\bx \otimes \by) =
\displaystyle\sum_{\bw \in \Ta \cap \Tg} \left( \displaystyle\sum_{\{\psi \in \pi_{2}(\bx, \by, \bw)| \mu(\psi) = 0, n_{z}(\psi) = 0\}} \left(\# \mathcal{M}\left(\psi\right)\right)\right)\bw
\end{equation*}
Here $\pi_2(\bx, \by, \bw)$ is the space of homotopy classes of 3-gons connecting $\bx$, $\by$, and $\bw$, $\mathcal{M}\left( \psi \right)$ is the moduli space of pseudo-holomorphic representatives for the class $\psi$, and $\mu(\psi)$ is the Maslov index of $\psi$.

We'll be particularly interested in maps induced by Heegaard moves.

\begin{df}\label{def:moves}
Let $\left(\Sigma; \ba; \bb; \bb'; z\right)$ be a pointed Heegaard triple-diagram.
\begin{enumerate}[(i)]
\item Let $\beta'_j$ differ from $\beta_j$ by an isotopy (avoiding $z$) such that $\beta'_j$ intersects $\beta_j$ transversely in two canceling points and $\beta_j \cap \beta'_i = \emptyset$ when $i \neq j.$  Then we say that $\bb'$ differs from $\bb$ by a \textit{pointed isotopy}.  A pointed isotopy which preserves the set of intersection points $\tor{\ba}\cap\tor{\bb}$ in the obvious way will be called a \textit{small pointed isotopy}.
\item Instead let $\beta_{1}$, $\beta_{2}$, and $\beta_{1}'$ bound an embedded pair of pants disjoint from $z$ such that $\beta'_{1}$ intersects $\beta_{1}$ transversely in two points.  Assume also that $\beta_j \cap \beta'_i = \emptyset$ for $i \neq j$, and that for $i > 1$, $\beta'_{i}$ relates to $\beta_i$ as $(i)$ above.  Then we say that $\bb'$ differs from $\bb$ by a \textit{pointed handleslide}.
\end{enumerate}
\end{df}

When a cobordism is induced by a pointed isotopy or a pointed handleslide relating two diagrams for the same manifold, one can use the chain map described above to define a chain homotopy equivalence.  More precisely, consider an admissible pointed Heegaard quadruple-diagram $\left( \Sigma; \ba; \bb; \bb'; \bbt; z \right)$ such that $\bb'$ differs from $\bb$ by a pointed handeslide or pointed isotopy and such that $\bbt$ differs from $\bb$ by a small pointed isotopy.  There is a distinguished representative $\thet{\bb\bb'} \in \Tb \cap \tor{\bb'}$ for the top-degree generator of $\widehat{HF}(\h_{\beta\beta'})$, and the 3-gon counting chain map $\widehat{f}_{\ba\bb\bb'}(\cdot \otimes \thetabb):\widehat{CF}(\h_{\ba\bb}) \rightarrow \widehat{CF}(\h_{\ba\bb'})$ is a homotopy equivalence whose homotopy inverse is given by $\widehat{f}_{\ba\bb'\bb}(\cdot \otimes \thet{\bb'\bb}):\widehat{CF}(\h_{\ba\bb'}) \rightarrow \widehat{CF}(\h_{\ba\bb})$.  The homotopies relating their compositions to identity maps are constructed from maps counting holomorphic representatives of 4-gons arising in the diagram $\left( \Sigma; \ba; \bb; \bb'; \bbt; z \right)$.

On the other hand, if $\ba'$ differs from $\ba$ by a pointed handleslide or pointed isotopy and $\bat$ differs from $\ba$ by a small pointed isotopy such that the pointed quadruple-diagram $\left( \Sigma; \bat; \ba';\ba; \bb; z \right)$ is admissible, then $\widehat{f}_{\ba'\ba\bb}(\thet{\ba'\ba} \otimes \cdot )$ is a homotopy equivalence with homotopy inverse given by $\widehat{f}_{\ba\ba'\bb}(\thet{ \ba\ba' } \otimes \cdot)$.  For a more detailed description of 3-gon counting chain maps and 4-gon counting chain homotopies, see the review in Section 2.1 of \cite{et:R} or the original discussion in Section 8 of \cite{os:disk}.

The following lemmas assure that certain pointed handleslides and pointed isotopies preserve admissibility of pointed Heegaard diagrams.  Lemma \ref{lem:isoad} is from \cite{sw:CHF}, and the proof appearing there involves a straightforward analysis of domain coefficients.  We'll use a similar method here to prove Lemma \ref{lem:hsad}.

\begin{lem}\label{lem:isoad}(\cite{sw:CHF})
Let $\h$ and $\h'$ be pointed Heegaard diagrams which differ in a local region as shown in Figure \ref{fig:isoad}, and coincide elsewhere.  Then if $\h$ is admissible, so is $\h'$.
\end{lem}

\begin{figure}[h!]
\centering
\begin{minipage}[c]{.60\linewidth}
\subfigure[Before the isotopy]{
\includegraphics[height=12mm]{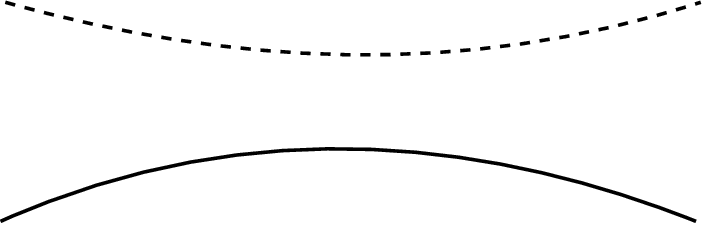}}\qquad
\subfigure[After the isotopy]{
\includegraphics[height=12mm]{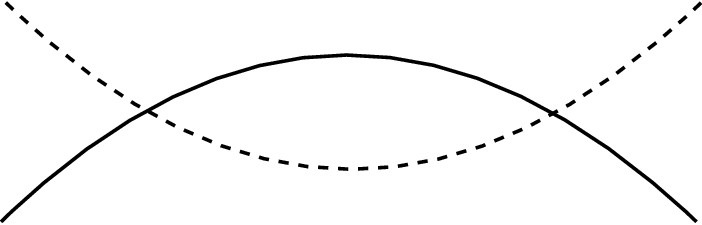}}
\end{minipage}
\begin{minipage}[c]{.35\linewidth}
\caption{A pointed isotopy\label{fig:isoad}}
\end{minipage}
\end{figure}

\begin{lem}\label{lem:hsad}
Let $\h$ and $\h'$ be two pointed Heegaard diagrams such that $\h$ can be obtained from $\h'$ by a handleslide of the form shown in Figure \ref{fig:hsad}.  Then if $\h$ is admissible, so is $\h'$.
 \end{lem}
 
 \begin{figure}[h!]
\centering
\subfigure[Before the handleslide]{
\labellist 
\small
\pinlabel* {$\DD_1$} at 570 180
\pinlabel* {$\DD_2$} at 475 185
\pinlabel* {$\DD_3$} at 350 185
\pinlabel* {$\DD_4$} at 130 285
\pinlabel* {$\DD_5$} at 110 60
\pinlabel* {$\DD_6$} at 260 180
\pinlabel* {$\DD_7$} at 165 230
\pinlabel* {$\DD_8$} at 160 110
\endlabellist
\includegraphics[height=45mm]{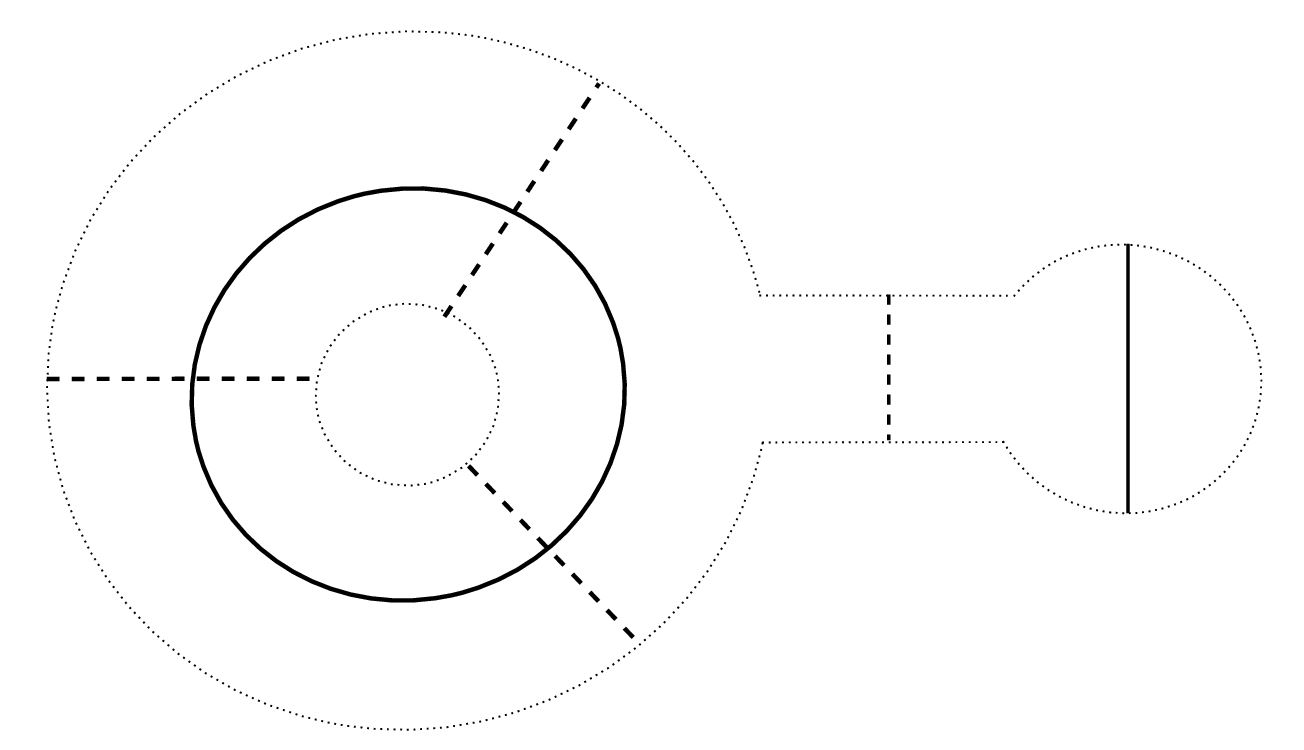}\label{fig:hsad1}}
\subfigure[After the handleslide]{
\labellist 
\small
\pinlabel* {$\DD'_1$} at 570 180
\pinlabel* {$\DD'_2$} at 515 218
\pinlabel* {$\DD'_3$} at 344 224
\pinlabel* {$\DD'_4$} at 134 313
\pinlabel* {$\DD'_5$} at 82 62
\pinlabel* {$\DD'_6$} at 260 180
\pinlabel* {$\DD'_7$} at 165 230
\pinlabel* {$\DD'_8$} at 160 110
\pinlabel* {$\tld{\DD}_3$} at 347 132
\pinlabel* {$\tld{\DD}_2$} at 517 140
\pinlabel* {$\tld{\DD}_6$} at 310 180
\pinlabel* {$\tld{\DD}_7$} at 140 270
\pinlabel* {$\tld{\DD}_8$} at 120 88
\endlabellist
\includegraphics[height=45mm]{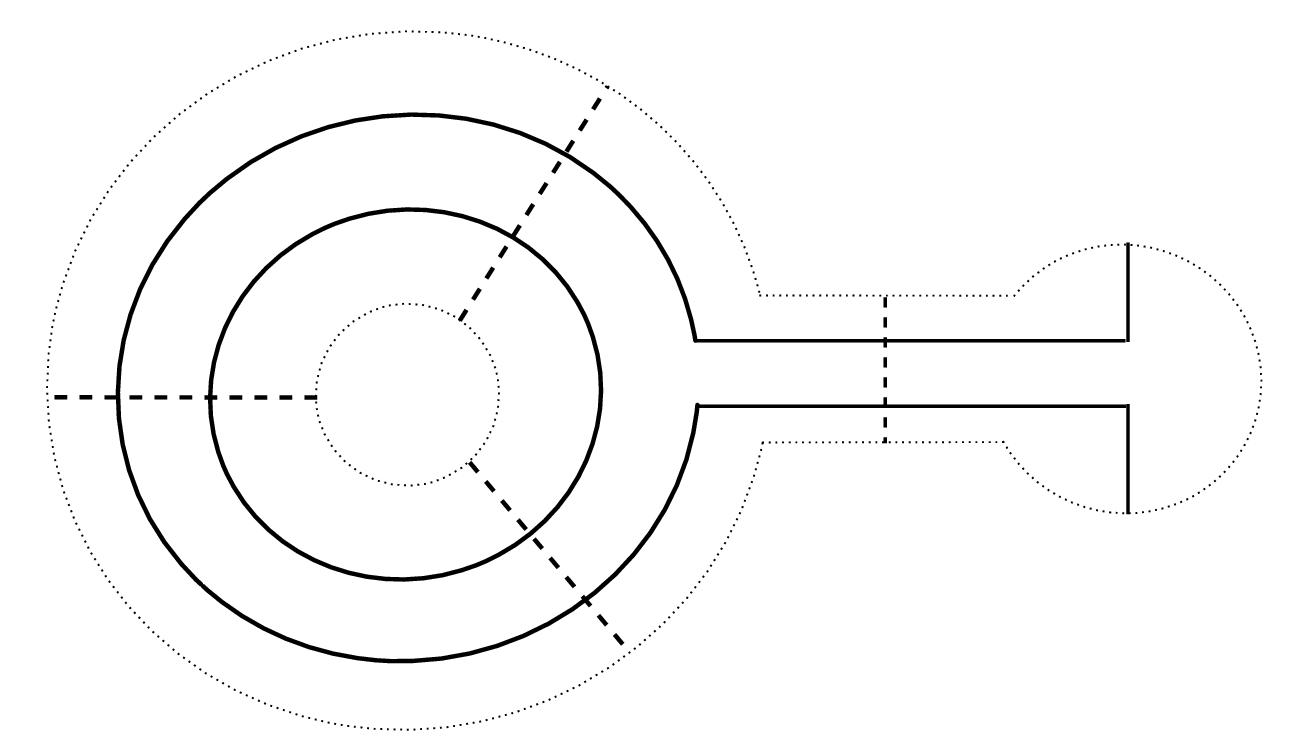}\label{fig:hsad2}}
\caption[Domains for checking admissibility of a certain pointed handleslide]{Heegaard diagrams before and after a certain pointed handeslide.  The $\alpha$ arcs are solid and the $\beta$ arcs are dotted (or vice versa).  There can be arbitrarily many radial arcs on the annulus and arbitrarily many vertical arcs on the ``neck'' in the center.\label{fig:hsad}}
\end{figure}
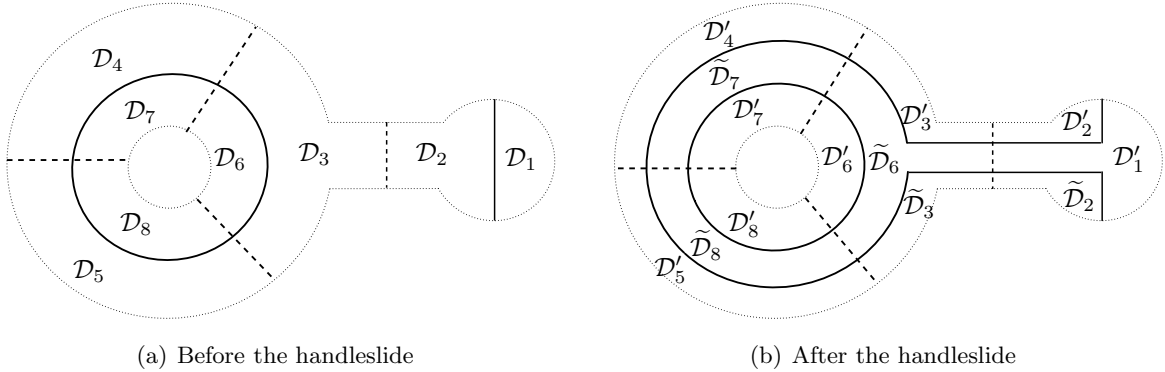
 \begin{proof}
Label the  $n$ (resp. $n+5$) regions of the pointed Heegaard diagram $\h$ (resp. $\h'$) as indicated in Figure \ref{fig:hsad1} (resp. \ref{fig:hsad2}), where $\DD'_{k}$ coincides with $\DD_k$ for $k > 8$.  Now let $\mathcal{P'}$ be a periodic domain in $\h'$ with
$$ \mathcal{P}' = b_2 \tld{\DD}_2 + b_3 \tld{\DD}_3 + b_6 \tld{\DD}_6+ b_7 \tld{\DD}_7 + b_8 \tld{\DD}_8 + \sum_{j = 1}^{n} c_j \DD'_j$$
Now notice that $c_3 -b_6 = c_2-c_1 =  b_3 - b_6 = b_2 - c_1 = c_4 - b_7 = c_5 - b_8$ and $b_6-c_6 = b_7-c_7 = b_8-c_8$.  As a result, $b_2 = c_2$, $b_3 = c_3$, and $c_3-c_6 = c_4 - c_7  = c_5 - c_8$, and so there is a periodic domain $\mathcal{P}$ in $\h$ with
$$ \mathcal{P} = \sum_{j = 1}^n c_j \DD_j.$$
Since $\h$ is admissible, there are both positive and negative integers among the $c_j$.
 \end{proof}

\section{A REDUCED FILTRATION ON $\widehat{CF}(\DBC{K})$}\label{sec:redR}

We describe how to define a reduced version of the $\rho$-filtration, denoted by $\red{\rho}$, which is a $\mathbb{Q}$-valued filtration on the chain complex $\widehat{CF}(\DBC{K})$ (a definition first mentioned by Manolescu in \cite{cm:R}).  This reduced version is much simpler to compute than the unreduced theory, and Theorem \ref{thm:redRthm} gives an invariance result for the reduced filtration.

First we discuss some terminology which we'll use when constructing the reduced filtration.

\subsection{Plat closures of braids}\label{sec:plat}

Let $\Bn$ denote the braid group on $2n$ strands.  This group is generated by $\{ \sigma_{1} , \ldots, \sigma_{2n-1}\}$, where $\sigma_{k}$ denotes a half-twist of the $k^{th}$ strand over the $(k+1)^{st}$ strand.  Given a braid $b \in \Bn$, we can obtain a diagram of a link called the \textit{plat closure} of $b$ by connecting ends of consecutive strands with segments at the top and bottom, as shown in Figure \ref{fig:explat}.

\begin{figure}[h!]
\centering
\centering\includegraphics[height = 30mm]{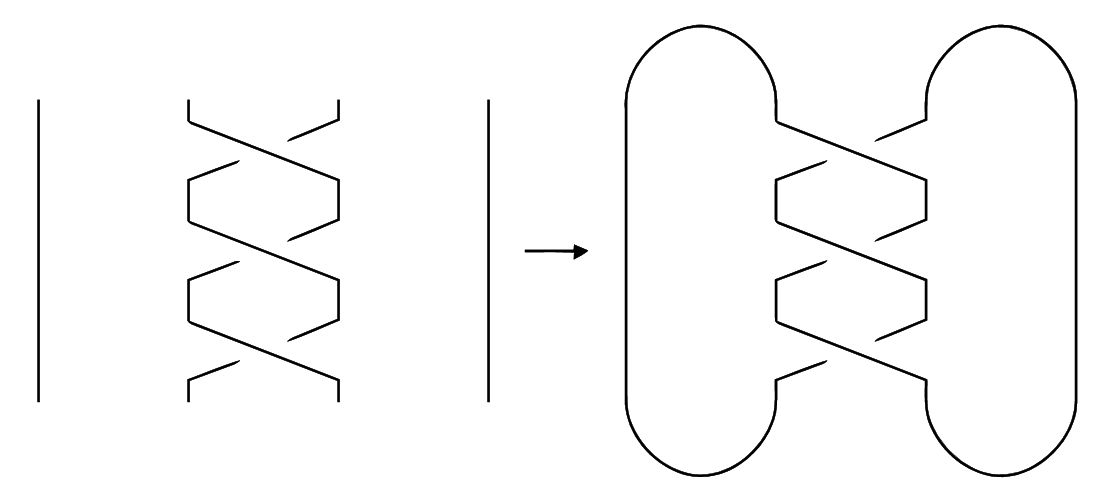}
\caption{The left-handed trefoil is the plat closure of $\sigma_2^3 \in \B{4}.$}\label{fig:explat}
\end{figure}

Clearly many pairs of braids have isotopic links as their plat closures; two braids have isotopic plat closures if and only if they can be related by a finite sequence of standard moves called \textit{Birman moves} (\cite{bir:moves}).

\subsection{Fork diagrams}\label{sec:fork}
Let $D \subset \mathbb{C}$ denote the unit disk, choose $2n$ points $\mu_{1} , \ldots, \mu_{2n}$ evenly spaced along $\mathbb{R} \cap D$, and let the set of punctures be denoted by $\tau$.  We can view the braid group $\B{2n}$ as the mapping class group of the punctured disk $D_{2n}:=D \setminus \tau$, where the generator $\sigma_{k}$ is a diffeomorphism which is the identity outside of a neighborhood of the $k^{th}$ and $(k+1)^{st}$ punctures and exchanges these two punctures by a counter-clockwise half-twist.  Any braid can be written as a word in the $\sigma_{k}$'s, and we view them as operating on $D_{2n}$ in this way, read from left to right.


\begin{df}\label{defforks}
Let $D \subset \mathbb{C}$ be the unit disk and let $b \in \Bn$ be an oriented braid on $2n$ strands.  
\begin{enumerate}[(i)]
\item Let the \emph{standard fork diagram} in $D_{2n}$ be a collection of embeddings
\begin{equation*}
\alpha_{1}, \ldots ,\alpha_{n}:I \rightarrow D \quad \text{and} \quad h_{1}, \ldots , h_{n}:I \rightarrow D
\end{equation*}
called \emph{tine edges} and \emph{handles}, respectively, such that the following hold:
\begin{enumerate}[(a)]
\item The arcs $\left\{\alpha_{k}(I) \right\}_{k=1}^n$ are pairwise disjoint horizontal segments and the arcs $\left\{h_{i}\right\}_{k=1}^n$ are pairwise disjoint vertical segments.
\item For each $k$, we have that
$$\alpha_{k}(0) = \mu_{2k-1}, \quad \alpha_{k}(1) = \mu_{2k}, \quad h_{k}(1) = d_{k} \in \partial D, \quad \text{and} \quad h_{k}(0) = m_{k} = \frac{1}{2}(\mu_{2k-1} + \mu_{2k}).$$
\end{enumerate}
\item Let a \emph{fork diagram for b} be the standard fork diagram along with the compositions $b \circ \alpha_{1}, \ldots, b \circ \alpha_{n}$ and $b \circ h_{1}, \ldots, b \circ h_{n}$.  We'll also let $\beta_{k} := b \circ \alpha_{k}$.
\item Let an \emph{augmented fork diagram for b} be obtained from a fork diagram by replacing each arc $\beta_{k}$ with $bE_{k}$, where $E_{k}$ is an immersed figure-eight which encircles $\mu_{2k-1}$ and $\mu_{2k}$ and is oriented such that it winds counter-clockwise about $\mu_{2k}$.
\end{enumerate}
\end{df}

\begin{figure}[h!]
\centering
\subfigure[Tine edges $\alpha_{k}$]{
\includegraphics[height=32mm]{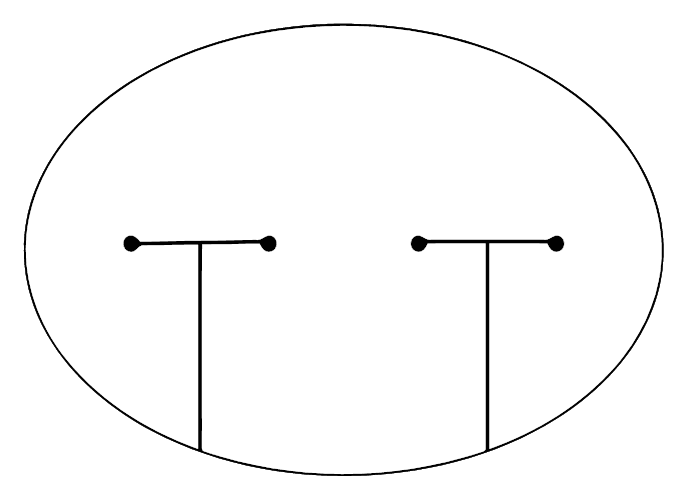}}\quad
\subfigure[Figure-eights $E_{k}$]{
\includegraphics[height=32mm]{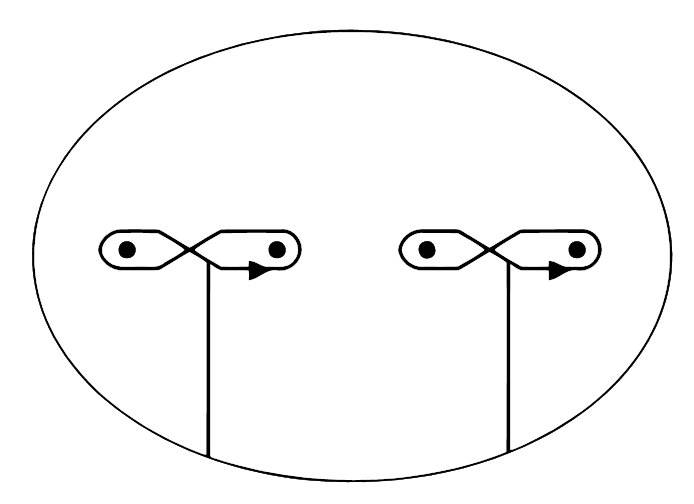}}
\caption{Structures in fork diagrams \label{fig:standardfork}}
\end{figure}

The reader should notice that by drawing a picture containing just the $\alpha$ and $\beta$ arcs and treating the $\alpha$ arcs as undercrossings at each intersection, we get a diagram of the plat closure of $b$.

If the plat closure of the braid $b$ is the knot $K$, then we will often refer to a fork diagram for $b$ as a fork diagram for $K$.

The reduced grading $\red{R}$ is computed for a reduced set $\red{\G}$ of Bigelow generators given by omitting a pair of arcs $\alpha_{n}, \beta_{n}$ from the fork diagram (with a mild restriction on the diagram used).  We'll see that the set $\red{\G}$ is in one-to-one correspondence with a set of generators for $\widehat{CF}(\red{\h})$, where $\red{\h}$ is an admissible Heegaard diagram for the manifold $\DBC{K}$ and is obtained from the reduced fork diagram.

\subsection{The reduced Bigelow generators}\label{sec:redbig}

We can define the reduced filtration for fork diagrams of a special type.
\begin{df}\label{def:red}
A \textit{reducible fork diagram} for a knot $K$ is a fork diagram for $K$ with at least four punctures such that
\begin{itemize}
\item $\mu_{2n} \in \alpha_{n} \cap \beta_{n}$ and
\item There exists an arc which avoids $\displaystyle \left( \bigcup_{i=1}^{n-1} \alpha_i \right) \cup \left( \bigcup_{i=1}^{n-1}\beta_i \right)$ and connects $\mu_{2n}$ to a point on the boundary of the unit disk.
\end{itemize}
\end{df}

Notice that a reducible fork diagram exists for any knot $K$, as one can be obtained by performing a Birman stabilization on any braid whose closure is $K$.  We'll define the reduced theory by omitting the pair of arcs $\alpha_{n}$, $\beta_{n}$ from a reducible fork diagram for $K$ (and subsequently omitting the figure-eight $bE_{n}$ from the augmented fork diagram).

Consider a reducible fork diagram for $K$ induced by a braid $b \in \B{n}$.  Denote by $\red{\Ztil}$ the set of intersections $\alpha_{i} \cap \beta_{j}$, where $i,j \leq n-1$.  Similarly, define $\red{\Zcal}$ to be points $\alpha_{i} \cap b E_{j}$, $i,j \leq n-1$.  We then define
\begin{align*}
\underline{\Gtil} = (\alpha_{1} \times \ldots \times \alpha_{n-1}) \cap (\beta_{1} \times \ldots \times \beta_{n-1}) \subset \conf{n-1}, \\
\underline{\G} = (\alpha_{1} \times \ldots \times \alpha_{n-1}) \cap (bE_{1} \times \ldots \times bE_{n-1}) \subset \conf{n-1},
\end{align*}
where $\conf{k}$ denotes the $k$-fold configuration space of $\mathbb{C}$, i.e. the set of unordered $k$-tuples of distinct points in $\mathbb{C}$.  The relationship between $bE_k$ and $\beta_k$ induces a natural projection map $p: \red{\Zcal} \rightarrow \red{\Ztil}$.  Letting $\red{\tau}$ denote the set containing the leftmost $(2n-2)$ punctures, we see that $\red{\tau} \subset \red{\Ztil}$.  If $x \in \red{\tau}$, $p^{-1}(x)$ contains one point.  If $x \in \red{\Zcal} \setminus \red{\tau}$, then $p^{-1}(x) = \{ e_{x}, e'_{x} \}$, a two-point set; we  distinguish between $e_{x} \text{ and } e'_{x}$ by requiring that the loop traveling along a figure-eight from $e_{x} \text{ to } e'_{x}$ and back to $e_{x}$ along an $\alpha$ arc has winding number +1 around the puncture.

\begin{rmk} Via an abuse of notation, we'll often refer to the points corresponding to $x \in \red{\Ztil}$ as $x \in \red{\Zcal}$ (if $x \in \red{\tau}$) or $x,x' \in \red{\Zcal}$ (if $x \in \red{\Ztil} \setminus \red{\tau}$).
\end{rmk}


\subsection{Reduced gradings on the Bigelow generators}\label{sec:redgrad}

We define $\mathbb{Z}$-valued functions $\red{Q}$, $\red{T}$, and $\red{ P}$ on $\red{\G}$; they only vary from their unreduced counterparts $Q$, $T$, and $ P$ found in \cite{et:R} in that the reduced versions ignore the omitted pair $\alpha_n$, $\beta_n$.  For a concrete computation (albeit in the unreduced case), see Section 3.2 of \cite{et:R}.

We'll first establish some notation that will be used throughout this section.  Let $x \in \alpha_i \cap \beta_j$ for $1 \leq i,j \leq n-1$.  Let the loop $\gamma_x:I \rightarrow D$ be the path from $d_j$ to $d_i$ formed by concatenating the following paths in the following order (here ``$-\gamma$'' denotes the reversal of a path $\gamma$):
\begin{enumerate}[(i)]
\item The path $-bh_j$  traveling from $d_j$ to $bh_j(0) \in \beta_j$
\item The section of the path $\pm \beta_j$ traveling from $bh_j(0)$ to $x$
\item The section of the path $\pm\alpha_i$ traveling from $x$ to $h_i(0) = m_i$
\item The path $h_i$ traveling from $m_i$ to $d_i$
\end{enumerate}

Furthermore, for a point $e \in \alpha_i \cap bE_j$, let $\tld{\gamma}_e$ denote the path from $d_j$ to $d_i$ formed analogously, except with the segment of $\pm \beta_j$ replaced with a corresponding segment of the figure-eight $\pm bE_j$.

Given a collection of paths $\gamma_1, \ldots, \gamma_k: I \rightarrow D$ with $\gamma_{i}(t) \neq \gamma_j(t)$  for each $t \in I$ if $i \neq j$, the path $(\gamma_1, \ldots, \gamma_k):I \rightarrow D \times \ldots \times D$ descends to a path in $\text{Conf}^{k}(D)$ by composing with the quotient map $D \times \ldots \times D$.  Futhermore, if the sets $\left\{ \gamma_1(0), \ldots, \gamma_{k}(0) \right\}$ and $\left\{ \gamma_1(1), \ldots, \gamma_{k}(1) \right\}$ coincide, the image is a loop in $\text{Conf}^k(D)$.  Also recall that one can identify the braid group $B_k$ on $k$ strands with the fundamental group of the $k$-fold configuration space $\text{Conf}^k(D)$.  We'll denote by $\pi^{k}: B_k = \pi_1\left(\text{Conf}^k(D)\right) \rightarrow \mathbb{Z}$ the usual abelianisation map.

\subsubsection{The reduced version of $T$}\label{sec:Tdef}
Given a Bigelow generator $\be = e_{1} e_{2} \ldots e_{n-1}  \in \red{\G}$, we have that for each $k$, $p(e_k) = x_k$ for some $x_{k} \in \red{\Ztil}$.  Now for each $k$, let $\gamma_k:=\gamma_{x_k}$ be the path described above.  Then let $\gamma: I \rightarrow \text{Conf}^{n-1}(D)$ be the loop which is the image of $(\gamma_1, \ldots, \gamma_{n-1}):I \rightarrow D \times \ldots \times D$.  We then define the grading $\red{T}$ to be $\red{T}(\be) := \pi^{n-1}(\gamma) \in \mathbb{Z}$.

There's a convenient way to view $\red{T}$ as a relative grading.  Choose two reduced Bigelow generators $\be^1,\be^2 \in \red{\G}$ with images $\bx^1, \bx^2 \in \red{\Gtil}$, and construct a loop in $\text{Conf}^{n-1}(D)$ by traveling from $\bx^1$ to $\bx^2$ along the $\beta$ curves and traveling from $\bx^2$ back to $\bx^1$ along the $\alpha$ curves.  Then the difference $\red{T}(\be^2) - \red{T}(\be^1)$ is the linking number of this loop with the fat diagonal $\nabla \subset \text{Sym}^{n-1}(D)$.  In practice, this is equal to the number of half-twists among the $\beta$-arcs connecting $\bx^1$ to $\bx^2$.  The function $\red{T}:\red{\G} \rightarrow \mathbb{Z}$ is determined by this relative grading information along with the value of $\red{T}(\be)$ for any one generator $\be \in \red{\G}$.

\subsubsection{The reduced version of $Q$}

Consider some $\be =  e_{1} e_{2} \ldots e_{n-1}  \in \red{\G}$.  For each $j$, let $\tld{\gamma}_{k}:=\tld{\gamma}_{e_k}$ be the path defined as above.  Then the path $\left(\mu_1, \ldots, \mu_{2n}, \tld{\gamma}_1, \ldots, \tld{\gamma}_{n-1}\right) : I \rightarrow D \times \ldots \times D$ induces a loop in $\text{Conf}^{3n-1}(D)$ based at $\left\{ \mu_1, \ldots, \mu_{2n}, d_1, \ldots, d_{n-1} \right\}$.  Note that the integer $\pi^{3n-1}\left(\tld{\gamma}\right)$ has the same parity as the integer $\pi^{n-1}(\gamma)$ appearing above in the definition of $\red{T}(\be)$.  We then define the $\red{Q}$-grading by 
$$\red{Q}(\be) := \frac{1}{2} \left(\pi^{3n-1}\left(\tld{\gamma}\right)  - \pi^{n-1}(\gamma) \right) \mathbb{Z}$$

In practice, the function $\red{Q}: \red{\G} \rightarrow \mathbb{Z}$ can be computed additively from a function $Q^{*}:\Zcal \rightarrow \mathbb{Z}$.  Consider some $x \in \alpha_{i} \cap bE_{j}$ ($1 \leq i,j \leq n-1$).  Let $\delta_x$ be the loop based at $d_i$ obtained by concatenating the path $\tld{\gamma}_x$ (as above) with the segment of $\partial D$ traveling from $d_j$ to $d_i$.  Now define $Q^{*}(x)$ be the winding number of the loop $\delta_x$ about the set of punctures $\left\{ \mu_1, \ldots, \mu_{2n} \right\}$.  Then for each $\be =  e_{1} e_{2} \ldots e_{n}  \in \red{\G}$, it is clear that
\begin{equation*}
\red{Q}(\be) := \sum_{i = 1}^{n-1} Q^{*}(e_{i}).
\end{equation*}

\subsubsection{The reduced version of $P$}\label{sec:Pdef}
 
The function $\red{P}$ will be computed additively from $P^{*} : \Zcal \rightarrow \mathbb{Z}$, which measures twice the relative winding number of the tangent vectors to the figure eights $E'_{k}$ at the points in $\Zcal$.
 
For $x \in \Zcal$, where $x \in \alpha_{i} \cap bE_{j}$, we define $ P^{*}(x)$ in the following way. We view the arc $bh_{j}$ as being oriented downward at the point where it intersects $\partial D$.  Let $bE_{j}$ have the orientation induced by the orientation on $E_{j}$ in the standard fork diagram.  Then we let $ P^{*}(x)$ be twice the winding number of the tangent vector relative to the downward-pointing tangent vector at the point $h'_{j} \cap \partial D$.  In other words, if the tangent vector makes $k$ counter-clockwise half-revolutions and $m$ clockwise half-revolutions as we travel first along $bh_{j}$ from $bh_{j} (0)$ to $bh_{j} (1)$ then along $bE_{j}$ to $x$, then we set $ P^{*}(x) = m-k$.  This number is an integer because we assume that at any point $x \in \Zcal$, the figure-eight intersects the $\alpha$ arc at a right angle.

Then for $\be = e_{1} e_{2} \ldots e_{n-1} \in \G$, we define
\begin{equation*}
\red{P}(\be) = \sum_{i = 1}^{n-1}   P^{*}(e_{i}).
\end{equation*}

\subsection{An admissible pointed Heegaard diagram for $\widehat{HF}(\DBC{K})$}\label{sec:redHD}

Now let $P_{\mu} \in \mathbb{C}[t]$ be the monic polynomial whose set of roots is $\tau$, the set of punctures.  We define an affine variety $\widehat{S}$ by
\begin{equation*}
\widehat{S} = \{ (u, z) \in \mathbb{C}^{2} : u^{2} + P_{\mu}(z) = 0 \} \subset \mathbb{C}^{2}.
\end{equation*}
Also , for $k = 1, \ldots , n-1$, define the subspaces $\ah_{k}$ and $\bh_{k}$ of $\widehat{S}$ by
\begin{align*}
\ah_{k} &= \left\{ (u,z) \in \mathbb{C} : z = \alpha_{k}(t),\text{ for some } t\in [0,1]; u = \pm \sqrt{-P_{\mu}(z)} \right\} \text{ and}\\
\bh_{k} &= \left\{ (u,z) \in \mathbb{C} : z = \beta_{k}(t),\text{ for some } t\in [0,1]; u = \pm \sqrt{-P_{\mu}(z)} \right\}.
\end{align*}

Now denote by $\Tahr$ and $\Tbhr$ the totally real tori in $\text{Sym}^{n-1}(\widehat{S})$ defined by
\begin{equation*}
\Tahr = \ah_{1} \times \ldots \times \ah_{n-1} \text{ and } \quad \Tbhr = \bh_{1} \times \ldots \times \bh_{n-1}.
\end{equation*}
The map $\widehat{S} \rightarrow \mathbb{C}$ given by $(u,z) \mapsto z$ is in fact a 2-fold branched covering map, where the branch set downstairs is the set of punctures $\tau$.  Thus (by viewing $S^2$ as $\mathbb{C} \cup \{ +\infty \}$), one can see that in fact $\widehat{S} = \Sigma_{n-1} \setminus \{ \pm \infty \}$, where $\Sigma_{n-1}$ is a closed surface of genus $n-1$.  In Section 4.2 of \cite{et:R} (following \cite{cm:R}), we described the construction of a pointed Heegaard diagram for $\DBCs{K}$ in the unreduced setting - the extra $\SxS$ summand appeared because one was required to stabilize the Heegaard surface.  Since we've omitted a pair of curves here, we don't need to stabilize.

\begin{prop}\label{prop:DBCred}
The collection of data
\begin{equation*}
\red{\h} = (\Sigma_{n-1}; \ah_{1}, \ldots , \ah_{n-1} ; \bh_{1}, \ldots , \bh_{n-1} ; +\infty  )
\end{equation*}
is an admissible pointed Heegaard diagram for $\DBC{K}$.
\end{prop}
\begin{proof}

%

We first check that $\red{\h}$ actually represents $\DBC{K}$.  It suffices to show that
\begin{equation*}
\h' = \left(\Sigma_{n} = \Sigma_{n-1} \# \Sigma_{1}; \ah_{1}, \ldots , \ah_{n-1}, \alpha_{0} ; \bh_{1}, \ldots , \bh_{n-1}, \beta_{0} ; +\infty\right)
\end{equation*}
gives a pointed Heegaard diagram for $\DBCs{K}$, where $\left(\Sigma_{1}; \alpha_{0}; \beta_{0}; +\infty\right)$ is the standard pointed Heegaard diagram for $S^{2} \times S^{1}$ shown in Figure \ref{fig:s2s1hd}.  We accomplish this by showing that $\h$ can be obtained from $\h'$ by a sequence of handleslides, where 
$$\h = \left( \Sigma_n; \ah_1, \ldots, \ah_n; \bh_1, \ldots, \bh_{n}; +\infty \right)$$
is the pointed Heegaard diagram for $\DBCs{K}$ obtained as the stabilized double branched cover of the unreduced fork diagram as in \cite{et:R}.

\begin{figure}[h!]
\centering
\begin{minipage}[c]{.35\linewidth}
\labellist 
\small
\pinlabel* {$\alpha_{0}$} at 275 30
\pinlabel* {$\beta_{0}$} at 235 190
\pinlabel* {\reflectbox{$a$}} at 272 117
\pinlabel* {$a$} at 120 117
\pinlabel* {$+\infty$} at 70 175
\endlabellist 
\includegraphics[height = 30mm]{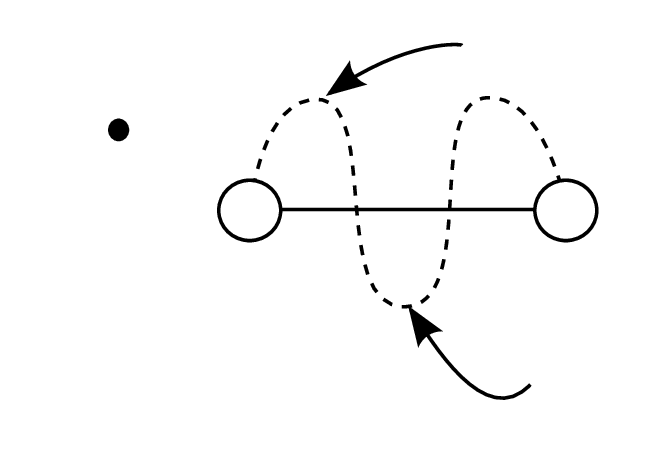}
\end{minipage}
\begin{minipage}[c]{.55\linewidth}
\caption{A pointed Heegaard diagram for $S^{2} \times S^{1}$ \label{fig:s2s1hd}}
\end{minipage}
\end{figure}
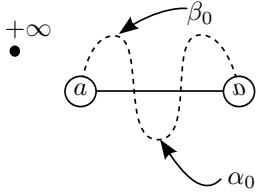

The set of attaching circles $\ah_{1}, \ldots, \ah_{n-1},\alpha_0$ can be obtained via a sequence of handleslides and isotopies (avoiding the basepoint $+\infty$) from the set $\ah_{1}, \ldots, \ah_{n}$.  Indeed, one should slide $\ah_{n}$ over $\ah_{n-1}$, then slide the resulting curve over $\ah_{n-2}$, and so on  (for a total of $n-1$ such handleslides), to arrive at $\ah_{1}, \ldots, \ah_{n-1},\alpha_0$.  This process is illustrated in Figure \ref{fig:alphas}.  As a result, $\beta_{i}$ is obtained from $\alpha_i$ (for each $i$) via the diffeomorphism induced by the braid, $\bh_{1}, \ldots, \bh_{n}$ can be obtained via an analogous sequence of handleslides and isotopies from $\bh_{1}, \ldots, \bh_{n-1}, \beta_{0}$.

\begin{figure}[h!]
\centering

\subfigure[]{
\labellist 
\small
\pinlabel* {$a$} at 43 140
\pinlabel* {$b$} at 145 140
\pinlabel* {$c$} at 242 140
\pinlabel* {$\ah_1$} at 30 100
\pinlabel* {$\ah_2$} at 95 155
\pinlabel* {$\ah_3$} at 200 155
\pinlabel* {\rotatebox{180}{\reflectbox{$a$}}} at 45 65
\pinlabel* {\rotatebox{180}{\reflectbox{$b$}}} at 145 65
\pinlabel* {\rotatebox{180}{\reflectbox{$c$}}} at 240 65
\endlabellist
\includegraphics[height = 32mm]{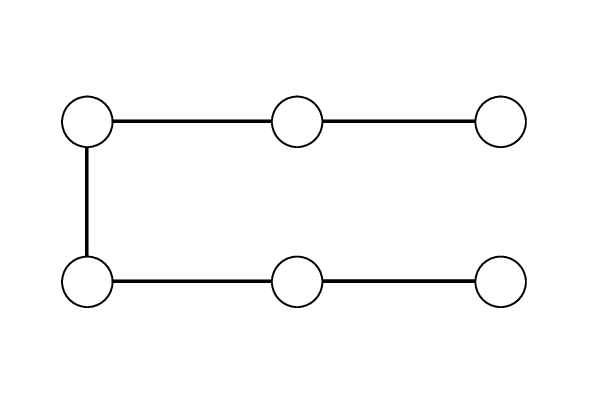}}\quad
\subfigure[]{
\labellist 
\small
\pinlabel* {$a$} at 43 140
\pinlabel* {$b$} at 145 140
\pinlabel* {$c$} at 242 140
\pinlabel* {\rotatebox{180}{\reflectbox{$a$}}} at 45 65
\pinlabel* {\rotatebox{180}{\reflectbox{$b$}}} at 145 65
\pinlabel* {\rotatebox{180}{\reflectbox{$c$}}} at 240 65
\endlabellist
\includegraphics[height = 32mm]{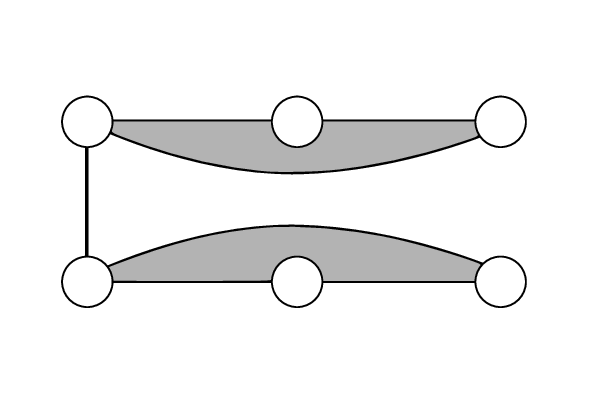}}\quad
\subfigure[]{
\labellist 
\small
\pinlabel* {$a$} at 45 142
\pinlabel* {$b$} at 145 142
\pinlabel* {$c$} at 243 142
\pinlabel* {\rotatebox{180}{\reflectbox{$a$}}} at 45 67
\pinlabel* {\rotatebox{180}{\reflectbox{$b$}}} at 145 63
\pinlabel* {\rotatebox{180}{\reflectbox{$c$}}} at 243 67
\pinlabel* {$\alpha_0$} at 190 155
\endlabellist 
\includegraphics[height = 32mm]{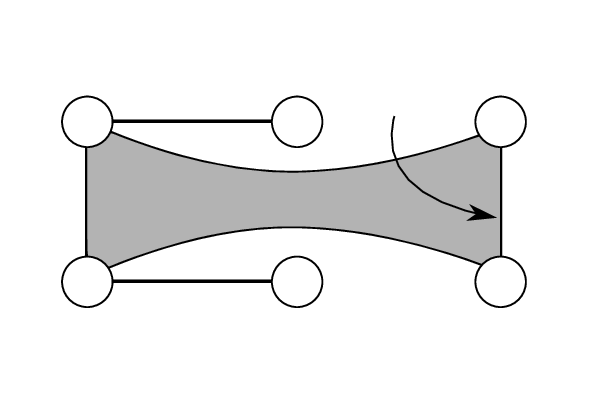}}
\caption[Handleslides appearing in the proof of Proposition \ref{prop:DBCred}]{Handleslides connecting the set $\{ \ah_1, \ah_2, \ah_3 \}$ to $\{ \ah_1, \ah_2, \alpha_0 \}$ in the proof of Proposition \ref{prop:DBCred}.  Pairs of pants are shaded.
\label{fig:alphas}}
\end{figure}
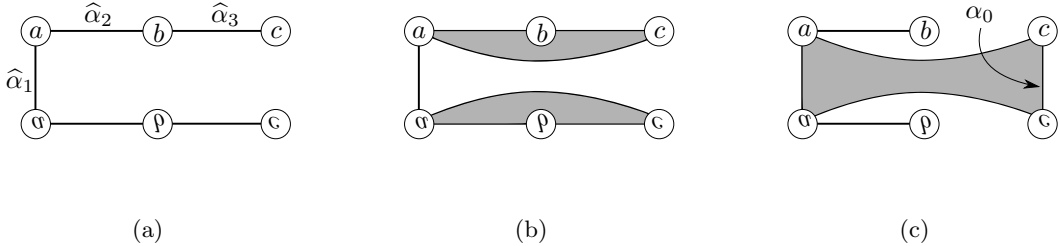

Admissibility of the diagram $\h'$ follows by mimicking the argument in the proof of Proposition 7.4 from \cite{cm:R}; when examining a periodic domain $\mathcal{P}$, one exploits the involution on the regions of $\displaystyle\Sigma_{(n-1)} \setminus \left( \bigcup_{i=1}^{(n-1)} \ah_{i}  \right) \setminus \left( \bigcup_{i=1}^{(n-1)}\bh_i\right)$ (induced by the involution on the Heegaard surface) to find positive and negative coefficients in the expansion of $\mathcal{P}$.
\end{proof}

\begin{rmk}
With respect to the branched covering map $\widehat{S} \rightarrow \mathbb{C}$, each puncture $\mu_{k} \in \mathbb{C}$ has a single point as its preimage.  However, the preimage of a point $x \in \beta_{j} \cap \text{int}(\alpha_{i})$ consists of a pair of points upstairs.  This suggests an identification between the intersection $\Tahr \cap \Tbhr$ and the set $\red{\G}$.  However, this identification isn't canonical, since for some $x \in \red{\Ztil} \setminus \red{\tau}$ it is only required that the pair $\{ e_{x}, e'_{x} \}$ is identified with the two preimages of $x$ upstairs.  Still, it is straightforward to check that $\left(\red{P}-\red{Q}\right)$ and $\red{T}$ descend to functions $\red{\Gtil} \rightarrow \mathbb{Z}$ and thus give well-defined functions $\Tahr \cap \Tbhr \rightarrow \mathbb{Z}$.
\end{rmk}

\subsection{Gradings on reduced tori}\label{sec:redmas}

Recall that a system of coordinates on $\text{Sym}^{(n-1)}(\widehat{S})$ is given by $\left\{ (u_1,z_1), (u_2,z_2), \ldots, (u_{(n-1)},z_{(n-1)}) \right\}$ where $u_i^2 + P_{\mu}(z_i) = 0$. Define the \emph{anti-diagonal} by the submanifold $\AD \subset \text{Sym}^{(n-1)}(\widehat{S})$ given by
$$ \AD := \left\{  \left\{ \left(u_i,z_i\right)_i \right\} \quad \text{where $z_j = z_k$ and $u_j = - u_k$ for some $j \neq k$}   \right\},$$
and let $\red{W} = \text{Sym}^{(n-1)}(\widehat{S}) \setminus \AD$.   One can lift the relative Maslov grading on $\Tahr \cap \Tbhr \subset \red{W}$ to an absolute $\mathbb{Z}$-valued grading.  This lift is achieved by first choosing a particular $\mathbb{C}$-valued holomorphic volume form on $\red{W}$ and using this volume form to improve $\Tahr$ and $\Tbhr$ to \emph{graded submanifolds}.  From this pair of graded submanifolds, one obtains an absolute Maslov grading on the set $\Tahr \cap \Tbhr$.  This construction was introduced by Seidel in \cite{s:GL} and reviewed in Section 4.2 of \cite{et:R}.  By comparing this situation to that in the unreduced case, we obtain the following.

\begin{prop}\label{prop:cmRred}
There exists a volume form $\red{\Theta}$ on $\red{W}$ inducing Seidel gradings on $\Tahr$ and $\Tbhr$ in the sense of Section 4.1 of \cite{et:R}, and the resulting absolute Maslov grading on $\Tahr \cap \Tbhr \subset \red{W}$ is exactly $\red{ P} - \red{Q} + \red{T}$.
\end{prop}

\begin{proof}
Consider the $\mathbb{C}$-valued form $\red{\Theta}$ on $\text{Sym}^{(n-1)}(\widehat{S})$ given by 
\begin{equation}\label{eqn:vol}
\red{\Theta} := \prod_{1 \leq i < j \leq (n-1)} (z_i - z_j) \prod_{k = 1}^{(n-1)} \frac{dz_k}{u_k}.
\end{equation}

In the proof of Proposition 7.4 in \cite{cm:R}, Manolescu introduced a form $\Theta$ on $\text{Sym}^{n}(\widehat{S})$, and in fact
\begin{equation}\label{eqn:theta}
\Theta =\red{\Theta} \wedge  \left( \prod_{i = 1}^{(n-1)} (z_i - z_n) \right) \frac{dz_n}{u_n}.
\end{equation}

Via an argument analogous to that in \cite{cm:R}, one can see that $\red{\Theta}$ in fact induces a volume form on $\red{W}$.  Let $\tld{R}$ denote the absolute Maslov grading on $\Tah \cap \Tbh \subset W$ induced by $\Theta$ as in Proposition 4.3.1 from \cite{et:R}, and let $\red{\tld{R}}$ denote the one induced on $\Tahr \cap \Tbhr \subset \red{W}$ by $\red{\Theta}$.  Then with Equation \ref{eqn:theta} in mind, one can see that if $\red{\bx} \in \Tahr \cap \Tbhr$ and $\bx := \red{\bx}x_n \in  \Tah \cap \Tbh$,
$$ \tld{R}(\bx) - \red{\tld{R}}(\red{\bx}) = \left( P-Q+T\right)(\bx) - \left( \red{P}-\red{Q}+\red{T}\right)(\red{\bx}).$$
Since we have from \cite{et:R} that $\tld{R} = P - Q +T$, the result follows.
\end{proof}

We then define $\red{R}$ via a shift which depends on the signed count of braid generators in $b$ and the writhe of the diagram $D$ which is its plat closure.
\begin{equation*}
\red{R} = \red{\tld{R}} + s_{\red{R}}(b,D),\text{ where } s_{\red{R}}(b,D) = \frac{\epsilon(b) - w(D) - 2(n-1)}{4} \in \mathbb{Q}.
\end{equation*}

Notice that the second condition in Definition \ref{def:red} guarantees that $\pm\infty$ lie in the same component of $\displaystyle\Sigma_{(n-1)} \setminus \left( \bigcup_{i=1}^{(n-1)} \ah_{i} \right) \setminus \left( \bigcup_{i=1}^{(n-1)} \bh_{i} \right).$
Therefore, if $\bx, \by \in \Tahr \cap \Tbhr$, and $\phi \in \pi_{2}(\bx, \by)$, then $n_{+\infty}(\phi)=0$ if and only if $n_{-\infty}(\phi)=0$

When viewed as a grading, the function $\red{R}$ is also not compatible with Maslov index calculations in the  entire symmetric product.  By mimicking the argument used in the unreduced case, one can see that the volume form $\red{\Theta}$ has an order-one zero along the anti-diagonal $\AD \subset \text{Sym}^{(n-1)}(\widehat{S})$.  Therefore, if $\bx, \by \in \Tahr \cap \Tbhr$, and $\phi \in \pi_{2}(\bx, \by)$ with $n_{+\infty}(\phi)=0$, then
\begin{equation*}\label{eqn:diff}
\red{R}(\bx) - \red{R}(\by) = \mu(\phi)+ 2[\phi]\cdot[\AD].
\end{equation*}
Recall that for every knot $K$, the manifold $\DBC{K}$ is a rational homology sphere.  All elements of $\text{Spin}^{c}(\DBC{K})$ are then torsion, and thus the absolute grading $\tld{gr}$ can be defined for all generators in $\red{\h}$.  One then can obtain a filtration $\red{\rho}$ on the complex $\widehat{CF}(\red{\h})$ via
\begin{equation*}
\red{\rho}(\bx) = \red{R}(\bx) - \tld{gr}(\bx).
\end{equation*}

\subsection{The anti-diagonal and covering Heegaard diagrams}

When analyzing Heegaard moves below, we'll encounter intermediate Heegaard diagrams which don't inherently cover fork diagrams.  We'll need to define a new type of decorated Heegaard diagram.  Fix $g$, and let $\left\{ \mu_1, \mu_2, \ldots \mu_{2n} \right\} \subset D \subset \mathbb{C}$ be $2n$ distinct points evenly spaced along $(-1,1) \subset \mathbb{R} \subset \mathbb{C}$. 

\begin{df}\label{def:chd}
A \emph{covering Heegaard diagram} (CHD) of genus $g$ is a tuple of data $\left( \Sigma; \ba; \bb; z_+; z_-; \pi\right)$, where $\mathcal{H} = \left( \Sigma; \ba; \bb; z_+ \right)$ is a pointed Heegaard diagram of genus $g$ and where $\pi: \Sigma \setminus \left\{ z_+, z_- \right\} \rightarrow \mathbb{C}$ is a two-fold branched covering map with branch set $\left\{ \mu_1, \ldots, \mu_{2n} \right\} \subset \mathbb{C}$.  Two such diagrams $\left( \Sigma; \ba; \bb; z_+; z_-; \pi\right)$ and $\left( \Sigma'; \ba'; \bb'; z'_+; z'_-; \pi'\right)$are \emph{covering diffeomorphic} if there is a diffeomorphism $f:\left(  \Sigma; \ba; \bb; z_+; z_-\right) \rightarrow \left(  \Sigma'; \ba'; \bb'; z'_+; z'_-\right)$ with $\pi = f \circ \pi'$.
\end{df}

We saw in Section \ref{sec:redHD}, CHDs arise naturally as the branched covers of fork diagrams, but not every CHD arises in this way.  Each CHD gives rise to an associated \emph{anti-diagonal} $\nabla \subset \text{Sym}^{k}(\Sigma)$ given by
$$ \nabla := \left\{ \left\{ x_1, x_2, \ldots, x_k \right\} | \pi(x_i) = \pi(x_j) \text{ for some } i \neq j \right\}$$
A diffeomorphism between two CHDs induces a diffeomorphism between the associated anti-diagonals.  Fix a CHD $\left( \Sigma; \ba; \bb; z_+; z_-; \pi\right)$.  It isn't hard to verify that given $x,y \in \Ta \cap \Tb$ and a 2-gon class $\phi \in \pi_2(x,y)$,  the intersection number $[\phi] \cdot [\nabla] \in \mathbb{Z}$ depends only on the covering diffeomorphism type of the CHD.

There is an obvious generalization of Definition \ref{def:chd} to Heegaard multi-diagrams which we won't explicitly state.  Multi-CHDs induce anti-diagonals, the above observation about intersection numbers still holds for $n$-gon classes in general.

A common situation below will consist of a pair of CHDs $\h$ and $\h'$ which cover fork diagrams and are connected by a sequence of pointed isotopies and pointed handleslides passing through intermediate CHDs which don't cover fork diagrams in any obvious way.  When studying such sequences of diffeomorphism classes of CHDs, we'll always pick representatives which have precisely the same $z_{\pm}$ and precisely the same standard projection map $\pi: \Sigma\setminus \left\{ z_+, z_-\right\} \rightarrow \mathbb{C}$ - namely, the one depicted in (FIGURE).  This will allow us to effectively keep track of the anti-diagonal when studying 3-gons associated to Heegaard moves - $\nabla$ won't change in our pictures of surfaces, but the $\alpha$ and $\beta$ circles will.

\begin{figure}[h!]
\centering
\subfigure[A fork diagram cut along dashed arcs (the unit disk has been elongated for readability)]{
\includegraphics[height = 30mm]{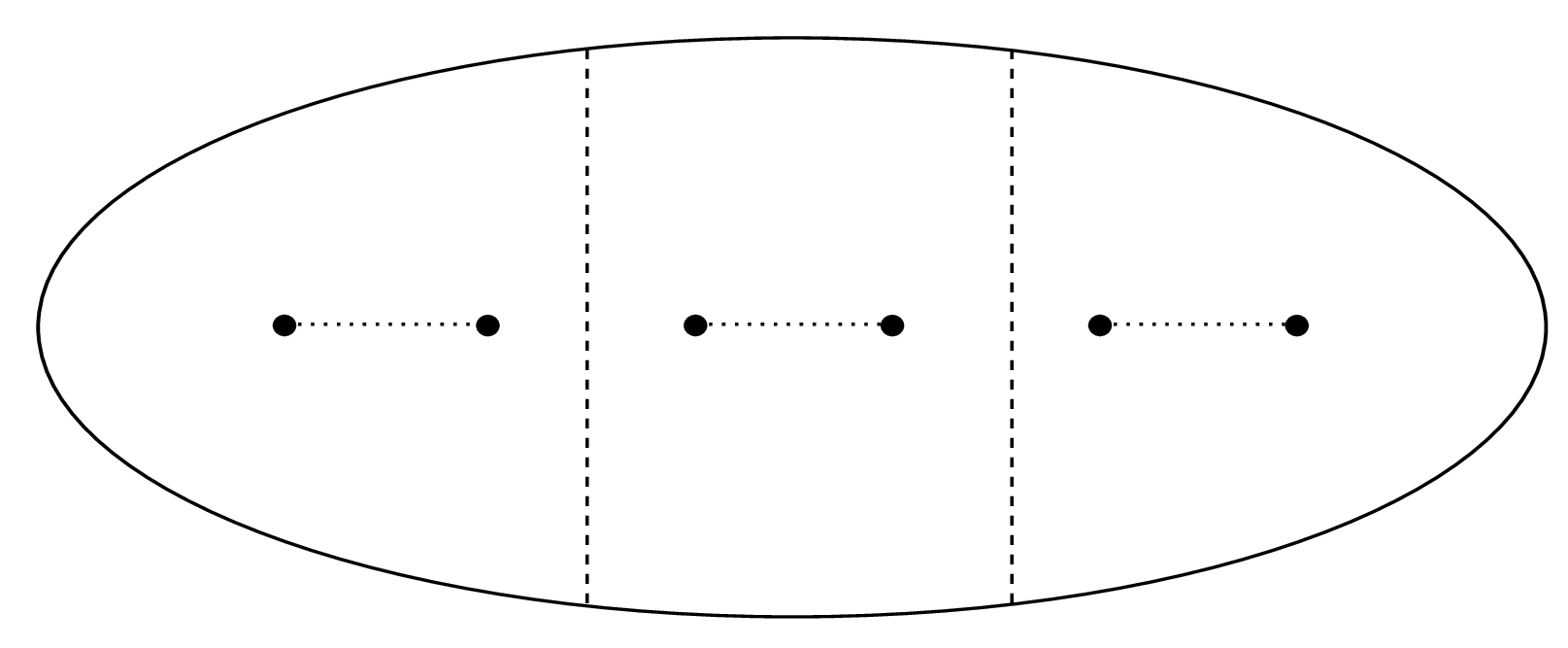}\label{fig:forkcut}}\\
\subfigure[The annuli covering the three pieces from Figure \ref{fig:forkcut}]{
\includegraphics[height = 30mm]{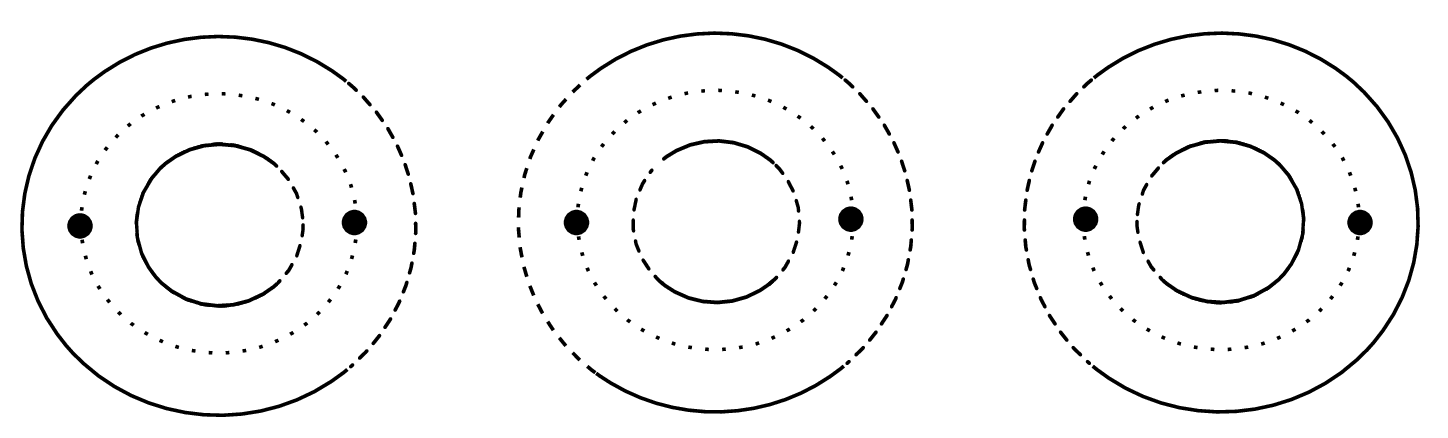}\label{fig:hdcut}}\\
\subfigure[The glued-up genus-two surface]{
\labellist 
\small
\pinlabel* $a$ at 310 220
\pinlabel* \rotatebox{180}{\reflectbox{$a$}} at 310 95
\pinlabel* $b$ at 490 220
\pinlabel* \rotatebox{180}{\reflectbox{$b$}} at 490 95
\endlabellist 
\includegraphics[height = 30mm]{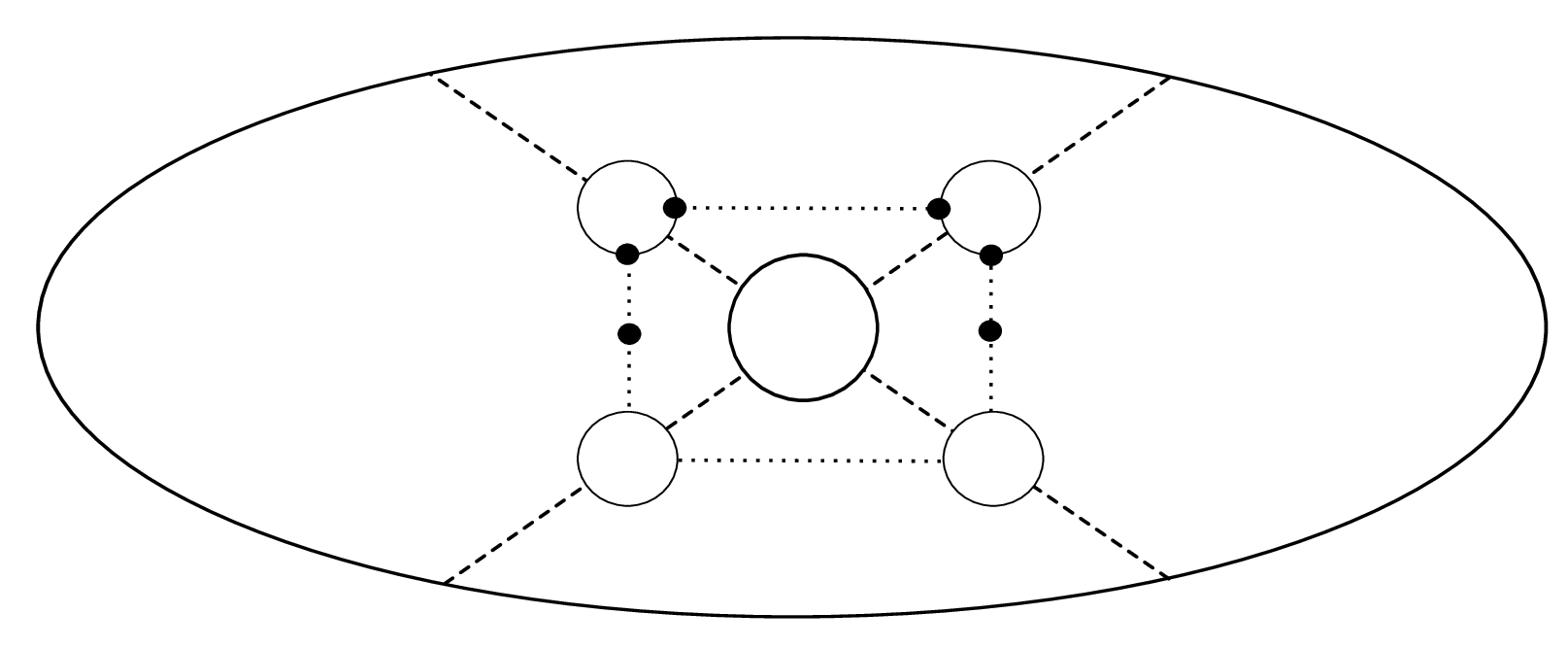}\label{fig:hdglue}}
\caption[A standard branched covering map]{A description of the standard branched covering map.}
\end{figure}

Notice that if we insist on fixing $\pi$, many isotopies of attaching circles aren't \emph{a priori} realizable as covering diffeomorphisms.  However, we can still analyze these isotopies by studying 3-gon classes, as we do with handleslides.

\subsubsection{Thin 3-gon classes}

We dispense with a type of 3-gon class which will arise often.  Assume that $\ba$, $\bb$, and $\bb'$ be n-tuples of attaching curves on $\Sigma$ such that $\bb'$ differs from $\bb$ by a pointed handeslide or pointed isotopy.  For $\bx \in \Ta \cap \Tb$ and $\by \in \Ta \cap \tor{\bb'}$, let $\psi \in \pi_{2}(\bx,\thet{\bb,\bb'},\by)$ be a 3-gon class avoiding the basepoint with $\mu(\psi) = 0$, where the domain $\mathcal{D}(\psi)$ is a sum of $n$ disjoint 3-sided regions $\mathcal{D}_1, \ldots, \mathcal{D}_n$.  A point in $\text{Im}(\psi) \subset \text{Sym}^n(\Sigma)$ is of the form $\bx = \{ x_1, \ldots, x_n\}$, where each $x_i \in \mathcal{D}_i$.  If we further assume that at least $n-1$ of the regions $\mathcal{D}_i$ are thin 3-sided regions of the type in Figure \ref{fig:smalltri}, then $\pi^{-1}(\pi(\mathcal{D}_{i})$ will be the disjoint union of two thin 3-sided regions and so
\begin{equation*}\label{eqn:adtri}
\mathcal{D}_i \cap \pi^{-1} \left( \pi \left( \mathcal{D}_j \right) \right) = \emptyset \quad \text{for} \quad i \neq j.
\end{equation*}
As a result, $\text{Im}(\psi) \cap \nabla = \emptyset$.

\begin{figure}[h!]
\centering
\begin{minipage}[c]{.38\linewidth}
\labellist
\small
\pinlabel* {$\theta_{\bb \bb'}$} at 84 94
\pinlabel* {$x_i$} at 240 50
\pinlabel* {$y_i$} at 240 140
\endlabellist
\includegraphics[height = 20mm]{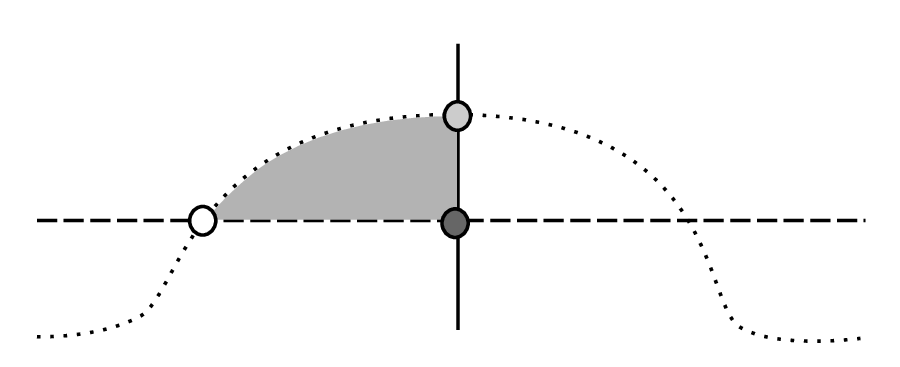}
\end{minipage}
\begin{minipage}[c]{.60\linewidth}
\caption[A small triangle]{A small 3-sided region appearing in a small isotopy}
\label{fig:smalltri}
\end{minipage}
\end{figure}

\subsection{Triangle injections}
We define some auxiliary maps which will be useful for comparing $\red{\rho}$-filtered chain homotopy types CHDs which cover fork diagrams.
\begin{df}\label{df:trimap}
Let $\left(\Sigma; \ba; \bb; z_{\pm}; \pi\right)$ and $\left(\Sigma; \ba'; \bb';z_{\pm}; \pi \right)$ be two admissible CHDs of genus $n$ appearing in some sequence of pointed isotopies and handleslides connecting two CHDs which cover fork diagrams, and let $\AD \subset \text{Sym}^n(\Sigma)$ denote the anti-diagonal.
\begin{enumerate}[(i)]
\item If $\bb' = \bb$ and $\ba'$ differs from $\ba$ by a pointed isotopy or pointed handleslide, then a $\ba$\textit{-triangle injection} is a function $g: \Ta \cap \Tb \hookrightarrow \tor{\ba'} \cap \Tb$ such that the following hold:
\begin{enumerate}[(a)]
\item There is an admissible triple-CHD $\left(\Sigma; \ba^{+}; \ba; \bb; z_{\pm}; \pi\right)$ (where $\ba^+$ differs from $\ba'$ by a small pointed isotopy and for each $k$, $\alpha^{+}_k$ intersects $\alpha_k$ transversely in two points) such that for each $\bx \in \Ta \cap \Tb$, there is a 3-gon class $\psi_{g}^{+} \in \pi_{2}(\thet{\ba^{+}\ba}, \bx, \by^{+})$ with $\mu(\psi_{g}^{+}) = 0$, $\psi_{g}^{+} \cap \AD = \emptyset$, and $n_{z}(\psi_{g}^{+}) = 0$, where $\by^{+} \in \Ta \cap \tor{\ba^{+}}$ is the nearest neighbor to $g(\bx)$.
\item There is an admissible triple-CHD $\left(\Sigma; \ba; \ba^{-}; \bb; z_{\pm}; \pi\right)$ (where for each $k$, $\alpha_{k}^{-}$ differs from $\alpha'_{k}$ by a small pointed isotopy and for each $k$, $\alpha^{-}_k$ intersects $\alpha_k$ transversely in two points)  such that for each $\bx \in \Ta' \cap \Tb$, there is a 3-gon class $\psi_{g}^{-} \in \pi_{2}(\thet{\ba\ba^{-}}, \by^{-}, \bx)$ with $\mu(\psi_{g}^{-}) = 0$, $\psi_{g}^{-} \cap \AD = \emptyset$, and $n_{z}(\psi_{g}^{-}) = 0$, where $\by^{-} \in \Ta \cap \tor{\ba^{-}}$ is the nearest neighbor to $g(\bx)$.
\end{enumerate}
\item If $\ba' = \ba$ and $\bb'$ differs from $\bb$ by a pointed isotopy or pointed handleslide, then a $\bb$\textit{-triangle injection} is a function $g: \Ta \cap \Tb \hookrightarrow \Ta \cap \tor{\bb'}$ such that the following hold:
\begin{enumerate}[(a)]
\item There is an admissible triple-CHD $\left(\Sigma; \ba; \bb; \bb^{+}; z_{\pm}; \pi\right)$ (where $\bb^+$ differs from $\bb'$ by a small pointed isotopy and for each $k$, $\beta^{+}_k$ intersects $\beta_k$ transversely in two points) such that for each $\bx \in \Ta \cap \Tb$, there is a 3-gon class $\psi_{g}^{+} \in \pi_{2}(\bx, \thet{\bb\bb^{+}}, \by^{+})$ with $\mu(\psi_{g}^{+}) = 0$, $\psi_{g}^{+} \cap \AD = \emptyset$, and $n_{z}(\psi_{g}^{+}) = 0$, where $\by^{+} \in \Tb \cap \tor{\bb^{+}}$ is the nearest neighbor to $g(\bx)$.
\item There is an admissible triple-CHD $\left(\Sigma; \ba; \bb^{-}; \bb; z_{\pm}; \pi\right)$ (where $\bb^-$ differs from $\bb'$ by a small pointed isotopy and for each $k$, $\beta^{-}_k$ intersects $\beta_k$ transversely in two points) such that for each $\bx \in \Ta \cap \Tb'$, there is a 3-gon class $\psi_{g}^{-} \in \pi_{2}(\by^{-}, \thet{\bb^{-}\bb}, \bx)$ with $\mu(\psi_{g}^{-}) = 0$, $\psi_{g}^{-} \cap \AD = \emptyset$, and $n_{z}(\psi_{g}^{-}) = 0$, where $\by^{-} \in \Tb \cap \tor{\bb^{-}}$ is the nearest neighbor to $g(\bx)$.
\end{enumerate}
\end{enumerate}
\end{df}

Although the chain maps induced by moves on CHDs count holomorphic representatives of 3-gon classes, one can understand how they interact with the filtration $\red{\rho}$ by studying the appropriate triangle injections - this will allow us to avoid a difficult direct analysis of moduli spaces of 3-gons.  This was accomplished via Lemma 5.0.6 in \cite{et:R}; the following is a modification of that result.

%
%
%
\begin{lem}\label{lem:tri}
Let $K \subset S^3$ be a knot and let $\h = \left(\Sigma; \ba; \bb; z_{\pm}; \pi\right)$ and $\h' = \left(\Sigma; \ba'; \bb'; z_{\pm}; \pi\right)$ be two CHDs for $\DBC{K}$ which are obtained from braids $b$ and $b'$ (possibly after Heegaard stabilization), and suppose that the following hold:

\begin{enumerate}[(i)]
\item For integers $m,n$ with $m \geq n \geq 0$, there is a sequence of pointed isotopies and/or handleslides
$$ \h =\h_0 \mapsto \h_1\mapsto \ldots \mapsto \h_{n}  \mapsto \h_{n+1} \mapsto \ldots \mapsto \h_m = \h'$$
such that each intermediate CHD $\h_{i} = \left( \Sigma; \ba^i; \bb^i; z_{\pm}; \pi \right)$ is admissible.  Recall that we use a fixed projection map $\pi: \Sigma \rightarrow S^2$ throughout this sequence.
\item There are triangle injections $g_i: \tor{\ba^{i-1}} \cap \tor{\bb^{i-1}} \rightarrow \tor{\ba^i} \cap \tor{\bb^i}$ (for $1 \leq i \leq n$) and $g_i:\tor{\ba^i} \cap \tor{\bb^i} \rightarrow \tor{\ba^{i-1}} \cap \tor{\bb^{i-1}}$ (for $n < i \leq m$) satisfying
$$ \text{Im} \left( g_n \circ g_{n-1} \circ \ldots \circ g_1  \right) \subseteq \text{Im} \left( g_{n+1} \circ g_{n+2} \circ \ldots \circ g_{m} \right)$$
\item The composition $g: \Ta \cap \Tb \rightarrow \Tap \cap \Tbp$ given by 
$$ g := \left(  g_{n+1} \circ g_{n+2} \circ \ldots \circ g_{m}\right)^{-1} \circ \left( g_n \circ g_{n-1} \circ \ldots \circ g_1 \right)$$
satisfies $\red{R}(g(\bx)) = \red{R}(\bx)$ for each $\bx \in \tor{\ba} \cap \tor{\bb}$.
\end{enumerate}
Then for each $\mathfrak{s} \in \Sc{K}$, the $F$-filtered complexes $\widehat{CF}(\h, \mathfrak{s})$ and $\widehat{CF}(\h', \mathfrak{s})$ have the same filtered chain homotopy type, where $F$ can be chosen to be either $\red{R}$ or $\red{\rho}$.
\end{lem}

\begin{proof}
The case in which $m = n$ (i.e. $\h_n = \h'$) is very similar to that appearing in Lemma 5.0.6 of \cite{et:R} - the statement there concerned the unreduced filtrations $R$ and $\rho$ and the manifold $\DBCs{K}$.  Due to the similarities between the reduced and unreduced constructions, the same argument works here as well.

An analogous argument works when $m > n$, though the triangle injection $g_i$ ``goes the wrong way'' in the sequence of Heegaard moves when $n < i \leq m$.  To remedy this, one simply reverses the roles of the 3-gons $\psi_{g_{i}}^{\pm}$ associated to $g_{i}$ for $n < i \leq m$ when running the argument from \cite{et:R}.

It is implicit in the proof of Lemma 5.0.6 of \cite{et:R} that under these hypotheses, the $\red{R}$-filtered complexes also have the same filtered chain homotopy type.
\end{proof}

\begin{rmk}\label{rmk:links}
Notice that if the plat closure of a braid $b$ is a link $L \subset S^3$ (rather than a knot), versions of Propositions \ref{prop:DBCred} and \ref{prop:cmRred} still hold and provide a filtration $\red{R}$ on $\widehat{CF}(\DBC{L})$.  However, the absolute grading $\tld{gr}$ (and thus the filtration $\red{\rho}$) is only defined for the summands $\widehat{CF}(\DBC{L},\mathfrak{s})$ with $\mathfrak{s} \in \text{Spin}^c(\DBC{L})$ torsion.

Assuming one restricts to these torsion summands, an analogue of Lemma \ref{lem:tri} holds for links as well (the proof doesn't rely on the number of components); this will be applied to a particular two-component link $L_K$ in the proof of Proposition \ref{prop:RredR} below.  
Section \ref{sec:links} will discuss the situation for links in more generality.
\end{rmk}

\subsection{The reduced-unreduced correspondence}\label{sec:RredR}

Proposition \ref{prop:RredR} provides the relationship between the filtrations $\rho$ and $\red{\rho}$.  However, we'll need Lemma \ref{lem:RredR}, a technical fact which concerns a braid $b \in \B{2n-2}$ inducing a reducible fork diagram which looks as in Figure \ref{fig:redfork} in a neighborhood of the two rightmost punctures.  Such a braid can be obtained for any knot via some appropriate stabilization.

\begin{figure}[h!]
\centering
\begin{minipage}[c]{.40\linewidth}
\includegraphics[height = 20mm]{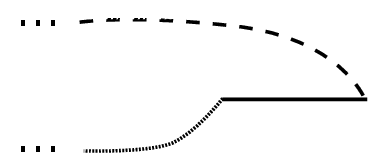}
\end{minipage}
\begin{minipage}[c]{.55\linewidth}
\caption[A fork diagram for comparing reduced and unreduced theories]{The local configuration near $\mu_{2n-3}$ and $\mu_{2n-2}$ in Lemma \ref{lem:RredR}.  The dashed arc is $\beta_{n-1}$.}
\label{fig:redfork}
\end{minipage}
\end{figure}

Given a knot $K$, let $L_{K}$ be the two-component split link whose components are $K$ and an unknot; recall that $\DBC{L_K} \cong \DBCs{K}$.  As mentioned in Remark \ref{rmk:links}, one can endow $\widehat{CF}(\DBCs{K}, \mathfrak{s} \# \mathfrak{s}_{0})$ with the $\red{R}$ and $\red{\rho}$ filtrations for $\mathfrak{s} \in \text{Spin}^{c}(\DBC{K})$ and $\mathfrak{s}_{0} \in \text{Spin}^{c}(\SxS)$ torsion.

Now suppose that a fork diagram has some puncture $\mu_{k} \in \alpha_{j} \cap \beta_{j}$ (for arbitrary $j$ and $k = 2j-1$ or $k=2j$) and an arc avoiding the other $\alpha_i$ and $\beta_i$ and connecting $\mu_k$ to the boundary of the unit disk.  Then one can define versions of  $\red{R}$ and $\red{\rho}$ in the obvious way by omitting $\alpha_{j}$, $\beta_{j}$ (this is slightly more general than the usual reduction procedure described immediately after Definition \ref{def:red}).

\begin{rmk}
From this point forward, all sequences of isotopies and handleslides passing through CHDs will preserve the points $z_{\pm}$ and the branched covering map $\pi$.  Therefore, we'll drop these from the notation (though we'll often keep track of the basepoint $z_+$ for Heegaard Floer purposes.  Furthermore, all Heegaard diagrams should be understood to be CHDs (and we'll often refer to them as ``Heegaard diagrams'').
\end{rmk}

\begin{lem}\label{lem:RredR}
Let $b \in \B{2n-2}$ be a braid inducing a fork diagram of the form shown in Figure \ref{fig:redfork}.  Denote by $\h_{n}$ the Heegaard diagram obtained by omitting the pair $\alpha_{n}$, $\beta_{n}$ from the reducible fork diagram induced by the braid $b \times (1)^{2} \in \B{2n}$, and by $\h_{n-1}$ theHeegaard diagram obtained by omitting the pair $\alpha_{n-1}$, $\beta_{n-1}$ from the same fork diagram.  Then $\widehat{CF}(\h_{n})$ and $\widehat{CF}(\h_{n-1})$ have the same $\red{R}$-filtered chain homotopy type.
\end{lem}

\begin{proof}
We compare the filtered chain homotopy types of the complexes obtained from $\h_{n}$ and $\h_{n-1}$.  This will be accomplished via an intermediate picture: let $\alpha'_{n}$ be the arc obtained from $\alpha_{n}$ via the finger-like isotopy shown in Figure \ref{fig:redalpha}, and let $\beta'_{n} = b\alpha'_{n}$.

\begin{figure}[h!]
\centering
\begin{minipage}[c]{.55 \linewidth}
\includegraphics[height = 20mm]{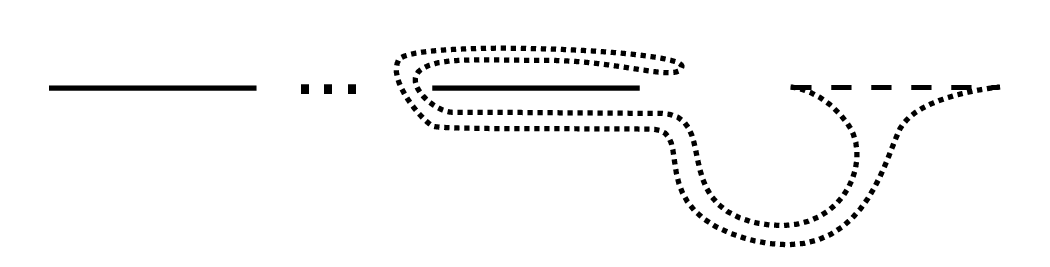}
\end{minipage}
\begin{minipage}[c]{.35\linewidth}
\caption[Finger-like isotopy for comparing arc omission choices]{The $\alpha$-arcs in the fork diagram $\mathcal{F}'$; $\alpha_{n}$ is dashed, $\alpha'_{n}$ is dotted, and $\alpha_{i}$ is solid for $i < n$.\label{fig:redalpha}}
\end{minipage}
\end{figure}

Let $\mathcal{F}$ denote the reduced fork diagram obtained by omitting the pair $\alpha_{n}$, $\beta_{n}$, and let $\mathcal{F}'$ denote the diagram obtained from $\mathcal{F}$ by replacing $\alpha_{n-1}$ and $\beta_{n-1}$ with $\alpha'_{n}$ and $\beta'_{n}$.  We'll denote by $\h'_{n-1}$ the diagram for $\DBC{K}$ covering $\mathcal{F}'$.  Furthermore, let
\begin{align*}
\red{\Zcal} &= \left( \alpha_{1} \times \ldots \times \alpha_{n-1} \right) \cap \left( \beta_{1} \times \ldots \times \beta_{n-1} \right) \quad \text{and} \quad\\
\red{\Zcal'} &= \left( \alpha_{1} \times \ldots \times \alpha_{n-2} \times \alpha'_{n} \right) \cap \left( \beta_{1} \times \ldots \times \beta_{n-2} \times \beta'_{n} \right).
\end{align*}

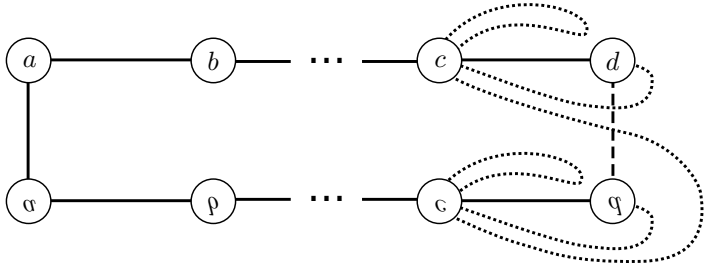
\begin{figure}[h!]
\begin{center}
\begin{minipage}[c]{.60\linewidth}
\labellist
\small
\pinlabel $a$ at 22 122
\pinlabel \rotatebox{180}{\reflectbox{$a$}} at 22 43
\pinlabel $b$ at 122 122
\pinlabel \rotatebox{180}{\reflectbox{$b$}} at 122 43
\pinlabel $c$ at 246 122
\pinlabel \rotatebox{180}{\reflectbox{$c$}} at 246 43
\pinlabel $d$ at 340 122
\pinlabel \rotatebox{180}{\reflectbox{$d$}} at 341 43
\endlabellist
\includegraphics[height = 40mm]{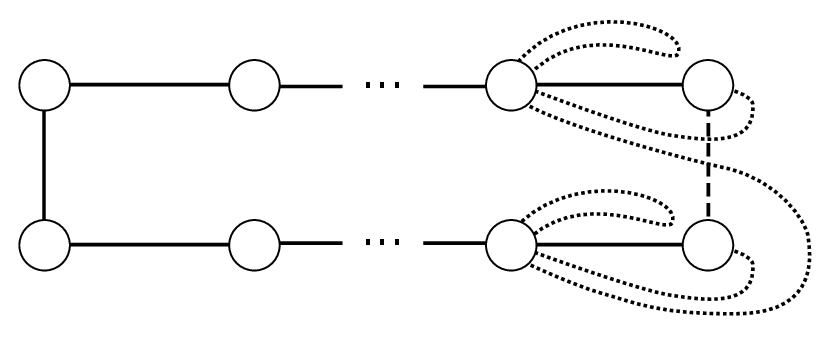}
\end{minipage}
\begin{minipage}[c]{.30\linewidth}
\caption[The Heegaard diagram covering Figure \ref{fig:redalpha}]{The $\alpha$-circles in the Heegaard diagram $\h'_{n-1}$\label{fig:redalphaH}}
\end{minipage}
\end{center}
\end{figure}

The set of attaching circles $\ba' = \{ \ah_{1}, \ldots, \ah_{n-2}, \ah_{n}' \}$ can be obtained from $\ba = \{ \ah_{1}, \ldots, \ah_{n-1} \}$ via a sequence of $n-2$ pointed handleslides.  The sequence of handleslides is illustrated in Figure \ref{fig:alphar} (for $n=5$).  Let $\alpha^1:=\ah_{n-1}$.  For $1 \leq i \leq n-3$, the $i^{th}$ handleslide consists of sliding $\alpha^i$ over $\ah_{n-(i+1)}$ to obtain $\alpha^(i+1)$.  Notice that we choose $\alpha^2$ to be the member of its isotopy class that has the ``fingers'' depicted in Figure \ref{fig:alphar1}.  For $2 \leq i \leq n-3$, we choose $\alpha^{i+1}$ to be the member of its isotopy class which is very close to $\alpha^{i}$, as seen in Figure \ref{fig:alphar2}.  In the final handleslide, we slide $\alpha^{n-2}$ over $\ah_{1}$ to obtain $\alpha^{n-1}$, choosing the representative of its isotopy class that is depicted in Figure \ref{fig:alphar3}.  Notice that the curve is the result of applying two ``finger move'' isotopies to the circle which is precisely the cover of the right-most arc $\alpha_n$ that we omitted to reduce the fork diagram.

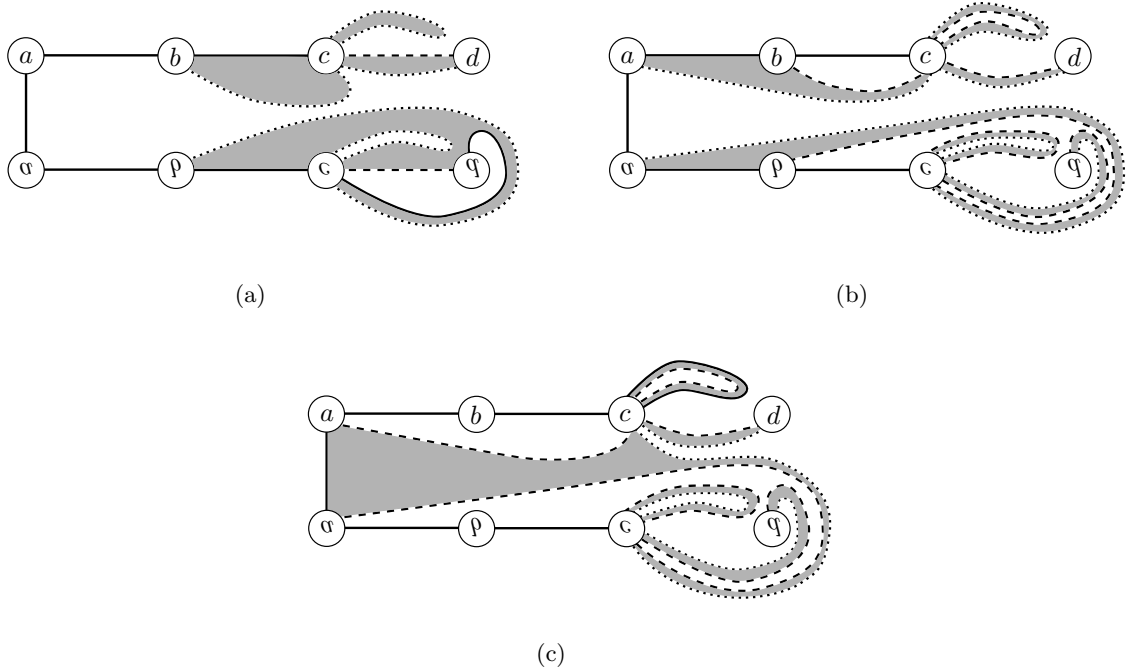
\begin{figure}[h!]
\centering
\subfigure[]{
\labellist 
\small
\pinlabel* {$a$} at 32 140
\pinlabel* {$b$} at 132 140
\pinlabel* {$c$} at 232 140
\pinlabel* {$d$} at 332 140
\pinlabel* {\rotatebox{180}{\reflectbox{$a$}}} at 32 65
\pinlabel* {\rotatebox{180}{\reflectbox{$b$}}} at 132 65
\pinlabel* {\rotatebox{180}{\reflectbox{$c$}}} at 232 65
\pinlabel* {\rotatebox{180}{\reflectbox{$d$}}} at 332 65
\endlabellist
\includegraphics[height = 40mm]{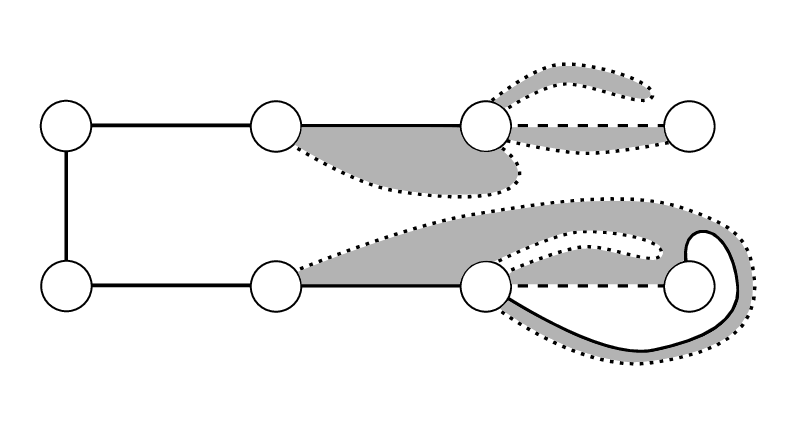}\label{fig:alphar1}}\quad
\subfigure[]{
\labellist 
\small
\pinlabel* {$a$} at 32 140
\pinlabel* {$b$} at 132 140
\pinlabel* {$c$} at 232 140
\pinlabel* {$d$} at 332 140
\pinlabel* {\rotatebox{180}{\reflectbox{$a$}}} at 32 65
\pinlabel* {\rotatebox{180}{\reflectbox{$b$}}} at 132 65
\pinlabel* {\rotatebox{180}{\reflectbox{$c$}}} at 232 65
\pinlabel* {\rotatebox{180}{\reflectbox{$d$}}} at 332 65
\endlabellist
\includegraphics[height = 40mm]{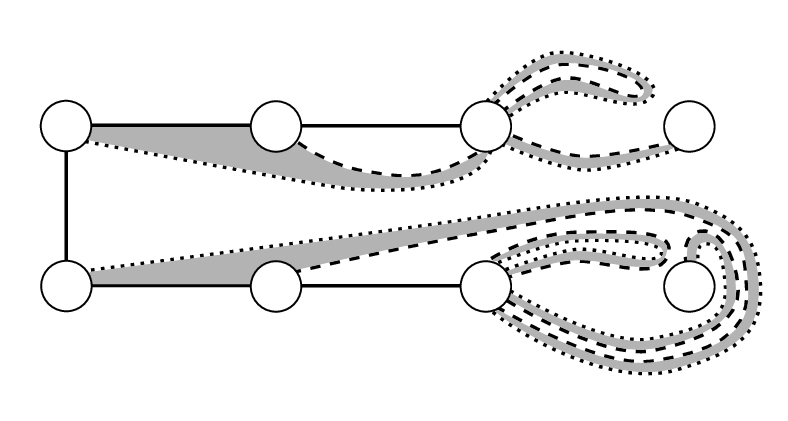}\label{fig:alphar2}}\quad
\subfigure[]{
\labellist 
\small
\pinlabel* {$a$} at 32 140
\pinlabel* {$b$} at 132 140
\pinlabel* {$c$} at 232 140
\pinlabel* {$d$} at 332 140
\pinlabel* {\rotatebox{180}{\reflectbox{$a$}}} at 32 65
\pinlabel* {\rotatebox{180}{\reflectbox{$b$}}} at 132 65
\pinlabel* {\rotatebox{180}{\reflectbox{$c$}}} at 232 65
\pinlabel* {\rotatebox{180}{\reflectbox{$d$}}} at 332 65
\endlabellist 
\includegraphics[height = 40mm]{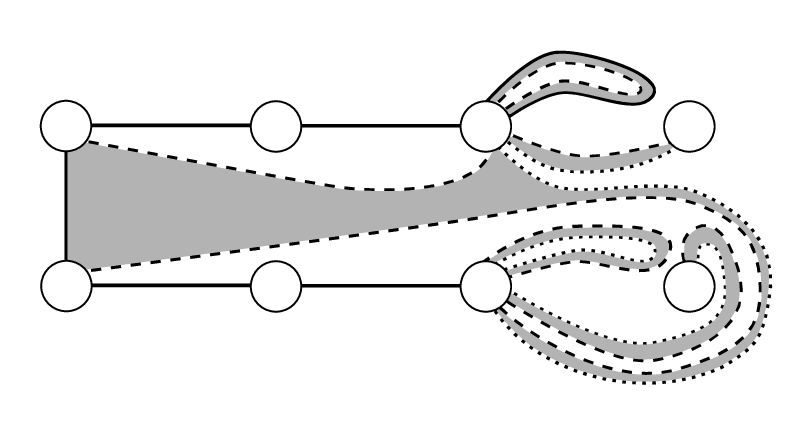}\label{fig:alphar3}}
\caption[Handleslides appearing in the proof of Lemma \ref{lem:RredR}]{Handleslides connecting the set $\{ \ah_1, \ah_2, \alpha^1:=\ah_3 \}$ to $\{ \ah_1, \ah_2, \alpha^{3} := \ah_3' \}$.  In each picture, the old curve $\alpha^{i}$ is dashed, the new curve $\alpha^{i+1}$ is dashed, the unaffected $\ah_{k}$ are solid, and the pair of pants is dashed.
\label{fig:alphar}}
\end{figure}

As the diffeomorphism induced by the braid carries $\ah_i$ to $\bh_i$,  the set  of circles $\bb' =  \{ \bh_{1}, \ldots, \bh_{n-2}, \bh_{n}' \}$ can be obtained from the set $\bb = \{ \bh_{1}, \ldots, \bh_{n-1} \}$ via an analogous sequence of $n-2$ pointed handleslides, where the new curve appearing in the $i^{th}$ handleslide is denoted by $\beta^{i+1}$ and where $\beta^1 := \bh_{n-1}$.  These handeslides are obtained by precomposing the $\alpha$-handleslides with the braid diffeomorphism.  Note that  the parallel $\alpha$-curves in these handleslides (and thus the $\beta$ curves which are their images under the braid diffeomorphism) should be very close together - they're shown somewhat separated in Figure \ref{fig:alphar} for readability.

Now we describe the sequence of Heegaard diagrams that we'll study.  Define the sets of attaching curves $\ba^i, \bb^i$ (for $1 \leq i \leq n-1$), where
$$ \ba^i := \left\{ \ah_{1}, \ldots, \ah_{n-2}, \alpha^{i} \right\} \quad \text{and}\quad
\bb^i := \left\{ \bh_{1}, \ldots, \bh_{n-2}, \beta^{i} \right\}$$

Then for $1 \leq i,j \leq n-1$, define $\h^{i,j}:= \left( \Sigma; \ba^i; \bb^j; +\infty \right)$.  Then we study the sequence of handleslides given by
$$ \h^{1,1} \mapsto \h^{2,1} \mapsto \ldots \mapsto \h^{n-2,1} \mapsto \h^{n-2,2} \mapsto \ldots \mapsto \h^{n-2,n-2} \mapsto \h^{n-1,n-2} \mapsto \h^{n-1,n-1}.$$
Notice that $\h_{n} = \h^{1,1}$ and $\h_{n-1}' = \h^{n-1,n-1}$.  Recall that because they cover fork diagrams, $\h_{n}$ and $\h_{n-1}$ are admissible.  We should address admissibility of intermediate diagrams.

We claim that the move connecting $\h^{1,1}$ (resp. $\h^{n-2,1}$) to $\h^{2,1}$ (resp. $\h^{n-2,2}$) is a composition of an admissibility-preserving pointed handleslide and several admissibility-preserving pointed finger isotopies.  To use Lemma \ref{lem:hsad}, we must exhibit an annular region of the diagram of the form of the region in the hypothesis of the Lemma.  Figure \ref{fig:alphaslide} illustrates this annulus in a local picture of the handleslide $\h^{i,1} \rightarrow \h^{i+1,1}$ for $2 \leq i \leq n-3$, and so those handleslides preserve admissibility.  The required annular region for $\h^{n-2, i+1} \rightarrow \h^{n-2, i+1}$ can be obtained by transforming Figure \ref{fig:alphaslide} by the braid diffeomorphism.  The handleslides $\h^{1,1} \rightarrow \h^{2,1}$ can be realized as the composition of a handleslide like the one in Figure \ref{fig:alphaslide} followed by a pair of finger isotopies, which preserve admissibility by Lemma \ref{lem:isoad}.  Notice also that one can pass from $\h_{n-1}$ to $\h_{n-1}'$ to $\h^{n-1,n-2}$ to $\h^{n-2,n-2}$ via several admissibility-preserving pointed finger isotopies.

\begin{figure}[h!]
\begin{center}
\begin{minipage}[c]{.55\linewidth}
\subfigure[$\alpha^{i-1}$ and $\ah_{n-i}$]{
\labellist 
\small
\pinlabel* {$a$} at 112 105
\pinlabel* {\rotatebox{180}{\reflectbox{$a$}}} at 112 28
\endlabellist
\includegraphics[height = 25mm]{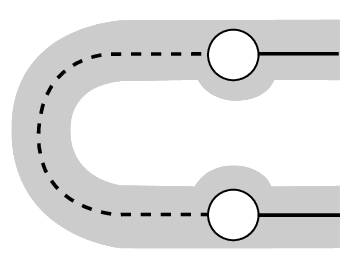}\label{fig:alphaslide1}}\quad
\subfigure[$\alpha^{i}$ and $\ah_{n-i}$]{
\labellist 
\small
\pinlabel* {$a$} at 112 105
\pinlabel* {\rotatebox{180}{\reflectbox{$a$}}} at 112 28
\endlabellist 
\includegraphics[height = 25mm]{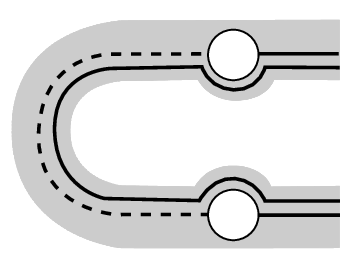}\label{fig:alphaslide2}}
\end{minipage}
\begin{minipage}[c]{.4\linewidth}
\caption[A $\alpha$-curve handleslide]{A local picture of the handleslide of $\alpha^{i-1}$ (solid) over $\ah_{n-i}$ (dashed) to yield $\alpha^{i}$ (solid) appearing in the proof of Lemma \ref{lem:RredR}.  The annulus of Lemma \ref{lem:hsad} is shaded.
\label{fig:alphaslide}}
\end{minipage}
\end{center}
\end{figure}

In Section \ref{sec:redtri} below, we'll construct triangle injections $g_{\alpha}^i, g_{\beta}^i$ for $1 \leq i \leq n-3$, where
\begin{equation*}
g_{\alpha}^i: \tor{\ba^i} \cap \tor{\bb^1} \rightarrow \tor{\ba^{i+1}} \cap \tor{\bb^1} \quad \text{and} \quad g_{\beta}^i: \tor{\ba^{n-1}} \cap \tor{\bb^i} \rightarrow \tor{\ba^{n-1}} \cap \tor{\bb^{i+1}}
\end{equation*}
as well as triangle injections 
$$g_{\alpha}^{n-2}: \tor{\ba^{n-1}} \cap \tor{\bb^{n-2}} \rightarrow \tor{\ba^{n-2}} \cap \tor{\bb^{n-2}} \quad \text{and} \quad
g_{\beta}^{n-2}: \tor{\ba^{n-1}} \cap \tor{\bb^{n-1}} \rightarrow \tor{\ba^{n-1}} \cap \tor{\bb^{n-2}}$$
Additionally, we'll have that
$$\text{Im} \left( g_{\beta}^{n-3} \circ \ldots \circ g_{\beta}^{1} \circ g_{\alpha}^{n-3} \circ \ldots \circ g_{\alpha}^{1} \right) \subseteq \text{Im}\left( g_{\alpha}^{n-2} \circ g_{\beta}^{n-2} \right),$$
which gives a function
$$ g = \left(  \left( g_{\alpha}^{n-2} \circ g_{\beta}^{n-2} \right)^{-1} \circ g_{\beta}^{n-3} \circ \ldots \circ g_{\beta}^{1} \circ g_{\alpha}^{n-3} \circ \ldots \circ g_{\alpha}^{1} \right): \tor{\ba} \cap \tor{\bb} \rightarrow \tor{\ba'} \cap \tor{\bb'}$$
By the version of Lemma \ref{lem:tri} for links (see Remark \ref{rmk:links}), it then suffices to check that $\red{R}(g(\bx)) = \red{R}(\bx)$ for all $\bx \in \tor{\ba} \cap \tor{\bb}$ - this is also checked in Section \ref{sec:redtri} below.
\end{proof}

\begin{rmk}
One might ask whether it is possible to reduce a fork diagram by deleting any pair $\alpha_{k}$, $\beta_{k}$, assuming that they satisfy technical conditions analogous to those found in Definition \ref{def:red}.  This can be done, and the methods used in the proof of Lemma \ref{lem:RredR} could be extended to show that the filtered chain homotopy type doesn't depend on the choice of pair to delete.  However, we don't prove this fact here, as it isn't necessary for proving Proposition \ref{prop:RredR}.
\end{rmk}

\begin{proof}[Proof of Proposition \ref{prop:RredR}]
Let $b \in \B{2n-2}$ be a braid whose plat closure is the knot $K$ (with $n > 1$), and which induces a reducible fork diagram of the form shown in Figure \ref{fig:redfork}.  The plat closure of $b'= b\times (1)^{2} \in \B{2n}$ is the two-component link $L_K$; recall that $\DBC{L_K} \cong \DBCs{K}$.  Let $\h_{n}$ and $\h_{n-1}$ be the Heegaard diagrams for $\DBC{L_K}$ obtained from $b'$ as in the statement of Lemma \ref{lem:RredR}.

Let $\mathfrak{s}_{0}$ denote the torsion $\text{Spin}^{c}$-structure on $\SxS$, and let $\sfr \in \text{Spin}^{c}(\DBC{K})$.  Equipped with their respective $\red{R}$ filtrations, Lemma \ref{lem:RredR} provides that the complexes $\widehat{CF}(\h_{n-1}, \sfr \# \sO)$ and $\widehat{CF}(\h_{n}, \sfr \# \sO)$ have the same filtered chain homotopy type.

Further, let $\h'$ denote the genus-$(n-1)$ Heegaard diagram for $\DBCs{K}$ obtained from $b$ in the unreduced sense.  One can see that this diagram is isotopic to $\h_{n}$.  Also, the fork diagram covered by $\h_{n}$ only differs from the fork diagram covered by $\h'$ by an extra pair of punctures which are far from any arcs in the diagram.  Identifying Bigelow generators in the obvious way and observing that $s_{R}(b) = s_{\red{R}}(b'),$ one obtains that the $R$-filtered complex $\widehat{CF}(\h')$ has the same filtered chain homotopy type as the $\red{R}$-filtered complex $\widehat{CF}(\h_{n})$.

Now denote by $\h$ the genus-$(n-2)$ Heegaard diagram for $\DBC{K}$ obtained from $b$ in the reduced sense by omitting the pair $\alpha_{n-1},\beta_{n-1}$.  Letting $\h_{0}$ denote the standard Heegaard diagram for $\SxS$ of genus 1 (containing only generators from $\sO$), one can see that $\h_{n-1} = \h \# \h_{0}$ (where the sum region is near $+\infty$).   As a result, all generators of $\widehat{CF}(\h_{n-1})$ lie in $\text{Spin}^{c}$ structures of the form $\sfr \#\sO$.  Further, the connected sum provides a natural correspondence between each generator $\bx$ for $\widehat{CF}(\h,\sfr)$ and a pair of generators $\bx y$ and $\bx z$ for $\widehat{CF}(\h_{n-1},\sfr \#\sO)$, where $y,z \in \ah_{n} \cap \bh_{n}$.  Now the fork diagram covered by $\h_{n-1}$ is obtained from the one covered by $\h$ via a connected sum with $\mathcal{F}_0$, the two-puncture fork diagram for $1^2 \in \B{2}$. One can then verify directly that
\begin{align*}
\red{ P}^{*}(y) = 2, \quad \red{ P}^{*}(z) = 0, \quad &\red{Q}^{*}(y) = 1, \quad \red{Q}^{*}(z) = 0,\\
 \red{T}(\bx y) =  \red{T}(\bx),  \quad \red{T}(\bx z) =  \red{T}(\bx), \quad &\text{and} \quad s_{\red{R}}(b') = s_{\red{R}}(b) - \frac{1}{2}, \quad \text{and thus}\\
\red{R}(\bx y) =  \red{R}(\bx) + \frac{1}{2} \quad &\text{and} \quad \red{R}(\bx z) =  \red{R}(\bx) - \frac{1}{2} 
\end{align*}

As a result, if one equips $\widehat{CF}(\h_{n-1})$ and $\widehat{CF}(\h)$ with their respective $\red{R}$-filtrations, then the filtered complex $\widehat{CF}(\h_{n-1})$ has the same filtered chain homotopy type as the filtered complex $\widehat{CF}(\h) \otimes V.$
Furthermore, notice that
$$ \tld{gr}(\bx y) = \tld{gr}(\bx) + \frac{1}{2} \quad \text{and} \quad \tld{gr}(\bz) = \tld{gr}(\bx) - \frac{1}{2},$$
and so $\widehat{CF}(\h_{n-1})$ and $\widehat{CF}(\h) \otimes W$ have the same $\red{\rho}$-filtered chain homotopy types.  The results follow.
\end{proof}

In light of this relationship, one can now obtain an invariance result for the reduced theory.

\begin{proof}[Proof of Theorem \ref{thm:redRthm}]
Let $\red{\h}_{1}$ and $\red{\h}_{2}$ be two Heegaard diagrams for $\DBC{K}$ obtained from reducible fork diagrams in the usual way, and let $\h_{1}$ $\h_{2}$ be corresponding diagrams for $\DBCs{K}$ obtained from unreduced fork diagrams.  Either fix $F$ to stand for $R$ and $\red{F}$ to stand for $\red{R}$, or fix $F$ to stand for $\rho$ and $\red{F}$ to stand for $\red{\rho}$.  For $\mathfrak{s} \in \text{Spin}^{c}(\DBC{K})$ and $\mathfrak{s}_{0} \in \text{Spin}^{c}(\SxS)$ torsion, equip $\widehat{CF}(\red{\h}_{i},\mathfrak{s} ; \mathbb{F})$ with $\red{F}$ filtrations and equip $\widehat{CF}(\h_{i},\mathfrak{s} \# \mathfrak{s}_{0}; \mathfrak{F})$ with $F$ filtrations.   Proposition \ref{prop:cancel} applies and the result follows via Proposition \ref{prop:RredR} and Theorem 1 from \cite{et:R}.
\end{proof}

\subsubsection{Triangle injections for Lemma \ref{lem:RredR}}\label{sec:redtri}

We'll define the sequence of triangle injections required for the proof of Lemma \ref{lem:RredR}, exhibit the required 3-gons, and check that the composition $g$ preserves $\red{R}$.

We first define the triangle injection $g_{\alpha}^1:\tor{\ba^1} \cap \tor{\bb^1} \rightarrow \tor{\ba^2} \cap \tor{\bb^1}$.  Consider some $\bw x \in \tor{\ba^1} \cap \tor{\bb^1}$, where $x \in \ah_{n-1} \cap \bh_j$, for $1 \leq j \leq n-1$.  Then let $g_{\alpha}^1(\bw x) = \bw y$, where $y \in \alpha^{2} \cap \bh_j$ is as indicated in Figure \ref{fig:ra1}.

\begin{figure}[h]
\centering
\begin{minipage}[c]{.55\linewidth}
\subfigure[]
{
\labellist
\small
\pinlabel* $x$ at 150 145
\pinlabel* $y$ at 30 90
\pinlabel* {$a$} at 25 125
\pinlabel* {$b$} at 120 125
\pinlabel* {\rotatebox{180}{\reflectbox{$a$}}} at 25 45
\pinlabel* {\rotatebox{180}{\reflectbox{$b$}}} at 120 45
\endlabellist
\includegraphics[height = 36mm]{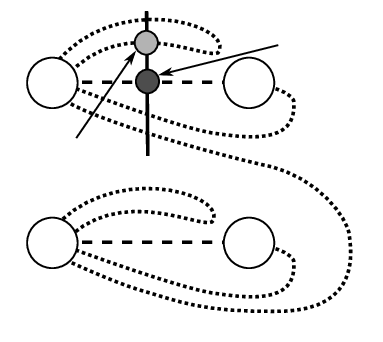}}\quad
\subfigure[]
{
\labellist
\small
\pinlabel* $x$ at 30 90
\pinlabel* $y$ at 120 70
\pinlabel* {$a$} at 25 125
\pinlabel* {$b$} at 120 125
\pinlabel* {\rotatebox{180}{\reflectbox{$a$}}} at 25 45
\pinlabel* {\rotatebox{180}{\reflectbox{$b$}}} at 120 45
\endlabellist
\includegraphics[height = 36mm]{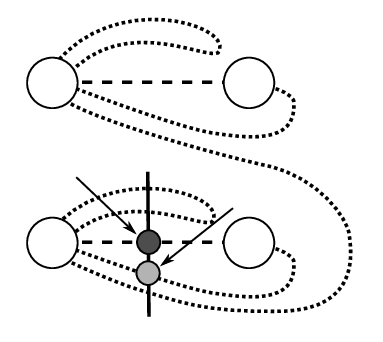}}
\end{minipage}
\begin{minipage}[c]{.40\linewidth}
\caption[Possible inputs for $g_{\alpha}^1$ in Lemma \ref{lem:RredR}]{ Two possibilities for intersections $x \in \ah_{n-1}\cap\bh_{j}$ and $y \in \alpha^2 \cap \bh_j$ appearing in the definition of the triangle injection $g_{\alpha}^1$.  In each picture, $\bh_j$ is solid, $\alpha^1 = \ah_{n-1}$ is dashed, and $\alpha^2$ is dotted.\label{fig:ra1}}
\end{minipage}
\end{figure}

We illustrate the 3-gons $\psi_{\alpha}^{1,+}$ and $\psi_{\alpha}^{1,-}$ associated to $g_{\alpha}^1$.  Recall from Definition \ref{df:trimap} that we seek $\psi_{\alpha}^{1,+} \in \pi_2(\boldsymbol{\theta}_{\ba^{+}\ba^1}, \bw x, \bw y^+)$ and $\psi_{\alpha}^{1,-} \in \pi_2(\boldsymbol{\theta}_{\ba^1\ba^{-}}, \bw y^{-}, \bw y)$, where $\bw^{\pm} y^\pm \in \tor{\ba^{\pm}} \cap \tor{\bb}$ and $\ba^{\pm}$ are small isotopic pushoffs of $\ba^2$ with corresponding curves intersecting transversely in pairs of points.  Components involving $x$ appear in Figure \ref{fig:ra1tri}.  All other components are ``small'' 3-sided regions as seen in Figure \ref{fig:isotri}.

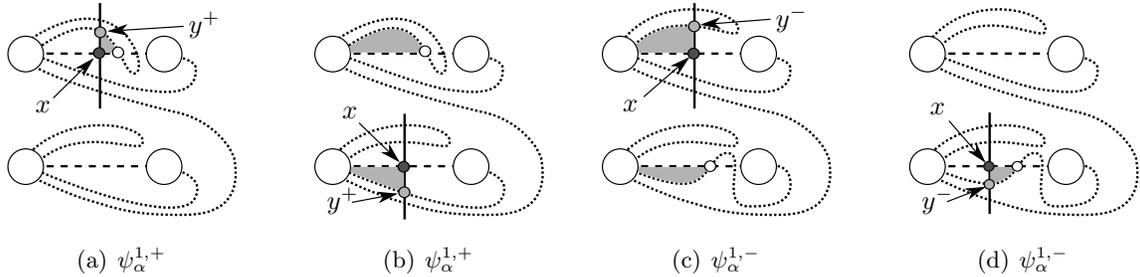
\begin{figure}[h]
\centering
\subfigure[$\psi_{\alpha}^{1,+}$]
{
\labellist
\small
\pinlabel* $x$ at 35 84
\pinlabel* $y^+$ at 145 140
\endlabellist
\includegraphics[height = 32mm]{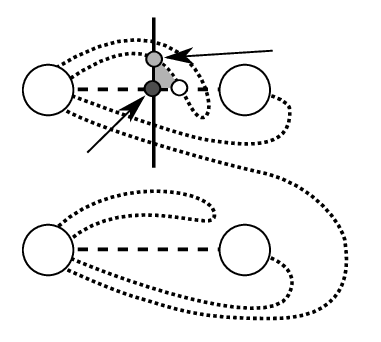}}\quad
\subfigure[$\psi_{\alpha}^{1,+}$]
{
\labellist
\small
\pinlabel* $x$ at 35 80
\pinlabel* $y^+$ at 30 20
\endlabellist
\includegraphics[height = 32mm]{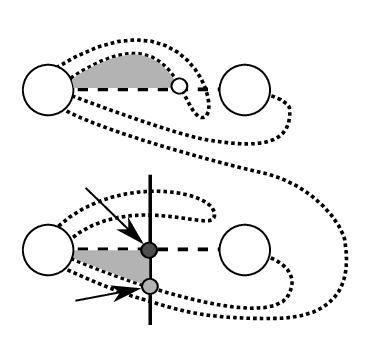}}
\subfigure[$\psi_{\alpha}^{1,-}$]
{
\labellist
\small
\pinlabel* $x$ at 30 85
\pinlabel* $y^-$ at 140 144
\endlabellist
\includegraphics[height = 32mm]{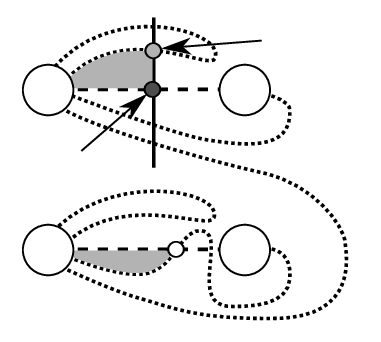}}\quad
\subfigure[$\psi_{\alpha}^{1,-}$]
{
\labellist
\small
\pinlabel* $x$ at 30 83
\pinlabel* $y^-$ at 30 20
\endlabellist
\includegraphics[height = 32mm]{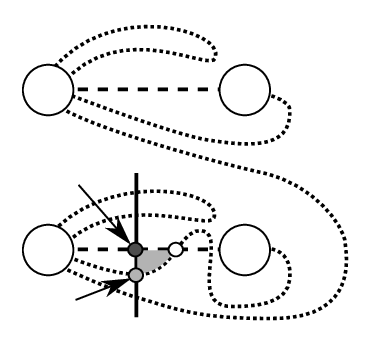}}
\caption[Components of 3-gons for $g_{\alpha}^1$ in the proof of Lemma \ref{lem:RredR}]{ Some components of 3-gons associated to $g_{\alpha}^1$.  Parts (a) and (b) show possible components of $\psi_{\alpha}^{1+}$ and (c) and (d) show possible components of $\psi_{\alpha}^{1-}$.  The dashed curve is $\ah_{n-1}$ and the dotted curve is $(\alpha^2)^+$ in (a) and (b) and $(\alpha^2)^-$ in (c) and (d).  In each picture, the white dot represents a component of the appropriate top-degree $\boldsymbol{\theta}$-generator.\label{fig:ra1tri}}
\end{figure}

\begin{figure}[h]
\centering
\begin{minipage}[c]{.45\linewidth}
\subfigure[$\psi_{\alpha}^{1,+}$]
{
\labellist
\small
\pinlabel* $x$ at 25 8
\pinlabel* $y^+$ at 45 30
\endlabellist
\includegraphics[height = 32mm]{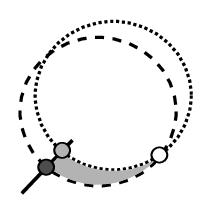}}\quad
\subfigure[$\psi_{\alpha}^{1,-}$]
{
\labellist
\small
\pinlabel* $x$ at 20 8
\pinlabel* $y^-$ at 44 30
\endlabellist
\includegraphics[height = 32mm]{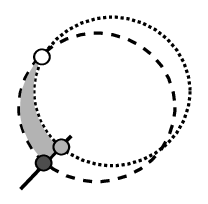}}
\end{minipage}
\begin{minipage}[c]{.50\linewidth}
\caption[Small components of 3-gons for $g_{\alpha}^1$ in the proof of Lemma \ref{lem:RredR}]{Small components of 3-gons associated to $g_{\alpha}^1$\label{fig:isotri}}
\end{minipage}
\end{figure}

Now for $2 \leq i \leq n-3$, we define $g_{\alpha}^i:\tor{\ba^i} \cap \tor{\bb^1} \rightarrow \tor{\ba^{i+1}} \cap \tor{\bb^1}$, the triangle injection associated to the handleslide $\ba^i \rightarrow \ba^{i+1}$.  Consider $\bw x \in \tor{\ba^i} \cap \tor{\bb^1}$, where $x \in \alpha^i \cap \bh_j$ for $1 \leq j \leq n-1$.  We let $g_{\alpha}^i(\bw x) = \bw y$, where $y \in \alpha^{i+1} \cap \bh_j$ is the obvious point close to $x$.  The 3-gon components involving $x$ are shown in Figure \ref{fig:hstri}; all other components are as in Figure \ref{fig:isotri}.

\begin{figure}[h]
\centering
\begin{minipage}[c]{.60\linewidth}
\subfigure[$\psi_{\alpha}^{i,+}$]
{
\labellist
\small
\pinlabel* {$x$} at 140 45
\pinlabel* {$y^+$} at 105 45
\endlabellist
\includegraphics[height = 32mm]{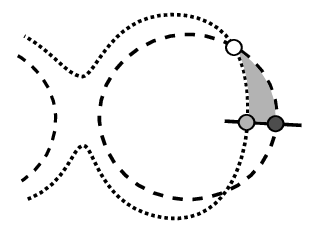}}\quad
\subfigure[$\psi_{\alpha}^{i,-}$]
{
\labellist
\small
\pinlabel* {$x$} at 140 45
\pinlabel* {$y^-$} at 105 45
\endlabellist
\includegraphics[height = 32mm]{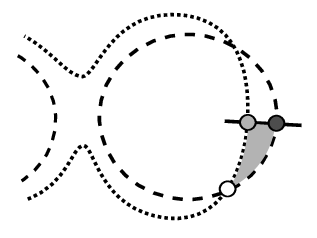}}
\end{minipage}
\begin{minipage}[c]{.35\linewidth}
\caption[Components of 3-gons for $g_{\alpha}^{i} (i>1)$ in the proof of Lemma \ref{lem:RredR}]{Some components of 3-gons associated to $g_{\alpha}^i$ with $2 \leq i \leq n-3.$  In each picture, the leftmost dashed curve is $\ah_{n-i-1}.$\label{fig:hstri}}
\end{minipage}
\end{figure}

The triangle injection $g_{\beta}^1:\tor{\ba^{n-2}} \cap \tor{\bb^1} \rightarrow \tor{\ba^{n-2}} \cap \tor{\bb^{2}}$ is defined in a way that is somewhat analogous to $g_{\alpha}^1$.  For $\bw x \in \tor{\ba^{n-2}} \cap \tor{\bb^1}$ with $x \in \alpha \cap \bh_{n-1}$ (where $\alpha$ is some curve in the tuple $\ba^{n-2}$), we define $g_{\beta}^1(\bw x) = \bw y$, where $y \in \alpha \cap \beta^2$ is as indicated in Figure \ref{fig:rb1}.

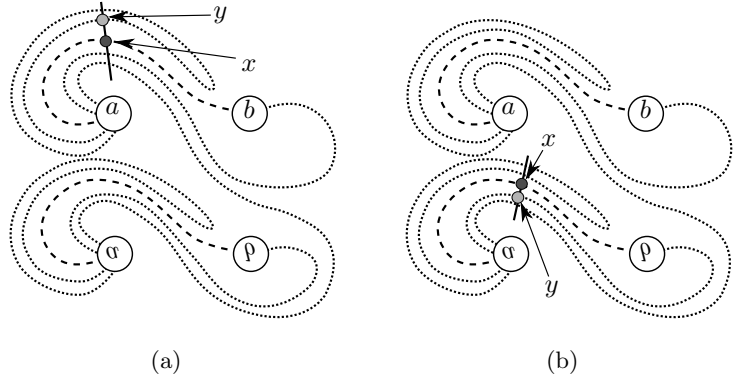
\begin{figure}[h]
\centering
\subfigure[]
{
\labellist
\small
\pinlabel* {$a$} at 80 155
\pinlabel* {$b$} at 175 155
\pinlabel* {$x$} at 175 185
\pinlabel* {$y$} at 155 220
\pinlabel* {\rotatebox{180}{\reflectbox{$a$}}} at 80 55
\pinlabel* {\rotatebox{180}{\reflectbox{$b$}}} at 175 55
\endlabellist
\includegraphics[height = 46mm]{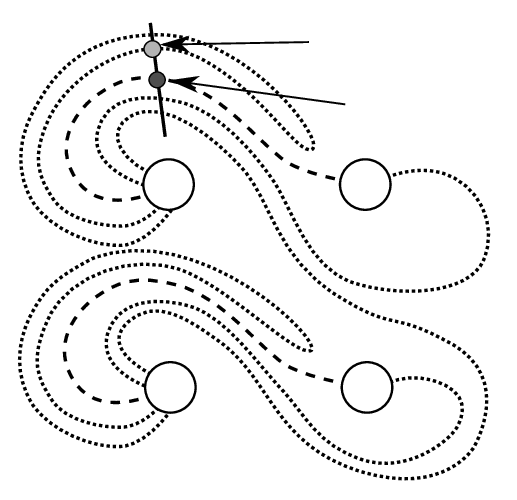}}\quad
\subfigure[]
{
\labellist
\small
\pinlabel* {$a$} at 80 155
\pinlabel* {$b$} at 175 155
\pinlabel* {$x$} at 108 133
\pinlabel* {$y$} at 110 30
\pinlabel* {\rotatebox{180}{\reflectbox{$a$}}} at 80 55
\pinlabel* {\rotatebox{180}{\reflectbox{$b$}}} at 175 55
\endlabellist
\includegraphics[height = 46mm]{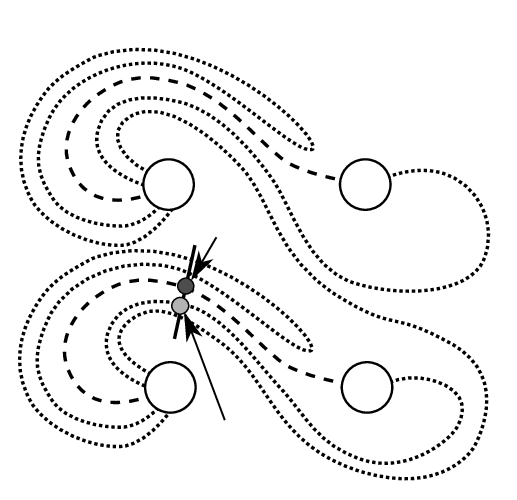}}
\caption[Possible inputs for $g_{\beta}^1$ in the proof of Lemma \ref{lem:RredR}]{ Two possibilities for intersections $x \in \alpha \cap \beta^1$ and $y \in \alpha \cap \beta^2$ appearing in the definition of the triangle injection $g_{\beta}^1$.  In each picture, $\alpha$ is solid, $\beta^1 = \bh_{n-1}$ is dashed, and $\beta^2$ is dotted.\label{fig:rb1}}
\end{figure}

\begin{figure}[h]
\centering
\subfigure[$\psi_{\beta}^{1,+}$]
{
\labellist
\small
\pinlabel* {$x$} at 106 140
\pinlabel* {$y^+$} at 145 215
\endlabellist
\includegraphics[height = 32mm]{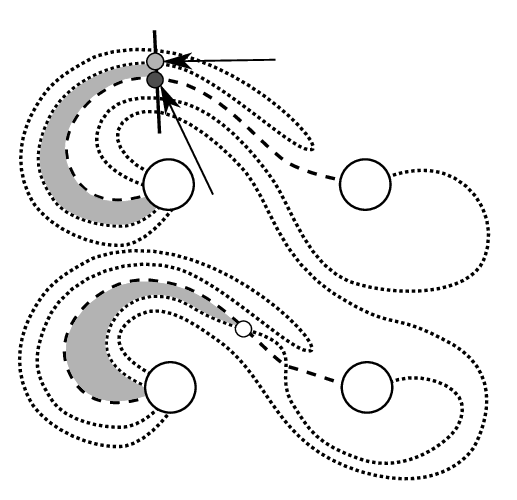}}\quad
\subfigure[$\psi_{\beta}^{1,+}$]
{
\labellist
\small
\pinlabel* {$x$} at 110 135
\pinlabel* {$y^+$} at 115 32
\endlabellist
\includegraphics[height = 32mm]{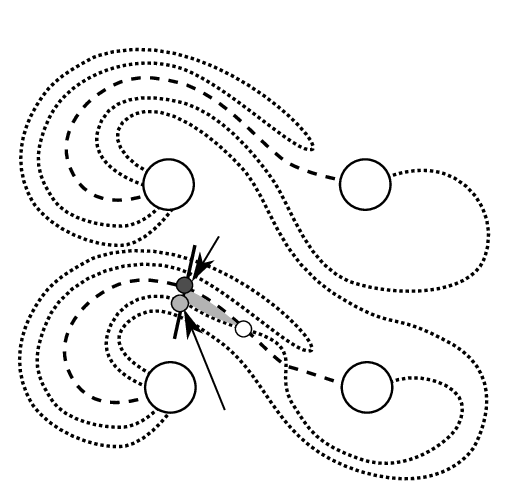}}
\subfigure[$\psi_{\beta}^{1,-}$]
{
\labellist
\small
\pinlabel* {$x$} at 107 133
\pinlabel* {$y^-$} at 170 215
\endlabellist
\includegraphics[height = 32mm]{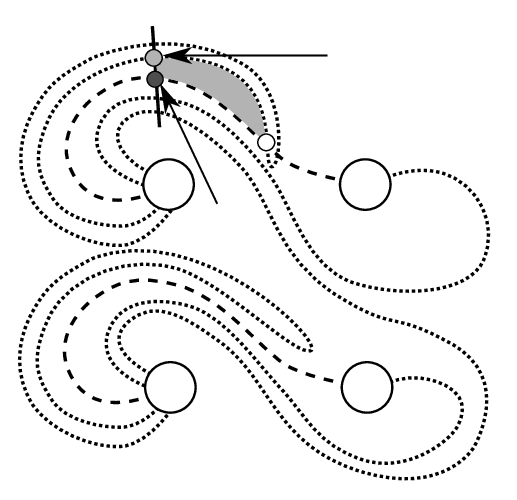}}\quad
\subfigure[$\psi_{\beta}^{1,-}$]
{
\labellist
\small
\pinlabel* {$x$} at 106 138
\pinlabel* {$y^-$} at 115 30
\endlabellist
\includegraphics[height = 32mm]{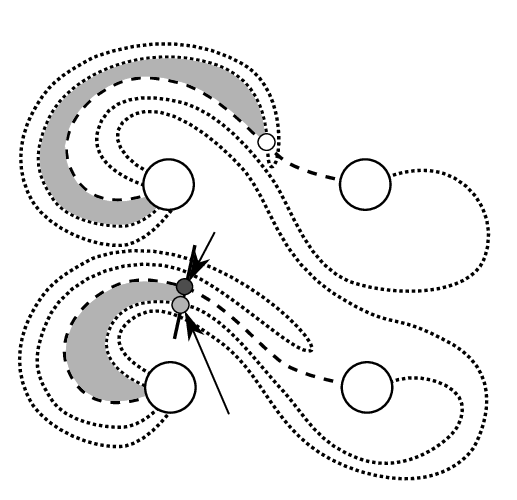}}
\caption[Components of 3-gons for $g_{\beta}^1$ in the proof of Lemma \ref{lem:RredR}]{ Some components of 3-gons associated to $g_{\beta}^1$.  Parts (a) and (b) show possible components of $\psi_{\beta}^{1+}$ and (c) and (d) show possible components of $\psi_{\beta}^{1-}$.\label{fig:rb1tri}}
\end{figure}
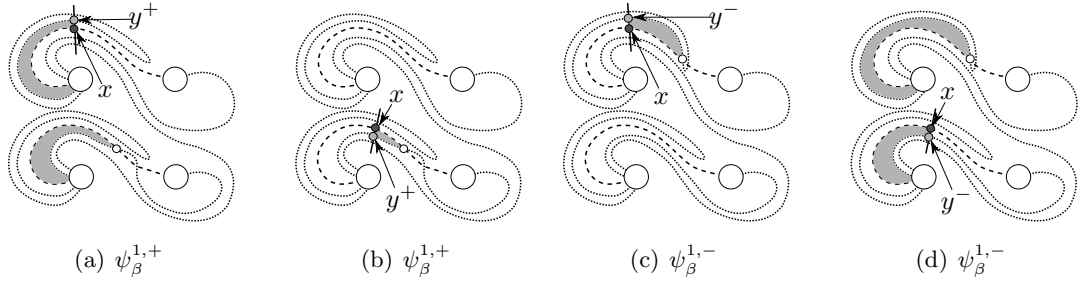

Figure \ref{fig:rb1tri} illustrates the component involving $x$ for each 3-gon $\psi_{\beta}^{1,\pm}$ associated to $g_{\beta}^1$.  Other components are as in Figure \ref{fig:isotri}.  The definitions of $g_{\beta}^i$ for $2 \leq i \leq n-3$ are completely analogous to those for $g_{\alpha}^i$, with associated 3-gons as in Figures \ref{fig:isotri} and \ref{fig:hstri}.

Notice that the last two handleslides $\ba^{n-2} \mapsto \ba^{n-1}$ and $\bb^{n-2} \mapsto \bb^{n-1}$ are inverses of handleslides which are of the form depicted in Figure \ref{fig:hsad}.  The injections $g_{\alpha}^{n-2}: \tor{\ba^{n-1}} \cap \tor{\bb^{n-2}} \rightarrow \tor{\ba^{n-2}} \cap \tor{\bb^{n-2}}$ and $g_{\beta}^{n-2}: \tor{\ba^{n-1}} \cap \tor{\bb^{n-1}} \rightarrow \tor{\ba^{n-1}} \cap \tor{\bb^{n-2}}$ can thus be defined in a way that is analogous to the above definition of $g_{\alpha}^i$ for $2 \leq i \leq n-3$.

It is easily verified that
$$ \text{Im} \left( g_{\beta}^{n-3} \circ \ldots \circ g_{\beta}^{1} \circ g_{\alpha}^{n-3} \circ \ldots \circ g_{\alpha}^{1} \right) \subseteq \text{Im}\left( g_{\alpha}^{n-2} \circ g_{\beta}^{n-2} \right).$$
As a result, we can define
$$ g = \left(  \left( g_{\alpha}^{n-2} \circ g_{\beta}^{n-2} \right)^{-1} \circ g_{\beta}^{n-3} \circ \ldots \circ g_{\beta}^{1} \circ g_{\alpha}^{n-3} \circ \ldots \circ g_{\alpha}^{1} \right): \tor{\ba} \cap \tor{\bb} \rightarrow \tor{\ba'} \cap \tor{\bb'}.$$
We'll break out analysis of gradings into two cases, depending on the form of the initial element of $\Zcal$.

First consider $\bw x \in \Zcal$, where $x \in \alpha_{n-1} \cap \beta_{n-1}$ and $\bw = \left\{ w_1, \ldots, w_{n-2} \right\}$.  Local pictures of $\mathcal{F}$ and $\mathcal{F}'$ can be seen in Figures \ref{fig:red1a} and Figure \ref{fig:red1b}, respectively.

\begin{figure}[h!]
\centering
\begin{minipage}[c]{.55\linewidth}
\labellist
\small
\pinlabel* $x$ at 200 50
\endlabellist
\includegraphics[height = 20mm]{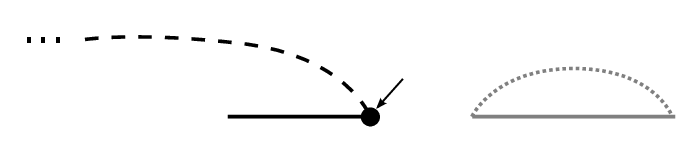}
\end{minipage}
\begin{minipage}[c]{.35\linewidth}
\caption[The fork diagram $\Fcal$ in case I of Lemma \ref{lem:RredR}]{The fork diagram $\mathcal{F}$.  The solid arc is $\alpha_{n-1}$ and the dashed arc is $\beta_{n-1}$.  The omitted arcs $\alpha_{n}$ and $\beta_{n}$ are grayed. \label{fig:red1a}}
\end{minipage}
\end{figure}

\begin{figure}[h!]
\centering
\begin{minipage}[c]{.55\linewidth}
\labellist
\small
\pinlabel* $z$ at 190 115
\endlabellist
\includegraphics[height = 35mm]{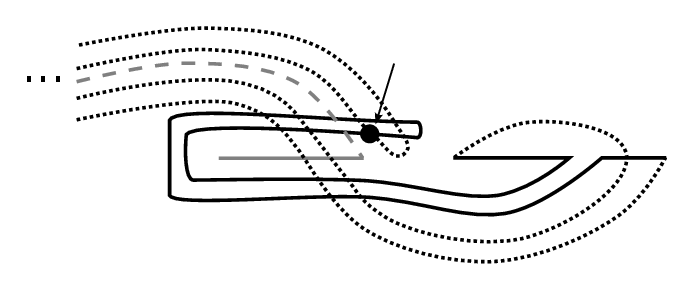}
\end{minipage}
\begin{minipage}[c]{.35\linewidth}
\caption[The fork diagram $\Fcal'$ in case I of Lemma \ref{lem:RredR}]{The fork diagram $\mathcal{F}'$.  The solid arc is $\alpha'_{n}$ and the dotted arc is $\beta'_{n}$.  The omitted arcs $\alpha_{n-1}$ and $\beta_{n-1}$ are grayed. \label{fig:red1b}}
\end{minipage}
\end{figure}

Figure \ref{fig:red1hd} shows local regions of the diagram $\h^{1,1}$ (resp. $\h^{n,n}$) covering a neighborhood of the puncture $\mu_{2n-2}$ in the fork diagrams $\Fcal$ (resp. $\Fcal'$).  It is easily verified that either $g(\bw x) = \bw z$ or $g(\bw x) = \bw z'$.

\begin{figure}[h]
\centering
\begin{minipage}[h]{.5\linewidth}
\subfigure[$\h^{1,1}$]
{
\labellist
\small
\pinlabel* $x$ at 120 120
\endlabellist
\includegraphics[height = 36mm]{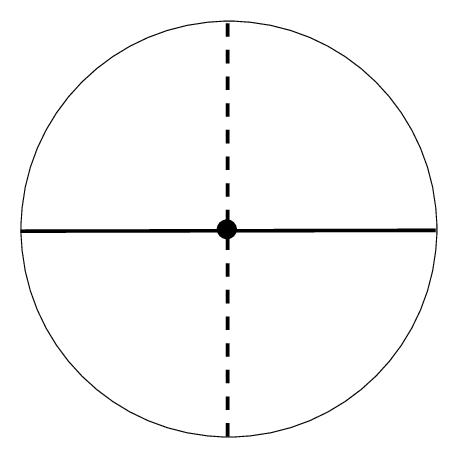}\label{fig:red1hd1}}\quad
\subfigure[$\h^{n,n}$]
{
\labellist
\small
\pinlabel* $z$ at 110 140
\pinlabel* $z'$ at 115 78
\endlabellist
\includegraphics[height =36mm]{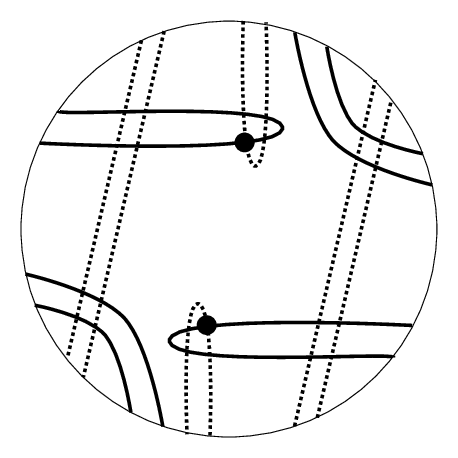}\label{fig:red1hd2}}
\end{minipage}
\begin{minipage}[c]{.45\linewidth}
\caption[Heegaard diagrams covering Figures \ref{fig:red1a} and \ref{fig:red1b}]{Local pictures of diagrams covering a neighborhood of the puncture $\mu_{2n-2}$ \label{fig:red1hd}}
\end{minipage}
\end{figure}


\begin{figure}[h]
\centering
\begin{minipage}[c]{.55\linewidth}
\labellist
\small
\pinlabel* $z$ at 220 160
\endlabellist
\includegraphics[height = 40mm]{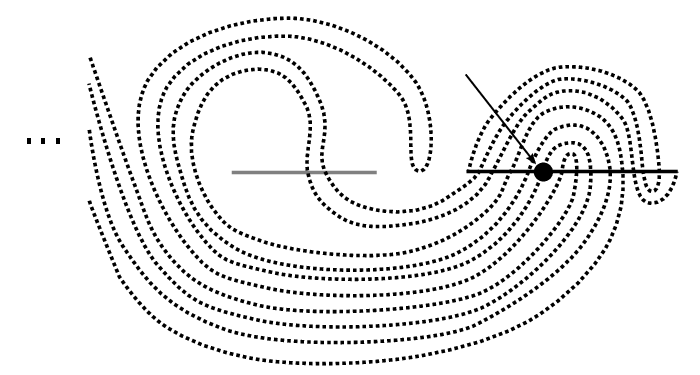}
\end{minipage}
\begin{minipage}[c]{.40\linewidth}
\caption{Result of an isotopy on Figure \ref{fig:red1b}}
\label{fig:red1biso}
\end{minipage}
\end{figure}

However, we should justify that $\red{R}(\bw z) = \red{R}(\bw x)$.  Strictly speaking, we should calculate the gradings for $\bw z$ after performing an isotopy on $\mathcal{F}'$ so that $\alpha'_{n}$ is a horizontal joining $\mu_{2n-1}$ and $\mu_{2n}$; the result is shown in Figure \ref{fig:red1biso}.

One can verify using the fork diagrams in Figured \ref{fig:red1a} and \ref{fig:red1biso} that indeed $\red{Q}(\bw z) = \red{Q}(\bw x)$, $\red{ P}(\bw z) = \red{ P}(\bw x)$, and $\red{T}(\bw z) = \red{T}(\bw x)$.  So, $g$ preserves $\red{R}$ in this case.

Now instead let $\bz x y \in \Zcal$ with $x \in \alpha_{n-1} \cap \beta_{j}$ and $y \in \alpha_{k} \cap \beta_{n-1}$ (where $j,k < n-1$).  Local regions of the fork diagram $\mathcal{F}$ containing possible choices for $x$ and $y$ are shown in Figure \ref{fig:red2a}, and corresponding local regions of $\mathcal{F}'$ are shown in Figure \ref{fig:red2b}.  Corresponding regions from the Heegaard diagrams $\h^{1,1}$ and $\h^{n,n}$ are shown in Figures \ref{fig:red2hd1} and \ref{fig:red2hd2}.  One can see that
\begin{align*}
g(\bz x_1 y_1) &\in \left\{ \bz u_1 v_1, \bz u_1' v_1, \bz u_1 v_1', \bz u_1' v_1' \right\}, \qquad
g(\bz x_1 y_2) &\in \left\{ \bz u_1 v_2, \bz u_1' v_2 \right\}, \qquad
g(\bz x_1 y_2') &\in \left\{ \bz u_1 v_2', \bz u_1' v_2' \right\},\\
g(\bz x_2 y_1) &\in \left\{ \bz u_2 v_1, \bz u_3' v_1, \bz u_2 v_1', \bz u_3' v_1' \right\}, \qquad
g(\bz x_2 y_2) &\in \left\{ \bz u_2 v_2, \bz u_3' v_2 \right\}, \qquad
g(\bz x_2 y_2') &\in \left\{ \bz u_2 v_2', \bz u_3' v_2' \right\},\\
g(\bz x_2' y_1) &\in \left\{ \bz u_2' v_1, \bz u_3 v_1, \bz u_2' v_1', \bz u_3 v_1' \right\}, \qquad
g(\bz x_2' y_2) &\in \left\{ \bz u_2' v_2, \bz u_3 v_2 \right\}, \qquad
g(\bz x_2' y_2') &\in \left\{ \bz u_2' v_2', \bz u_3 v_2' \right\}
\end{align*}

\begin{figure}[h!]
\centering
\subfigure[$x_{i} \in \alpha_{n-1} \cap \beta_{j_{i}}$, $i = 1,2$]
{
\labellist
\small
\pinlabel* $x_{1}$ at 75 80
\pinlabel* $x_{2}$ at 180 105
\endlabellist
\includegraphics[height = 30mm]{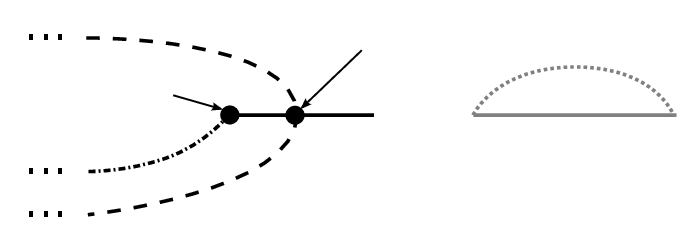}}\quad
\subfigure[$y_{i} \in \alpha_{k_{i}} \cap \beta_{n-1}$, $i = 1,2$]
{
\labellist
\small
\pinlabel* $y_{1}$ at 80 125
\pinlabel* $y_{2}$ at 185 110
\endlabellist
\includegraphics[height =30mm]{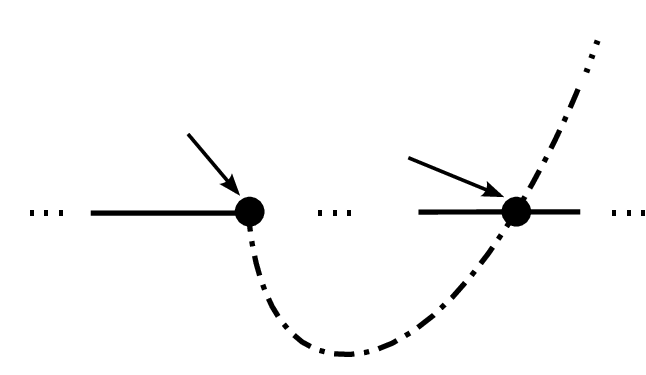}}
\caption[Fork diagram in case II of Lemma \ref{lem:RredR}]{Local pictures of $\mathcal{F}$ appearing case II of the proof of Lemma \ref{lem:RredR}.  One could choose $x$ to be either $x_1$ or $x_2$ and $y$ to be either $y_1$ or $y_2$.}
\label{fig:red2a}
\end{figure}

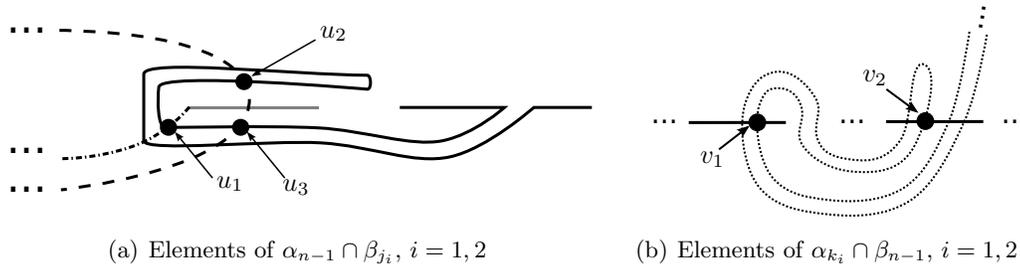
\begin{figure}[h!]
\centering
\subfigure[Elements of $\alpha_{n-1} \cap \beta_{j_{i}}$, $i = 1,2$]
{
\labellist
\small
\pinlabel* $u_{1}$ at 130 20
\pinlabel* $u_{2}$ at 185 100
\pinlabel* $u_{3}$ at 165 18
\endlabellist
\includegraphics[height = 30mm]{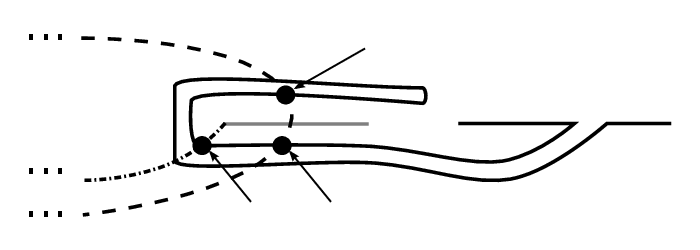}\label{fig:red2b1}}
\subfigure[Elements of $\alpha_{k_{i}} \cap \beta_{n-1}$, $i = 1,2$]
{
\labellist
\small
\pinlabel* $v_{1}$ at 60 50
\pinlabel* $v_{2}$ at 190 115
\endlabellist
\includegraphics[height =30mm]{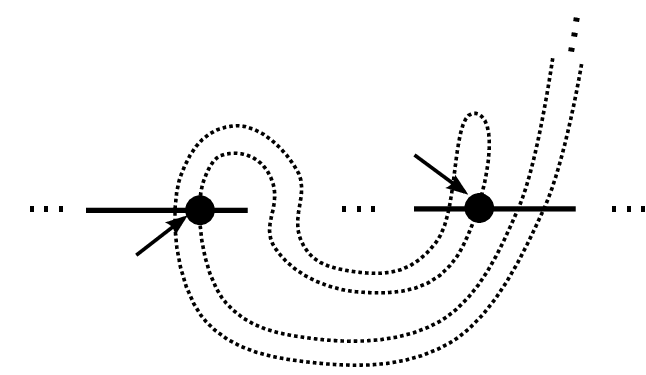}}\quad
\caption[The fork diagram $\Fcal'$ in case II of Lemma \ref{lem:RredR}]{Local pictures of the fork diagram $\mathcal{F}'$}
\label{fig:red2b}
\end{figure}

\begin{figure}[h!]
\centering
\subfigure[Near $\mu_{2n-3}$]
{
\labellist
\small
\pinlabel* $x_{1}$ at 165 75
\pinlabel* $x_{2}$ at 60 140
\pinlabel* $x'_{2}$ at 170 135
\endlabellist
\includegraphics[height =36mm]{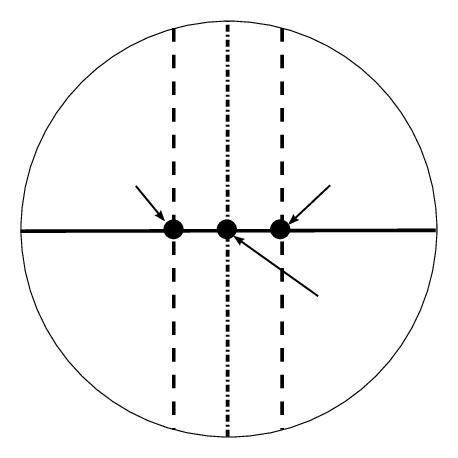}}\quad
\subfigure[Near $\mu_{2k_{1}-1}$]
{
\labellist
\small
\pinlabel* $y_{1}$ at 125 125
\endlabellist
\includegraphics[height = 36mm]{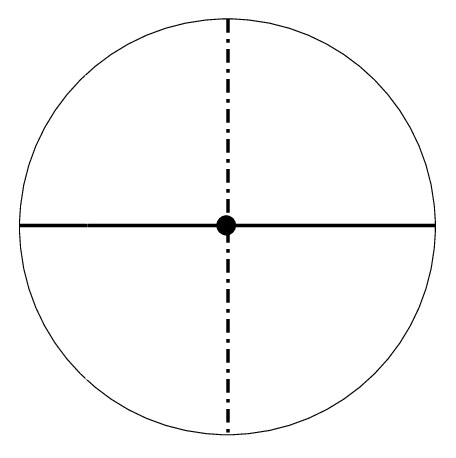}}\quad
\subfigure[Near $\mu_{2k_{2}-1}$]
{
\labellist
\small
\pinlabel* $y_{2}$ at 55 125
\pinlabel* $y'_{2}$ at 160 130
\endlabellist
\includegraphics[height = 36mm]{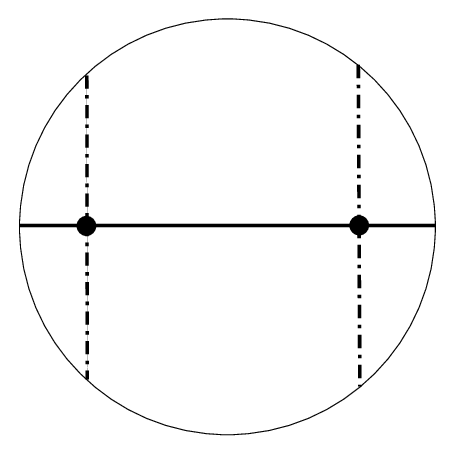}}
\caption[Heegaard diagrams covering Figure \ref{fig:red2a}]{Local pictures of the Heegaard diagram $\h_{n}$}
\label{fig:red2hd1}
\end{figure}

\begin{figure}[h!]
\centering
\subfigure[Near $\mu_{2n-3}$]
{
\labellist
\small
\pinlabel* $u'_{3}$ at 60 160
\pinlabel* $u'_{1}$ at 155 175
\pinlabel* $u'_{2}$ at 180 152
\pinlabel* $u_{2}$ at 45 70
\pinlabel* $u_{1}$ at 70 38
\pinlabel* $u_{3}$ at 165 67
\endlabellist
\includegraphics[height = 36mm]{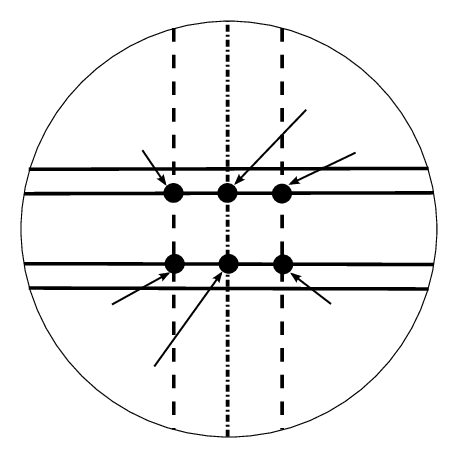}}\quad
\subfigure[Near $\mu_{2k_{1}-1}$]
{
\labellist
\small
\pinlabel* $v_{1}$ at 95 153
\pinlabel* $v'_{1}$ at 120 55
\endlabellist
\includegraphics[height = 36mm]{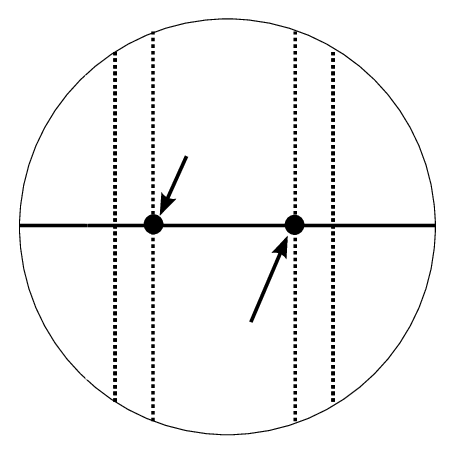}}\quad
\subfigure[Near $\mu_{2k_{2}-1}$]
{
\labellist
\small
\pinlabel* $v_{2}$ at 100 153
\pinlabel* $v'_{2}$ at 115 55
\endlabellist
\includegraphics[height = 36mm]{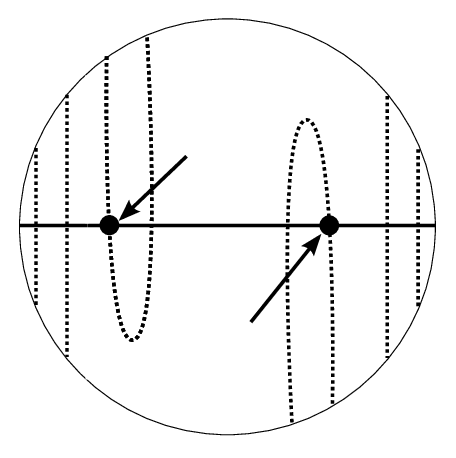}}
\caption[Heegaard diagrams covering Figure \ref{fig:red2b}]{Local pictures of the Heegaard diagram $\h'_{(n-1)}$}
\label{fig:red2hd2}
\end{figure}

%
%
%

\begin{figure}[h!]
\centering
\begin{minipage}[c]{.50\linewidth}
\labellist
\small
\pinlabel* $u_{3}$ at 195 80
\pinlabel* $u_{1}$ at 210 120
\pinlabel* $u_{2}$ at 310 140
\endlabellist
\includegraphics[height = 30mm]{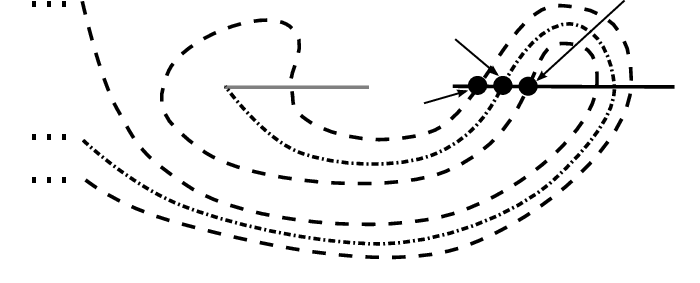}
\end{minipage}
\begin{minipage}[c]{.45\linewidth}
\caption{The result of an isotopy on Figure \ref{fig:red2b1}}
\label{fig:red2biso}
\end{minipage}
\end{figure}

Figure \ref{fig:red2biso} shows the result of the isotopy making $\alpha'_{n}$ horizontal.  One can verify that by examining Figures \ref{fig:red2a} and \ref{fig:red2biso} that for $1 \leq i \leq 2$, $\red{T}( \bz u_{1} v_{i}) = \red{T}( \bz x_{1} y_{i})$ and $\red{T}( \bz u_{2} v_{i}) = \red{T}( \bz u_{3} v_{i})  = \red{T}( \bz x_{2} y_{i})$.  Additionally,
\begin{align*}
\left(\red{P}^{*} - \red{Q}^{*}\right)(u_{2}) = \left(\red{P}^{*} - \red{Q}^{*}\right)(u_{3}) &= \left(\red{P}^{*} - \red{Q}^{*}\right)(x_{2}) + 1,
\quad\left(\red{P}^{*} - \red{Q}^{*}\right)(u_{1}) = \left(\red{P}^{*} - \red{Q}^{*}\right)(x_{1}) + 1\\
\text{and} \quad \left(\red{P}^{*} - \red{Q}^{*}\right)(v_{i}) &= \left(\red{P}^{*} - \red{Q}^{*}\right)(y_{i}) + 1\quad \text{for} \quad 1 \leq i \leq 2.
\end{align*}
So, $g$ preserves $\red{R}$ in this case as well.

\subsection{Some technical facts about filtered complexes of vector spaces}\label{sec:cancel}

In light of the ``cancellation rule'' provided by Proposition \ref{prop:cancel}, Proposition \ref{prop:RredR} indicates that the filtered chain homotopy type of the reduced filtered complex determines that of the unreduced filtered complex, and vice versa - provided that we work with coefficients in $\mathbb{F}$, a field.   Together with Theorem 1.0.1 from \cite{et:R}, this immediately implies Theorem \ref{thm:redRthm} above.  The goal of this section is to prove Proposition \ref{prop:cancel}.  First we recall some definitions related to filtrations and establish some notation.

Let $\mathbb{F}$ be a field and let $V$ be a vector space over $\mathbb{F}$ equipped with a $\mathbb{Z}$-filtration $\cdots \subseteq F_{i-1}V \subseteq F_{i}V \subseteq F_{i+1}V \subseteq \cdots$ so that $\bigcup F_{i}V = V$ and $\bigcap F_{i}V = \emptyset$.  Then recall that there is an associated graded vector space $\text{gr}(V)$ defined by
$$ \text{gr}(V) := \bigoplus_{i = -\infty}^{\infty} \text{gr}_i(V) \quad \text{where} \quad
\text{gr}_i(V) :=F_{i}V / F_{i-1}V$$
Given some $x \in V$, the \emph{filtration level} of $x$ is the integer $f(x) := \text{inf} \left\{ k \in \mathbb{Z} \big| x \in F_{k}V  \right\}$.  For each $k \in \mathbb{Z}$, there is a natural map $\iota_k:F_{k}V \rightarrow \text{gr}(V)$ given by composing the projection $F_{k}V \rightarrow F_{k}V/F_{k-1}V$ with the inclusion $F_{k}V / F_{k-1}V \rightarrow \text{gr}(V)$.  Then we can define a function $[\cdot]:V \rightarrow \text{gr}(V)$ via $[x]:=\iota_{f(x)} (x)$.

Furthermore, if $V = \oplus_{i \in \mathbb{Z}} V_i$ is a graded vector space, let $g(x) \in \mathbb{Z}$ denote the grading of a homogeneous element $x \in V$.  Recall that a filtered chain complex $(C, \partial)$ over a field $\mathbb{F}$ is a graded vector space $C$ over $\mathbb{F}$ together with a differential $\partial: C \rightarrow C$ such that $\partial^2 = 0$ and $\partial(F_k C_i) \subset F_k C_{i-1}$ for all $i, k \in \mathbb{Z}$.

\begin{df}
Let $C$ be a finite-dimensional $\mathbb{Z}$-graded, $\mathbb{Z}$-filtered chain complex of vector spaces over a field $\mathbb{F}$.  We say that $C$ is \emph{reduced} if for every $x \in C$, $f(\partial(x)) < f(x)$, i.e. the differential strictly lowers the filtration level.
\end{df}

\begin{df}
Let $V$ be a finite-dimensional $\mathbb{Z}$-filtered vector space over a field $\mathbb{F}$.
\begin{enumerate}[(i)]
\item A set $\left\{ x_1, \ldots, x_n \right\} \subset V$ is \emph{filtered linearly independent} if the set $\left\{ [x_1], \ldots, [x_n]\right\} \subset \text{gr}(V)$ is linearly independent in $\text{gr}(V)$.
\item A basis $\left\{ x_1, \ldots, x_n \right\}$ for $V$ is a \emph{filtered basis} if the set $\left\{ [x_1], \ldots, [x_n] \right\}$ is a basis for $\text{gr}(V)$.
\end{enumerate}
\end{df}

We visualize a filtered chain complex on the 1-dimensional lattice $\mathbb{Z}$ by choosing a filtered basis $\left\{x_1, \ldots x_n \right\}$ and plotting each generator $x_i$ at the position $f(x_i) \in \mathbb{Z}$.  The differential appears as a system of arrows, with one arrow pointing from $x_i$ to each generator appearing in the expression $\partial(x_i)$ (for each $i$).  We'll refer to these arrows as the \emph{components} of the differential and we'll say that a generator $x_i$ is \emph{incident} to some component if that component either emanates or terminates at $x_i$.

\begin{df}
Let $C$ be a finite-dimensional $\mathbb{Z}$-graded, $\mathbb{Z}$-filtered chain complex of vector spaces over a field $\mathbb{F}$.  Then a filtered basis $\left\{ x_1, \ldots, x_n \right\}$ for $C$ is \emph{simplified} if each $x_i$ is incident to at most one component of the differential.
\end{df}

We'll need several lemmas to prove Proposition \ref{prop:cancel}.  The following result is proved in \cite{lot}:

\begin{lem}[Proposition 11.52 from \cite{lot}]\label{lem:reduced}
Let $C$ be a finite-dimensional $\mathbb{Z}$-graded, $\mathbb{Z}$-filtered chain complex of vector spaces over a field $\mathbb{F}$.  Then there exists a reduced  filtered complex $C'$ which is filtered chain homotopy equivalent to $C$.  Furthermore, $C'$ has a filtered basis that is simplified.
\end{lem}

The following two lemmas are proved at the end of this section.  The proof of Proposition \ref{prop:cancel} relies heavily on \ref{lem:ss}, whose proof in turn relies on Lemma \ref{lem:cancel}.

\begin{lem}\label{lem:cancel}
Fix a positive integer $k$.  Let $V$, $W_1$, and $W_2$ be finite-dimensional $\mathbb{Z}^k$-graded vector spaces over a field $\mathbb{F}$.
If $W_1 \otimes_{\mathbb{F}} V \cong W_2 \otimes_{\mathbb{F}} V$ as $\mathbb{Z}^k$-graded vector spaces, then in fact $W_1 \cong W_2$.
\end{lem}

The above fact is of interest in its own right as a cancellation tool for graded vector spaces.  It was presumably previously known, but the author has been unable to find a particular reference.

The following lemma will allow us to prove Proposition \ref{prop:cancel}:

\begin{lem}\label{lem:ss}
Let $C$ be a finite-dimensional $\mathbb{Z}$-graded, $\mathbb{Z}$-filtered chain complex of vector spaces over a field $\mathbb{F}$.  Then the filtered chain homotopy type of the complex $C$ is completely determined by the $\mathbb{Z}^2$-graded isomorphism types of the pages $\left\{ E^r \right\}_{r \geq 1}$ of the spectral sequence induced by $C$.
\end{lem}

\begin{proof}[Proof of Proposition \ref{prop:cancel}]
For each $i$ and for $k \geq 1$, let $E^{i, k}$ denote the $k^{th}$ page of the spectral sequence induced by the filtered complex $C_i$.
For $k \geq 1$, the $k^{th}$ page of the spectral sequence induced by the complex $C_i \otimes_{\mathbb{F}} V^{\mathbb{F}}$ is isomorphic as a $\mathbb{Z}^2$-graded group to $E^{i, k} \otimes_{\mathbb{F}} \left( \mathbb{F}(1,0) \oplus \mathbb{F}(0,0) \right)$, and the $k^{th}$ page of the spectral sequence induced by $C_i \otimes_{\mathbb{F}} W^{\mathbb{F}}$ is isomorphic to $E^{i, k} \otimes_{\mathbb{F}} \left( \mathbb{F}(0,0)\right)^2$.  So, regardless of whether $X = V^{\mathbb{F}}$ or $X = W^{\mathbb{F}}$, Lemma \ref{lem:cancel} implies that for each $k \geq 1$, $E^{1,k} \cong E^{2,k}$ as $\mathbb{Z}^2$-graded vector spaces.  Now Lemma \ref{lem:ss} implies that $C_1$ and $C_2$ have the same filtered chain homotopy type. 
\end{proof}

\begin{proof}[Proof of Lemma \ref{lem:cancel}]
For concreteness, let $V = \displaystyle \bigoplus_{j = 1}^{n} \mathbb{F} \left( \bda_j \right)$, where $\bda_j \in \mathbb{Z}^k$.  After relabeling if necessary, assume that $m \in \{ 1, 2, \ldots, n \}$ satisfies $\bda_j = \bda_1$ for $j \leq m$ and $\bda_{j} \neq \bda_1$ for $j > m$.

For $i \in \{ 1,2 \}$, the $\mathbb{Z}^k$-graded isomorphism type of $W_i$ is determined by the function $f_i:\mathbb{Z}^k \rightarrow \mathbb{Z}_{\geq 0}$, where $f_i(\bx): = \text{dim}_{\mathbb{F}}\left( \left( W_i\right)_{\bx} \right)$.  Define an analogous function $g_i$ associated to $W_i \otimes_{\mathbb{F}} V$.  Notice that for each $\bx \in \mathbb{Z}^k$,
\begin{gather}\label{bigrading}
\begin{aligned}
g_i(\bx) &= \sum_{j = 1}^n f_i(\bx - \bda_j) \quad \text{and so}\\
mf_i(\bx) &= g_i(\bx + \bda_1) - \sum_{j = m+1}^{n} f_i\left( \bx + \bda_1 - \bda_j \right)
\end{aligned}
\end{gather}
Now Equation \ref{bigrading} implies that $f_2(\bx)- f_1(\bx)$ only depends on the equivalence class of $\bx$ in the quotient $\mathbb{Z}^k/H$, where $H:= \text{Span} \left\{ \bda_{1} - \bda_{j} | m < j \leq n  \right\}$.

Since $W_1$ and $W_2$ are finite-dimensional, $f_2 - f_1 = 0$ outside of a compact set in $\mathbb{Z}^k$.  Therefore, in the case that $m < n$ (i.e. not all of the $\bda_j$ are equal), $f_2 - f_1 \equiv 0$.

If $m = n$, Equation \ref{bigrading} implies that $0 \equiv g_2 - g_1 \equiv n(f_2 - f_1)$ and so $f_1 \equiv f_2$ since $\mathbb{Z}$ is torsion-free.
\end{proof}

Prior to proving Lemma \ref{lem:ss}, we'll define some particular filtered chain complexes.  Given integers $s,t \in \mathbb{Z}$, let $V(s,t)$ denote the $\mathbb{Z}$-filtered, $\mathbb{Z}$-graded chain complex of vector spaces given by $V(s,t) := \left(\F{x} , d \equiv 0 \right)$ where $x$ has grading $g(x) = s$ and filtration level $f(x) = t$.  For integers $s,t$ and positive integer $\delta$, let $W(s,t,\delta)$ denote the $\mathbb{Z}$-filtered, $\mathbb{Z}$-graded chain complex of vector spaces given by $ W(s,t,\delta) := \left( \F{x} \oplus \F{y}, d \right)$ with $(g,f)(x) = (s,t)$, $(g,f)(y) = (s-1,t-\delta)$, $d(x) = y$, and $d(y)=0$.

\begin{proof}[Proof of Lemma \ref{lem:ss}]
By Lemma \ref{lem:reduced} above, $C$ is filtered chain homotopy equivalent to some reduced filtered chain complex $C'$, and $C'$ has a simplified basis.  Then we in fact have that
 \begin{equation}\label{eqn:simp}
 C' = \left( \bigoplus_{i = 1}^{n} V(s_i, t_i) \right) \oplus \left( \bigoplus_{j=1}^{m} \bigoplus_{k = 1}^{l_j} W(a_{j,k}, b_{j,k}, \delta_j) \right)
 \end{equation}
 for some positive integers $m$, $n$, $\delta_j$ and $l_j$ (for each $1 \leq j \leq m)$, some integers $s_i$, $t_i$ for each $1 \leq i \leq n$, and some integers $a_{j,k}$, $b_{j,k}$ for each $1 \leq j \leq m$ and $1 \leq k \leq l_j$.  It is assumed that $\delta_1 < \delta_2 < \ldots < \delta_m$, and recall that $\delta_1 > 0$ because $C'$ is reduced.  Note that the summands $V(s_i, t_i)$ in Equation \ref{eqn:simp} are generated by elements in the simplified filtered basis for $C'$ which aren't incident to any components of the differential, and the components of the summands $W(a_{j,k}, b_{j,k}, \delta_j)$ are generated by elements each incident to exactly one component of the differential.
 
Let $E^r$ denote the pages of the spectral sequence induced by the filtered complex $C'$.  First notice that since $C'$ is reduced, $E^0 \cong E^1$ as $\mathbb{Z}^2$-graded vector spaces.  Furthermore, we have that
$$ E^{\infty} \cong \bigoplus_{i = 1}^n \mathbb{F}(t_i, s_i - t_i)$$
Recall that for each $r$, $E^{r+1}$ is isomorphic to a $\mathbb{Z}^2$-graded subspace of $E^r$, and the quotients $E^{r+1}/E^r$ are nonzero exactly when $r = \delta_j$ for some $j$.  Therefore, the higher pages determine the numbers $\delta_j$.  Now for each $j$ with $1 \leq j \leq m$,
\begin{gather}
\begin{aligned}\label{eqn:alg1}
\frac{E^{\delta_j + 1}}{E^{\delta_j}} &\cong
\bigoplus_{k = 1}^{l_j} \left( \mathbb{F}(b_{j,k},a_{j,k} - b_{j,k}) \oplus \mathbb{F}(b_{j,k} - \delta_j, a_{j,k} -b_{j,k} + \delta_j-1) \right)\\
&\cong \left( \bigoplus_{k=1}^{l_j} \mathbb{F}(b_{j,k}, a_{j,k} - b_{j,k}) \right)
\otimes_{\mathbb{F}} \left( \mathbb{F}(0,0) \oplus \mathbb{F}(-\delta_j, \delta_j - 1) \right)
\end{aligned}
\end{gather}

Now since we already know $\delta_j$, then by Lemma \ref{lem:cancel}, the $\mathbb{Z}^2$-graded isomorphism type of the last expression in Equation \ref{eqn:alg1} determines the $\mathbb{Z}^2$-graded isomorphism type of
$$\left( \bigoplus_{k=1}^{l_j} \mathbb{F}(b_{j,k}, a_{j,k} - b_{j,k}) \right)$$
which in turn determines the numbers $b_{j,k}$ and $a_{j,k}$ for all $k$.  So, we now know the filtered chain isomorphism type of the filtered chain complex $C'$.
\end{proof}

\section{APPLICATIONS AND COMPUTATIONS}

\subsection{Operations on knots}

\subsubsection{Connected sums}\label{sec:sum}

Recall that \OS showed in \cite{os:disk2} that the relatively-graded Heegaard Floer chain complexes satisfy a K\"unneth-type relationship under connected sums of $3$-manifolds.  We state the result for $\widehat{CF}$:
\begin{thm}[Proposition 6.1 from \cite{os:disk2}]\label{thm:ossum}
Let $M_{1}$ and $M_{2}$ be oriented $3$-manifolds with $\text{Spin}^{c}$ structures $\mathfrak{s_{i}} \in \text{Spin}^{c}(M_{i})$ for $i= 1,2$.  Then 
\begin{equation*}
\widehat{CF}(M_{1} \# M_{2} , \mathfrak{s}_{1} \# \mathfrak{s}_{2}) \cong
	\widehat{CF}(M_{1}, \mathfrak{s_{1}})
		\otimes_{\mathbb{Z}}
	\widehat{CF}(M_{2}, \mathfrak{s_{2}}).
\end{equation*}
\end{thm}

Recall also that for $M$ a rational homology $3$-sphere with $\mathfrak{s} \in \text{Spin}^{c}(M)$, Ozsv\'ath  and Szab\'o defined in \cite{os:abs} the correction term $d(M,\mathfrak{s}) \in \mathbb{Q}$ to be the minimal absolute grading $\tld{gr}$ of any non-torsion element in the image of $HF^{\infty}(M,\mathfrak{s})$ in $HF^{+}(M,\mathfrak{s})$.  It was shown in \cite{os:abs} that
\begin{equation*}
d(M_{1} \# M_{2}, \mathfrak{s}_{1} \# \mathfrak{s}_{2}) = d(M_{1},\mathfrak{s}_{1}) + d(M_{2},\mathfrak{s}_{2}).
\end{equation*}

It follows that when $M_{1}$ and $M_{2}$ are rational homology $3$-spheres, the $\widehat{CF}$ complexes satisfy a K\"unneth formula as absolutely-graded chain complexes with grading $\tld{gr}$, i.e.
\begin{equation}\label{eq:grkun}
\widehat{CF}_{\tld{gr} = k}(M_{1} \# M_{2}, \mathfrak{s}_{1} \# \mathfrak{s}_{2})
\cong \displaystyle \bigoplus_{i + j = k}
\bigg( \widehat{CF}_{\tld{gr} = i}(M_{1}, \mathfrak{s}_{1}) \otimes_{\mathbb{Z}} \widehat{CF}_{\tld{gr} = j}(M_{2}, \mathfrak{s}_{2}) \bigg).
\end{equation}
As stated in Theorem \ref{thm:sum}, when $M_{i} = \DBC{K_{i}}$ for $i=1, 2$, then the complexes for $M_{1}$, $M_{2}$, and $\DBC{K_{1} \# K_{2}} \cong M_{1} \# M_{2}$ satisfy a K\"unneth-type relationship as filtered complexes up to filtered chain homotopy type.
\begin{proof}[Proof of Theorem \ref{thm:sum}]
Let $K_{1}$ and $K_{2}$ be knots in $S^{3}$.  Then we can find braids $b_{1} \in \B{2m-1}$ and $b_{2} \in \B{2n-2}$ such that the plat closure of $b'_{1} = \left( b_{1} \times \sigma_{1}^{-1} \times 1\right) \in \B{2m+2}$ is a diagram $D_{1}$ for $K_{1}$ and the plat closure of $b'_{2} = (1\times b_{2} \times 1) \in \B{2n}$ is a diagram $D_{2}$ for $K_{2}$.   Then the plat closure of $b = \left(b_{1} \times \sigma_{1}^{-1} \times b_{2} \times 1 \right)\in \B{2n + 2m}$ is a diagram $D$ for $K = K_{1} \# K_{2}$ (where we view $\sigma_{1} \in \B{2}$).  These plat closures are illustrated in Figure \ref{fig:csplats}.

\begin{figure}[h]
\centering
\subfigure[$D_{1}$]
{
\labellist
\small
\pinlabel* $b_{1}$ at 70 140
\endlabellist
\includegraphics[height = 30mm]{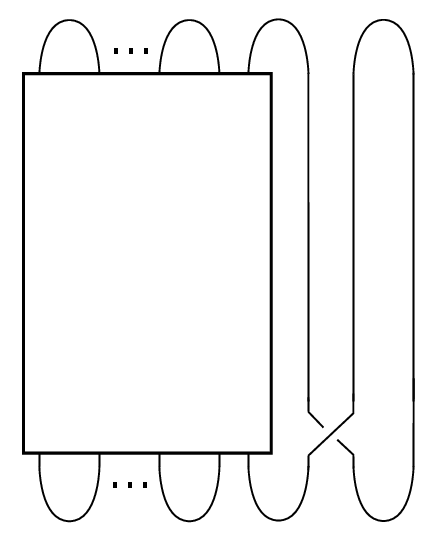}}\qquad
\subfigure[$D_{2}$]
{
\labellist
\small
\pinlabel* $b_{2}$ at 90 140
\endlabellist
\includegraphics[height = 30mm]{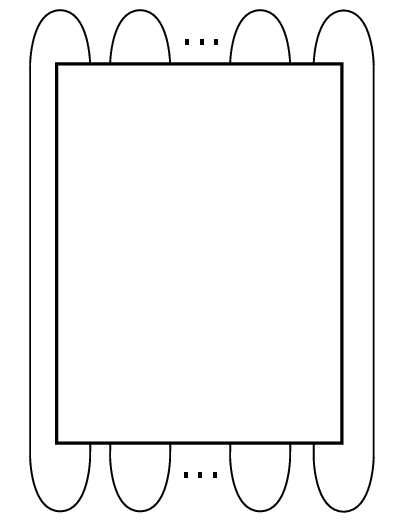}}\qquad
\subfigure[$D$]
{
\labellist
\small
\pinlabel* $b_{1}$ at 70 150
\pinlabel* $b_{2}$ at 240 150
\endlabellist
\includegraphics[height = 30mm]{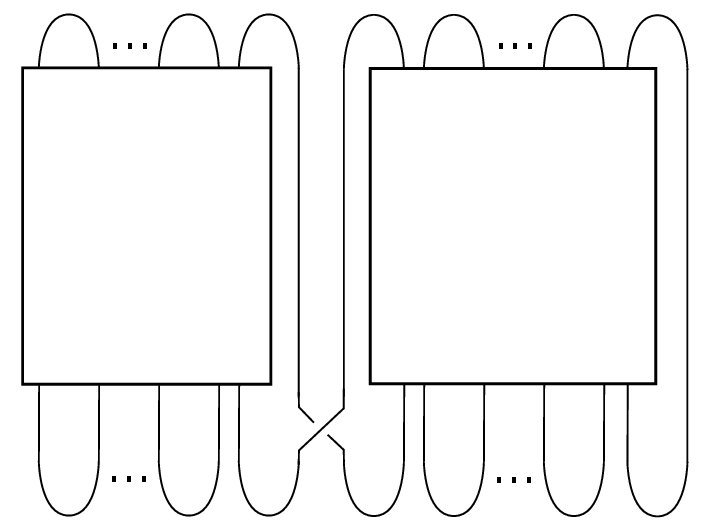}}
\caption[Plat closures associated to a connected sum]{Plat closures of $b'_{1} \in \B{2m+2}$, $b'_{2} \in \B{2n}$, and $b \in \B{2m + 2n}$}
\label{fig:csplats}
\end{figure}

\begin{figure}[h]
\centering
\subfigure[The fork diagram $\mathcal{F}_{1}$ with the arcs $\alpha_{m+1}, \beta_{m+1}$ grayed]
{
\labellist
\small
\pinlabel* $y$ at 50 40
\endlabellist
\includegraphics[height = 20mm]{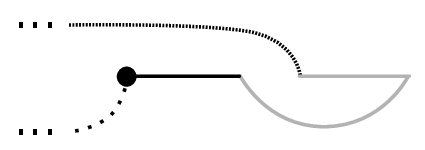}}\qquad
\subfigure[The fork diagram $\mathcal{F}_{2}$ with the arcs $\alpha_{n}, \beta_{n}$ grayed]
{
\labellist
\small
\pinlabel* $z_{1}$ at 10 22
\pinlabel* $z_{2}$ at 75 35
\endlabellist
\includegraphics[height = 20mm]{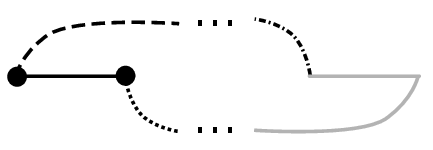}}
\caption{Local pictures of fork diagrams for summand knots $K_{i}$}
\label{fig:csflat1}
\end{figure}

\begin{figure}
\centering
\labellist
\small
\pinlabel* $y$ at 50 50
\pinlabel* $v_{1}$ at 155 50
\pinlabel* $v_{2}$ at 205 40
\endlabellist
\includegraphics[height = 24mm]{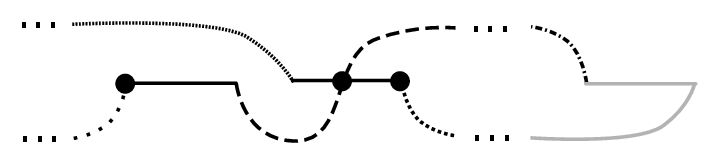}
\caption[Local picture of the fork diagram for $K_{1} \# K_{2}$]{The fork diagram $\mathcal{F}$ for $K_{1} \# K_{2}$ with the arcs $\alpha_{m+n}, \beta_{m+n}$ grayed.}
\label{fig:csflat2}
\end{figure}

Let $\mathcal{F}_{1}$, $\mathcal{F}_{2}$, and $\mathcal{F}$ denote the reducible fork diagrams induced by $b'_{1}$, $b'_{2}$, and $b$, respectively.  We choose the braids $b_{i}$ in such a way that the fork diagrams $\mathcal{F}_{i}$ have the local behavior indicated in Figure \ref{fig:csflat1};  this can always be achieved through stabilization.  The fork diagram $\mathcal{F}$ can be seen in Figure \ref{fig:csflat2}.

We'll compute the reduced gradings $\red{R}$ by omitting the pair $\alpha_{m}, \beta_{m}$ from $\mathcal{F}_{1}$, the pair $\alpha_{n}, \beta_{n}$ from $\mathcal{F}_{2}$, and the pair $\alpha_{m+n}, \beta_{m+n}$ from $\mathcal{F}$.

Figure \ref{fig:cshd} shows the Heegaard diagrams $\h_{1}$, $\h_{2}$, and $\h$ covering $\mathcal{F}_{1}$, $\mathcal{F}_{2}$, and $\mathcal{F}$, respectively.  Let the sets of attaching circles for these diagrams be denoted by $\{ \ba^{1}, \bb^{1}\}$, $\{ \ba^{2}, \bb^{2}\}$, and $\{ \ba, \bb\}$, respectively.

\begin{figure}[h]
\centering
\subfigure[$\h_{1}$]
{
\labellist
\small
\pinlabel* $y$ at 130 105
\pinlabel $a$ at 88 117
\pinlabel \rotatebox{180}{\reflectbox{$a$}} at 88 39
\pinlabel $b$ at 184 119
\pinlabel \rotatebox{180}{\reflectbox{$b$}} at 184 39
\endlabellist
\includegraphics[height = 36mm]{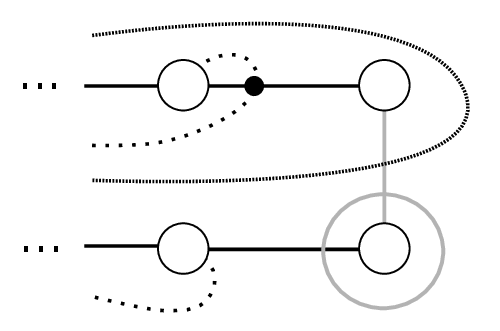}}\qquad
\subfigure[$\h_{2}$]
{
\labellist
\small
\pinlabel* $z_{1}$ at 110 25
\pinlabel* $z_{2}$ at 120 130
\pinlabel $c$ at 52 117
\pinlabel \rotatebox{180}{\reflectbox{$c$}} at 52 39
\pinlabel $d$ at 188 119
\pinlabel \rotatebox{180}{\reflectbox{$d$}} at 188 39
\endlabellist
\includegraphics[height = 36mm]{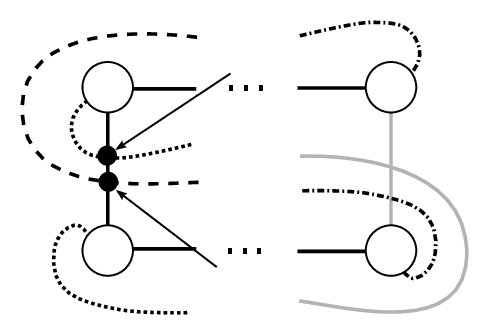}}\\
\subfigure[$\h$]
{
\labellist
\small
\pinlabel* $y$ at 145 95
\pinlabel* $v_{1}$ at 230 143
\pinlabel* $v'_{1}$ at 135 55
\pinlabel* $v_{2}$ at 230 20
\pinlabel $a$ at 88 117
\pinlabel \rotatebox{180}{\reflectbox{$a$}} at 88 39
\pinlabel $b$ at 184 119
\pinlabel \rotatebox{180}{\reflectbox{$b$}} at 184 39
\pinlabel $c$ at 280 117
\pinlabel \rotatebox{180}{\reflectbox{$c$}} at 280 39
\pinlabel $d$ at 417 119
\pinlabel \rotatebox{180}{\reflectbox{$d$}} at 417 39
\endlabellist
\includegraphics[height = 36mm]{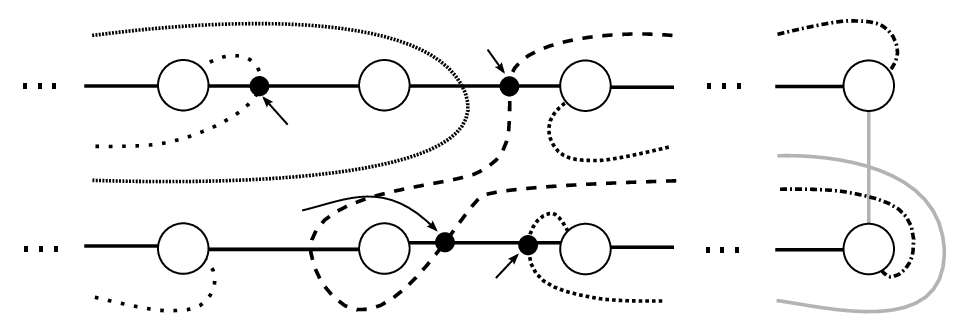}}
\caption{Heegaard diagrams depicting local connected sum behavior}
\label{fig:cshd}
\end{figure}
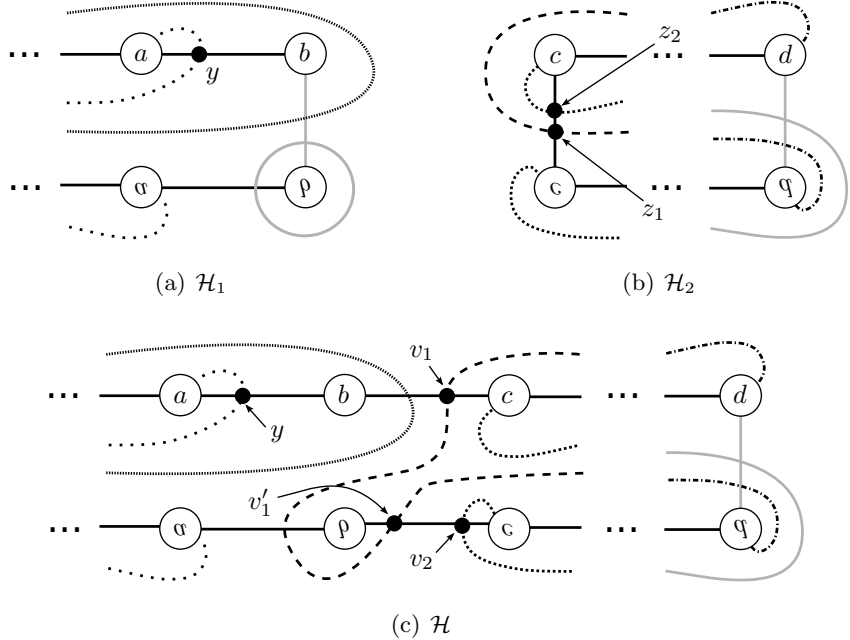

Recall that if the connected sum region in each Heegaard diagram is near the basepoint, then $\h_{1}$, $\h_{2}$, and $\h_{\#} = \h_{1} \# \h_{2}$ are Heegaard diagrams exhibiting the correspondences in Theorem \ref{thm:ossum} and Equation \ref{eq:grkun}.  The sets of attaching circles for $\h_{\#}$ are $\{ \ba^{\#}, \bb^{\#}\} = \{ \ba^{1} \sqcup \ba^{2}, \bb^{1} \sqcup \bb^{2} \}$.

The set $\ba$ can be obtained from the set $\ba^\#$ via a sequence of $m$ pointed handleslides
$$ \ba^{\#} = \bat_{0} \mapsto \bat^1 \mapsto \ldots \mapsto \bat^m = \ba,$$
as depicted in Figure \ref{fig:cshs} for $m = 3$.  The set $\bb$ can be obtained from the set $\bb^{\#}$ by an analogous sequence of $m$ pointed handleslides.  We then study the sequence of pointed Heegaard diagrams
\begin{equation}\label{eqn:sequence}
\h^{0,0} \mapsto \h^{1,0} \mapsto \ldots \mapsto \h^{m-1,0} \mapsto \h^{m-1,1} \mapsto \ldots \mapsto \h^{m-1,m-1} \mapsto \h^{m,m-1} \mapsto \h^{m,m}.
\end{equation}
where $\h^{i,j} = \left( \Sigma; \bat^i; \bbt^j; +\infty \right).$  Now $\h^{0,0} = \h_{\#}$ and $\h^{m,m} = \h$ are automatically admissible because they cover fork diagrams.  The first $2m-2$ handleslides in the above sequence are of the form shown in Figure \ref{fig:hsad} and thus preserve admissibility by Lemma \ref{lem:hsad}.  Although we have exhibited a single handleslide connecting $\h^{m,m-1}$ and $\h^{m,m}$, indeed $\h^{m,m-1}$ can alternately be obtained from $\h^{m,m}$ by a sequence of handleslides and isotopies like those in Lemmas \ref{lem:hsad} and \ref{lem:isoad}.  So, all intermediate diagrams in the sequence in Equation \ref{eqn:sequence} are admissible.
\begin{figure}
\centering
\subfigure[$\ba^{\#} = \bat^{0} \mapsto \bat^{1}$]{
\labellist 
\small
\pinlabel* {$a$} at 23 107
\pinlabel* {$b$} at 91 107
\pinlabel* {$c$} at 157 107
\pinlabel* {$d$} at 223 107
\pinlabel* {$e$} at 290 107
\pinlabel* {\rotatebox{180}{\reflectbox{$a$}}} at 23 30
\pinlabel* {\rotatebox{180}{\reflectbox{$b$}}} at 91 30
\pinlabel* {\rotatebox{180}{\reflectbox{$c$}}} at 157 30
\pinlabel* {\rotatebox{180}{\reflectbox{$d$}}} at 223 30
\pinlabel* {\rotatebox{180}{\reflectbox{$e$}}} at 290 30
\endlabellist
\includegraphics[height = 26mm]{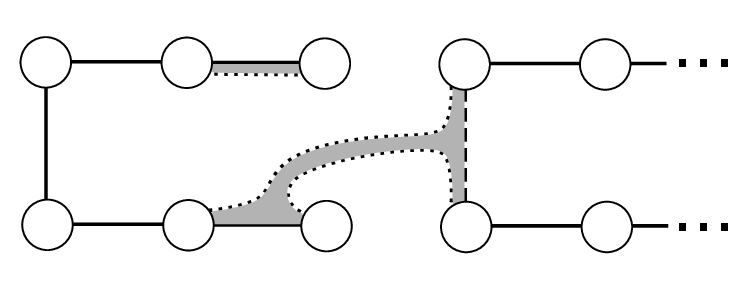}}\quad
\subfigure[$\bat^{1} \mapsto \bat^{2}$]{
\labellist 
\small
\pinlabel* {$a$} at 23 107
\pinlabel* {$b$} at 91 107
\pinlabel* {$c$} at 157 107
\pinlabel* {$d$} at 223 107
\pinlabel* {$e$} at 290 107
\pinlabel* {\rotatebox{180}{\reflectbox{$a$}}} at 23 30
\pinlabel* {\rotatebox{180}{\reflectbox{$b$}}} at 91 30
\pinlabel* {\rotatebox{180}{\reflectbox{$c$}}} at 157 30
\pinlabel* {\rotatebox{180}{\reflectbox{$d$}}} at 223 30
\pinlabel* {\rotatebox{180}{\reflectbox{$e$}}} at 290 30
\endlabellist
\includegraphics[height = 26mm]{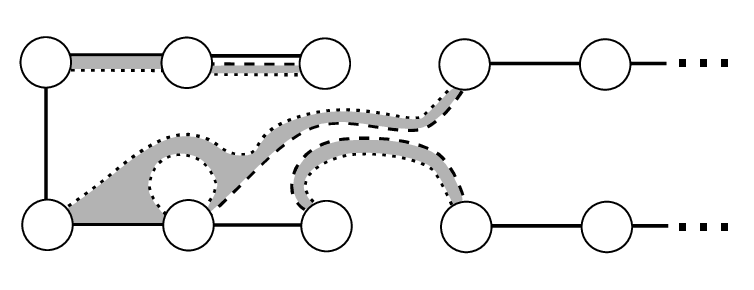}\label{fig:cshs2}}\quad
\subfigure[$\bat^{2} \mapsto \bat^{3} = \ba$]{
\labellist 
\small
\pinlabel* {$a$} at 23 107
\pinlabel* {$b$} at 91 107
\pinlabel* {$c$} at 157 107
\pinlabel* {$d$} at 223 107
\pinlabel* {$e$} at 290 107
\pinlabel* {\rotatebox{180}{\reflectbox{$a$}}} at 23 30
\pinlabel* {\rotatebox{180}{\reflectbox{$b$}}} at 91 30
\pinlabel* {\rotatebox{180}{\reflectbox{$c$}}} at 157 30
\pinlabel* {\rotatebox{180}{\reflectbox{$d$}}} at 223 30
\pinlabel* {\rotatebox{180}{\reflectbox{$e$}}} at 290 30
\endlabellist 
\includegraphics[height = 26mm]{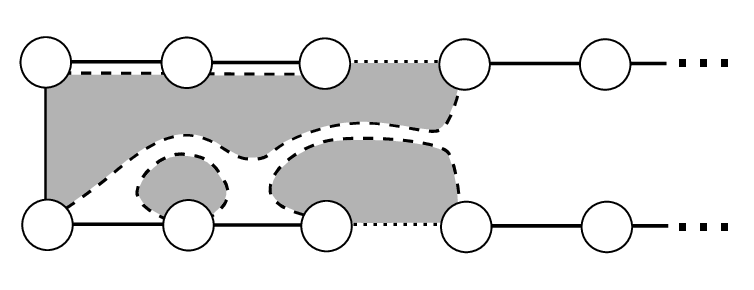}}
\caption[$\alpha$-Handleslides appearing in the proof of Theorem \ref{thm:sum}]{Handleslides connecting the sets $\ba^{\#}$ and $\ba$.  Notice that the first two handleslides are of the form shown in Figure \ref{fig:hsad}.  In each picture, the old curve is dashed, the new curve is dotted, unaffected curves are solid, and the pair of pants is shaded.
\label{fig:cshs}}
\end{figure}
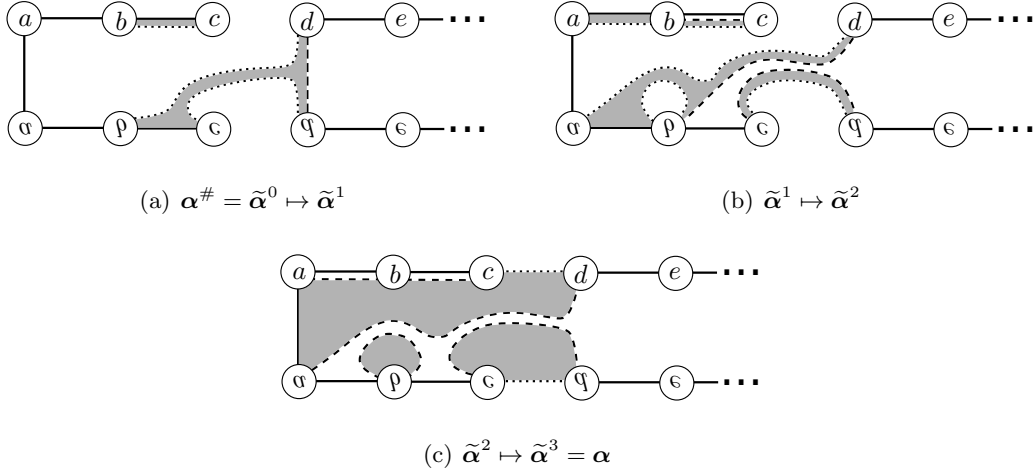

For the first $2m-2$ handleslides, we can construct triangle injections whose 3-gons have components like those in Figures \ref{fig:isotri} and \ref{fig:hstri}; denote by $\tld{g}:\tor{\bat^0} \cap \tor{\bbt^0} \rightarrow \tor{\bat^{m-1}} \cap \tor{\bbt^{m-1}}$ their composition.  As in Section \ref{sec:redtri}, we define triangle injections
$$g_{\alpha}: \tor{\bat^m} \cap \tor{\bbt^{m-1}} \rightarrow \tor{\bat^{m-1}}\cap\tor{\bbt^{m-1}} \quad
\text{and} \quad g_{\beta}: \tor{\bat^{m}} \cap \tor{\bbt^{m}} \rightarrow \tor{\bat^m} \cap \tor{\bbt^{m-1}}$$
using the 3-gons $\psi_{\alpha}^{\pm}$ and $\psi_{\beta}^{\pm}$ for various generators, whose components are found in Figures \ref{fig:CStriA} and \ref{fig:CStriB} (components not shown are like those in Figure \ref{fig:isotri}).

\begin{figure}[h!]
\centering
\subfigure[$\psi_{\alpha}^{+}$]
{
\includegraphics[height = 24mm]{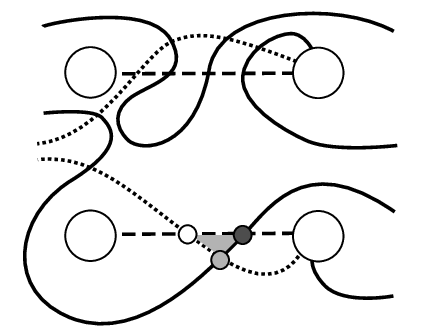}}\quad
\subfigure[$\psi_{\alpha}^{+}$]
{
\labellist 
\small
\pinlabel* {$v_2$} at 115 70
\endlabellist
\includegraphics[height = 24mm]{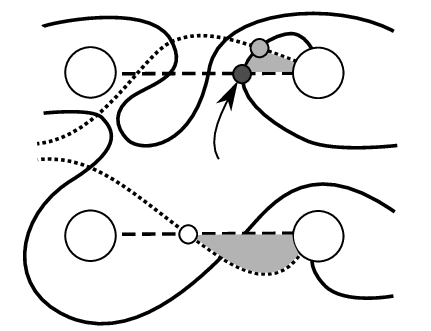}}\quad
\subfigure[$\psi_{\alpha}^{+}$]
{
\includegraphics[height = 24mm]{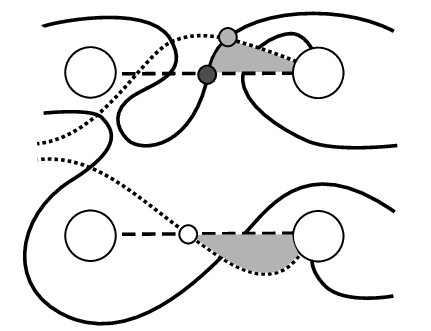}}\quad
\subfigure[$\psi_{\alpha}^{+}$]
{
\includegraphics[height = 24mm]{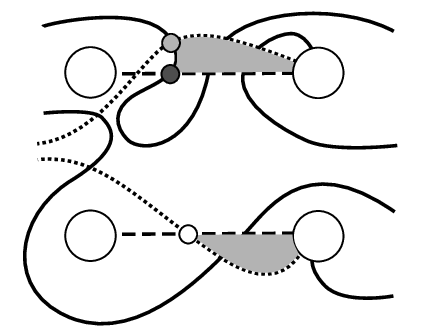}}\\
\subfigure[$\psi_{\alpha}^{-}$]
{
\includegraphics[height = 24mm]{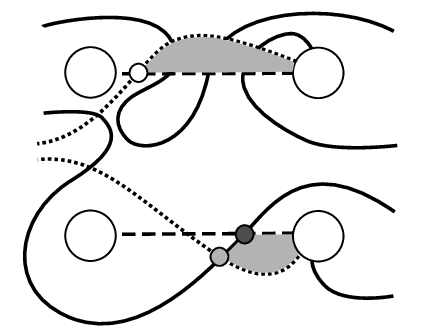}}\quad
\subfigure[$\psi_{\alpha}^{-}$]
{
\labellist 
\small
\pinlabel* {$v_2$} at 115 75
\endlabellist
\includegraphics[height = 24mm]{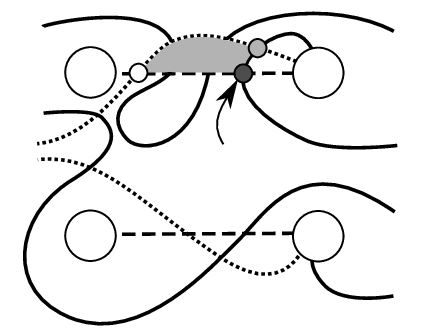}}\quad
\subfigure[$\psi_{\alpha}^{-}$]
{
\includegraphics[height = 24mm]{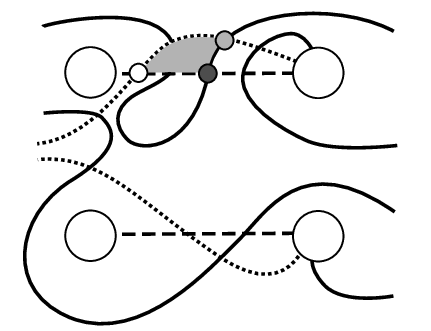}}\quad
\subfigure[$\psi_{\alpha}^{-}$]
{
\includegraphics[height = 24mm]{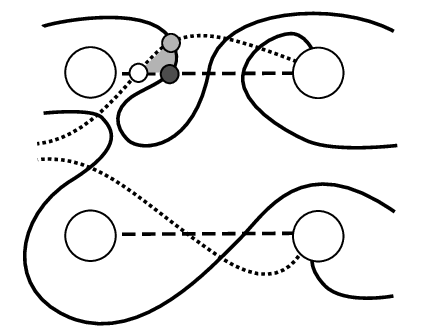}}
\caption[Components of domains of 3-gons associated to $g_{\alpha}$ in the proof of Theorem \ref{thm:sum}]{Some components of domains for the various 3-gons associated to a $\alpha$-triangle injection.  In each figure, the $\bat^m$ are dashed, the $\bat^{m-1}$ are dotted, the $\bbt^{m-1}$ are solid, the dark grey dot is a component of some $\bx$, and the light grey dot is the corresponding component of $g_\alpha(\bx)$. \label{fig:CStriA}}
\end{figure}

\begin{figure}[h!]
\centering
\subfigure[$\psi_{\beta}^{+}$]
{
\includegraphics[height = 24mm]{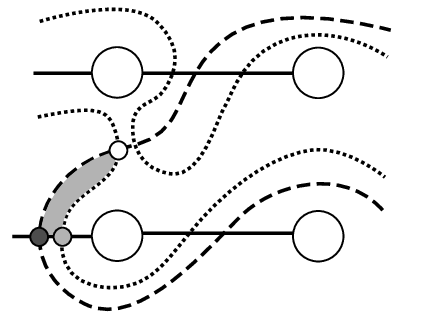}}\quad
\subfigure[$\psi_{\beta}^{+}$]
{
\labellist 
\small
\pinlabel* {$v'_1$} at 140 20
\endlabellist
\includegraphics[height = 24mm]{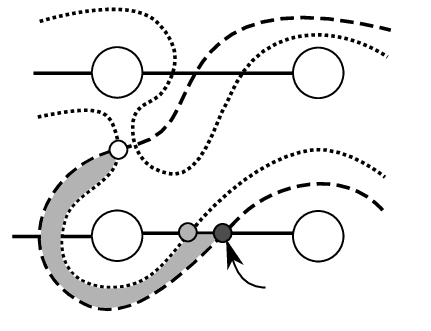}}\quad
\subfigure[$\psi_{\beta}^{+}$]
{
\labellist 
\small
\pinlabel* {$v_1$} at 130 100
\endlabellist
\includegraphics[height = 24mm]{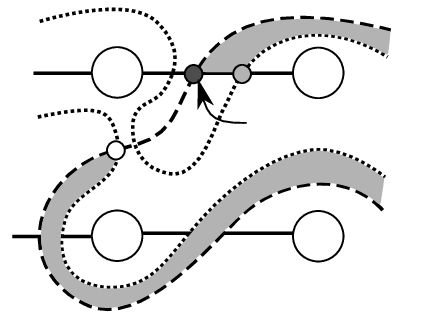}}\\
\subfigure[$\psi_{\beta}^{-}$]
{
\includegraphics[height = 24mm]{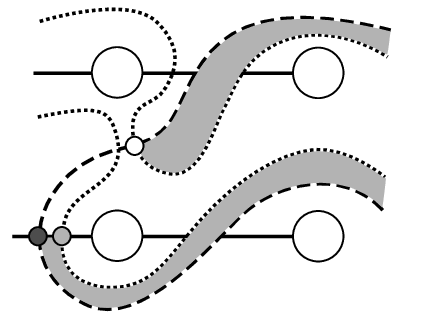}}\quad
\subfigure[$\psi_{\beta}^{-}$]
{
\labellist 
\small
\pinlabel* {$v'_1$} at 140 20
\endlabellist
\includegraphics[height = 24mm]{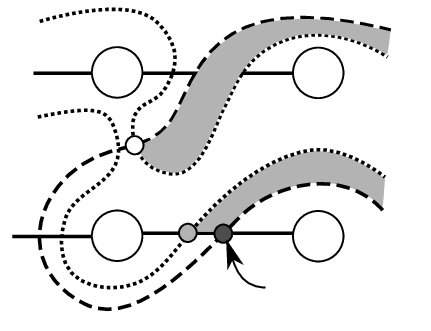}}\quad
\subfigure[$\psi_{\beta}^{-}$]
{
\labellist 
\small
\pinlabel* {$v_1$} at 135 95
\endlabellist
\includegraphics[height = 24mm]{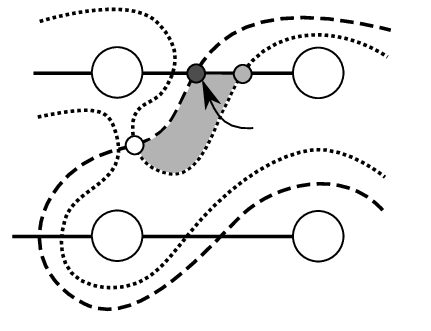}}
\caption[Components of domains of 3-gons associated to $g_{\beta}$ in the proof of Theorem \ref{thm:sum}]{Some components of domains for the various 3-gons associated to a $\beta$-triangle injection.  In each figure, the $\bbt^m$ are dashed, the $\bbt^{m-1}$ are dotted, the $\bat^{m}$ are solid, the dark grey dot is a component of some $\bx$, and the light grey dot is the corresponding component of $g_\beta(\bx)$.\label{fig:CStriB}}
\end{figure}

One can see that $\text{Im}\left( \tld{g} \right) \subseteq \text{Im} \left( g_{\alpha} \circ g_{\beta} \right)$, and so we can define
$$g = \left( g_{\alpha} \circ g_{\beta} \right)^{-1} \circ \tld{g} :\tor{\ba^{\#}}\cap\tor{\bb^{\#}} \rightarrow \tor{\ba}\cap\tor{\bb}.$$

Let  $\bx y$ be a generator in $\h_{1}$ and let $\bw z$ be a generator in $\h_{2}$ (where $\bx$ is a $\left( m-1 \right)$-tuple, $\bw$ is a $\left(n-2\right)$-tuple, $y \in \alpha^{1}_{m}$, and $z \in \alpha^{2}_{1}$).  The generators in $\h_{\#}$ are exactly of the form $\bx \bw yz$, and we set $\red{R}(\bx \bw yz) = \red{R}(\bx y) + \red{R}(\bw z)$.  It isn't hard to see that $g(\bx \bw y z_{i}) = \bx \bw y v_{i}$ for $i=1,2.$

By examining the fork diagrams in Figures \ref{fig:csflat1} and \ref{fig:csflat2}, one can verify that indeed $\tld{\red{R}}(\bx \bw y v_{i}) = \tld{\red{R}}(\bx y) + \tld{\red{R}}(\bw z_{i})$ for $i=1,2$.  Further, notice that $e(b) = e(b'_{1}) + e(b'_{2})$ and $w(D) = w(D_{1}) + w(D_{2})$.  Therefore,
\begin{align*}
s_{\red{R}}(b,D) 
&= \frac{e(b) - w(D) - 2(m + n - 2)}{4} \\
&= \frac{e(b'_{1}) - w(D_{1}) - 2(m-1)}{4} + \frac{e(b'_{2}) - w(D_{2}) - 2(n-1)}{4}\\
&=s_{\red{R}}(b_{1},D_{1}) + s_{\red{R}}(b_{2},D_{2}).
\end{align*}
\end{proof}

\subsubsection{Knot mirrors}\label{sec:mirror}

Given a braid word $b = \sigma_{i_{1}}^{k_{1}} \ldots \sigma_{i_{m}}^{k_{m}} \in \B{2n}$, let $-b$ denote the braid word
\begin{equation*}
-b = \sigma_{2n-i_{1}}^{-k_{1}} \ldots \sigma_{2n-i_{m}}^{-k_{m}} \in \B{2n}.
\end{equation*}
Notice that if the plat closure of $b$ is $K$, then the plat closure of $-b$ is $-K$, the mirror of $K$.

\begin{proof}[Proof of Theorem \ref{thm:mirror}]
Let $b \in \B{2n}$ be a braid whose closure is the knot $K$, and let $\h^{\pm}$ denote the admissible Heegaard diagram for $\DBC{\pm K}$ induced by $\pm b$.  Let $\mathfrak{s} \in \text{Spin}^{c}(\DBC{K})$.  Then by an argument analogous to that in the proof of Lemma 33 in \cite{et:R}, the natural isomorphism
$$ \Phi: \widehat{CF}_{*}\left( \h^{+}, \mathfrak{s} \right) \rightarrow \widehat{CF}^{-*}\left( \h^{-}, \mathfrak{s} \right)$$
is filtered when one equips the complexes with the filtrations $\red{R}$ and $\red{R}^{*}$, respectively.  In addition, if one equips $\widehat{CF}\left( \h^{-}, \mathfrak{s} \right)$ with an absolute grading $g$ on its generators given by $g(\bx^{*}) = -\tld{gr}(\bx)$, then $\Phi$ is graded with respect to $\tld{gr}$ and $g$.  So, $\Phi$ and $\Phi^{-1}$ are filtered with respect to the filtrations $\red{\rho}$ and $\red{\rho}^{*}$.
\end{proof}

\subsection{Some $\mathbb{Q}$-valued knot invariants}\label{sec:r}

The filtration $\red{R}$ gives rise to another more concise knot invariant.  Let $K$ be a knot which is the closure of a braid $b$, let $\h$ be the Heegaard diagram for $\DBC{K}$ induced by $b$, and let $\mathfrak{s} \in \text{Spin}^{c}(\DBC{K})$.  Let $\mathcal{F}(b,r, \mathfrak{s}) \subset \widehat{CF}(\h, \mathfrak{s})$ denote the subcomplex generated by intersection points whose filtration level is at most $r \in \mathbb{Q}$.  For each $\mathfrak{s} \in \text{Spin}^{c}(\DBC{K})$ and each $r \in \mathbb{Q}$, one obtains the homomorphism induced by inclusion:
$$ \iota^{r}_{K,\mathfrak{s}}: H_{*}(\mathcal{F}(b,r,\mathfrak{s}), \widehat{\del}) \rightarrow H_{*}(\widehat{CF}(\h, \mathfrak{s}), \widehat{\del}) = \widehat{HF}(\DBC{K},\mathfrak{s}).$$

Then define functions $r^{min}_{K}$, $r^{max}_{K}$, $r_{K}:\text{Spin}^{c}(\DBC{K}) \rightarrow \mathbb{Q}$ given by
\begin{align*}
r^{min}_{K}(\mathfrak{s}) &= \text{min}\{ r \in \mathbb{Q} |  \iota^{r}_{K,\mathfrak{s}}\neq 0\},\\
r^{max}_{K}(\mathfrak{s}) &= \text{min}\{ r \in \mathbb{Q} |  \iota^{r}_{K,\mathfrak{s}} \text{ is an isomorphism} \},
\quad \text{and}\\
r_{K}(\mathfrak{s}) &= \frac{1}{2}\left( r^{max}_{K}(\mathfrak{s}) + r^{min}_{K}(\mathfrak{s})\right).
\end{align*}

The maps $\iota^{r}_{K,\mathfrak{s}}$ don't depend on the choice of $b$, and so the above functions are indeed knot invariants.

We then set $r(K) = r_{K}(\mathfrak{s}_{0})$, where $\mathfrak{s}_{0} \in \text{Spin}^{c}(\DBC{K})$ is the unique Spin structure.  The reader should compare this definition to those of $\tau$ from \cite{os:4ball}, $s$ from \cite{ras:slice}, and $\delta$ from \cite{cmo:delta}.

In \cite{cmo:delta}, Manolescu and Owens defined a concordance invariant via $\delta(K) = 2 d(\DBC{K}, \mathfrak{s}_{0})$, where $d$ denotes the correction term defined by \OS in \cite{os:abs} and $\mathfrak{s}_{0} \in \text{Spin}^{c}(\DBC{K})$ is the unique Spin structure.  Furthermore, it was shown in \cite{cmo:delta} that whenever $K$ is an alternating knot, $\delta(K) = -\sigma(K)/2.$  Together with Theorem \ref{thm:Rsignthm} above, this implies that if $K$ is a two-bridge knot, then
$$ r(K)= \frac{\sigma(K)}{2} - d(\DBC{K}, \mathfrak{s}_{0}) = \frac{\sigma(K)}{2} - \frac{\delta(K)}{2} = \frac{3 \sigma(K)}{4}.$$

\subsection{Degeneracy}

The unreduced filtration $\rho$ described in \cite{et:R} induces a spectral sequence converging to the group $\widehat{HF}(\DBCs{K})$.  Here we'll state a version of a definition from \cite{et:R}, modified to the case of coefficients in a field $\mathbb{F}$:

\begin{df}\label{def:deg}
Let $\mathbb{F}$ be a field.  A knot $K$ is called \textit{$\rho$-degenerate} if the following hold:
\begin{enumerate}[(i)]
\item The spectral sequence induced by $\rho$ (wth coefficients in $\mathbb{F}$) converges at the $E^{1}$-page.
\item The filtration on $\widehat{HF}(\DBCs{K}; \mathbb{F})$ induced by $\rho$ is constant on each $\text{Spin}^c$-summand $\widehat{HF}(\DBCs{K}, \mathfrak{s}; \mathbb{F})$.
\end{enumerate}
\end{df}

Proposition \ref{prop:RredR} implies that a knot is $\rho$-degenerate if and only if the reduced spectral sequence over $\mathbb{F}$ (induced by $\red{\rho}$) satisfies properties analogous to those listed in Definition \ref{def:deg}.  Since $\red{R} - \red{\rho} = \widetilde{gr}$, Proposition \ref{prop:Rdeg} follows.

Notice that even when a knot is $\rho$-degenerate, the value taken by $r(K)$ may still be of interest.  As suggested in Section \ref{futureconc}, one could ask whether $r$ itself provides a concordance invariant and could compare it other known concordance invariants arising in Heegaard Floer theory and Khovanov homology.  Furthermore, it would be interesting to investigate whether $r = 3\sigma/4$ for all (quasi-)alternating knots.

\subsection{The left-handed trefoil and the lens space $L(3,1)$}\label{sec:example}

Let $K$ be the left-handed trefoil, viewed as the plat closure of the braid $\sigma_{2}^{3} \in \B{4}$.  This braid was studied in Section 3.2 of \cite{et:R}, and the fork diagram shown there is reducible.  We can omit the pair $\alpha_{2},\beta_{2}$ from that fork diagram and obtain via Proposition \ref{prop:DBCred} an admissible Heegaard diagram for $L(3,1)$ of genus 1; this diagram can be seen in Figure \ref{fig:exhdred}.  Let the generators be labelled from left to right as $t', t, \text{ and } x_{1}$
\begin{figure}[h!]
\centering
\begin{minipage}[c]{.45\linewidth}
\labellist 
\small
\pinlabel* {$+\infty$} at 280 250
\pinlabel* {$\ah_{1}$} at 295 190
\pinlabel* {$\bh_{1}$} at 295 30
\pinlabel* {$a$} at 161 132
\pinlabel* {\reflectbox{$a$}} at  309 132
\endlabellist 
\includegraphics[height = 40mm]{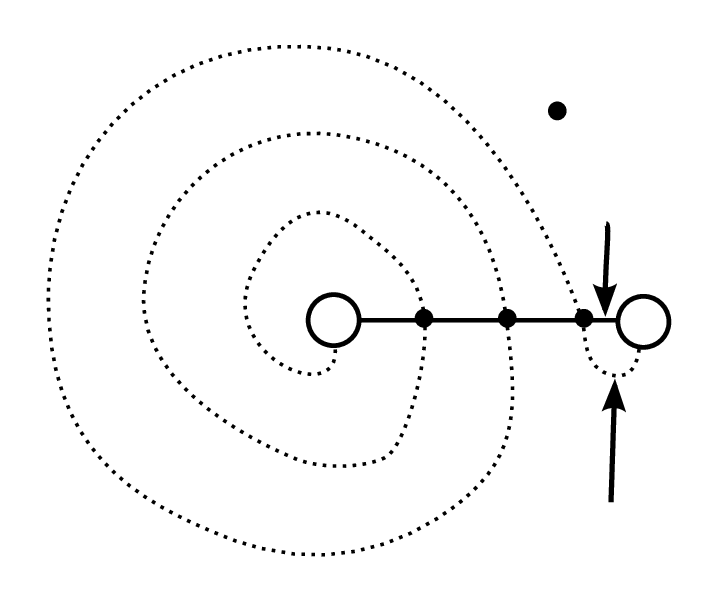}\end{minipage}
\begin{minipage}[c]{.45\linewidth}
\caption{Heegaard diagram  for $L(3,1)$ obtained from $\sigma_{2}^{3} \in \B{4}$}
\label{fig:exhdred}
\end{minipage}
\end{figure}
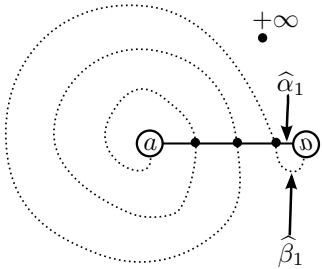

The set of reduced Bigelow generators is $\red{\G} = \{ t, t', x_{1} \}$, and one can verify that all elements occupy the level $\red{R} = 1.$  In this case $\widehat{\del} \equiv 0$, and indeed we see evidence of $\rho$-degeneracy (which was already found in \cite{et:R}).  Thus $\red{R}$ provides an absolute Maslov grading on the group $\widehat{HF}\left(L(3,1); \Zcaltwo\right)$, and
\begin{equation*}
\widehat{HF}\left(L(3,1); \Zcaltwo\right) = \left[\left( \Zcaltwo \right)^{\oplus 3}\right]_{\red{R}=1}.
\end{equation*}

One should observe that $\widehat{HF}(L(3,1); \mathbb{Z}/2\mathbb{Z})$ is supported entirely in the grading level $\red{R} = 1$ (and 1 is half of the classical signature $\sigma(K)$).  Theorem \ref{thm:Rsignthm} states that all two-bridge knots behave this way.

\subsection{Two-bridge knots}\label{sec2bridge}

Let us first give some background on two-bridge knots and links.

\subsubsection{The Conway form of a two-bridge knot or link}

Recall that a two-bridge knot is a knot which has a projection on which the natural height function has exactly two maxima and two minima.  A two-bridge link is defined similarly, with exactly one maximum and one minimum of the height function lying on each of the two components.

\begin{figure}[h]
\centering
\begin{minipage}[c]{.55\linewidth}
\subfigure[$k$ odd]{
\labellist
\small
\pinlabel* {$b_{2}$} at 5 240
\pinlabel* {$b_{k}$} at 45 110
\pinlabel* {$b_{1}$} at 45 340
\endlabellist
\includegraphics[height = 50mm]{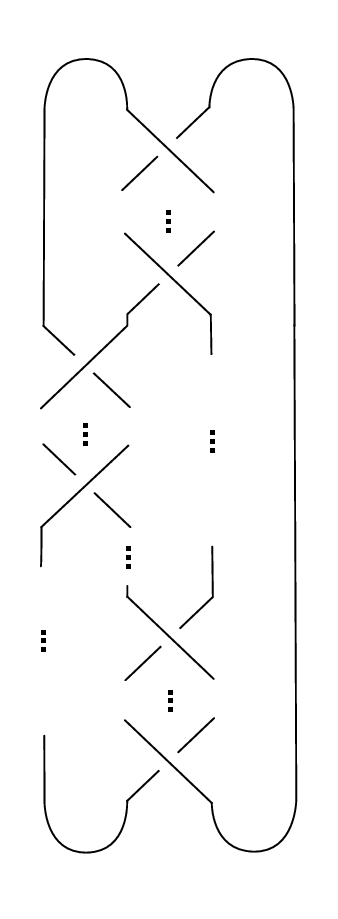}\label{conwayodd}}\quad
\subfigure[$k$ even]{
\labellist
\small
\pinlabel* {$b_{2}$} at 5 265
\pinlabel* {$b_{k}$} at 5 130
\pinlabel* {$b_{1}$} at 45 365
\pinlabel* {$b_{2}$} at 220 265
\pinlabel* {$b_{k}\text{-}1$} at 210 130
\pinlabel* {$b_{1}$} at 260 365
\endlabellist
\includegraphics[height = 50mm]{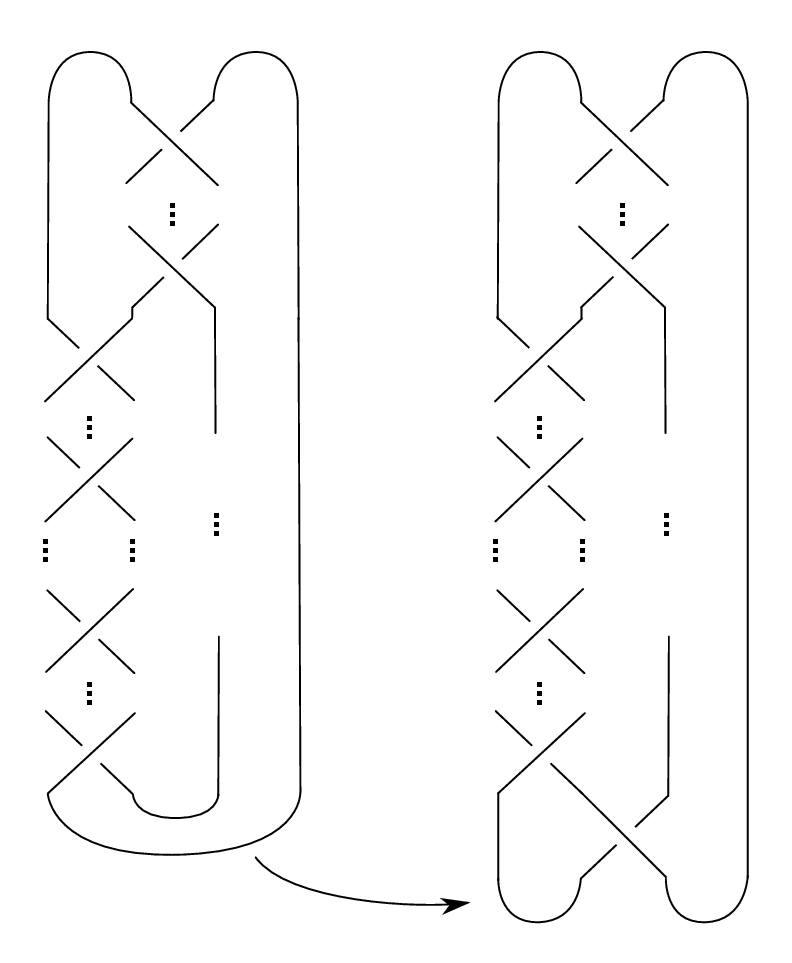}\label{conwayeven}}
\end{minipage}
\begin{minipage}[c]{.40\linewidth}
\caption[Conway forms for knots or links]{Conway forms for knots or links with Conway notation $[b_{1}, \ldots, b_{k}]$ with all $b_j > 0$.  The label ``$b_{j}$'' indicates a bundle of $b_{j}$ half-twists of the indicated direction.}
\label{fig:conway}
\end{minipage}
\end{figure}
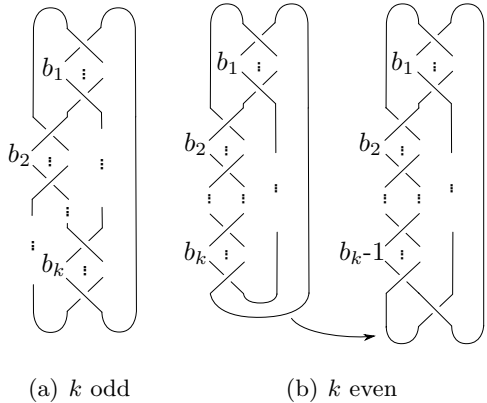

For each two-bridge knot or link $L$, there are nonzero integers $b_{1}, b_{2}, \ldots ,b_{k}$ all of the same sign such that one of the two diagrams in Figure \ref{fig:conway} is a projection of either $L$ or its mirror image (the mirror image must be taken when the numbers $b_j$ are negative).  After possibly performing the isotopy shown in Figure \ref{conwayeven}, we can assume that $k$ is odd.

This diagram is referred to as the Conway form, with Conway notation given by the continued fraction
\begin{equation*}
[b_{1}, \ldots, b_{k}] = b_{1} + \cfrac{1}{b_{2} + \cfrac{1}{b_{3} + \cfrac{1}{b_{4} + \ldots}}} = \frac{p}{q}.
\end{equation*}
\begin{rmk}
If $L$ is a knot or link with Conway notation $[b_{1}, \ldots, b_{k}]  = \frac{p}{q}$, $p$ and $q$ coprime, 
then the two-fold cover of $S^{3}$ branched along $L$ is the lens space $L(p,q)$.
\end{rmk}

For more about the Conway form and Conway notation, see \cite{rolf:knots}.

\subsubsection{The Goeritz matrix and the classical knot signature}

Gordon and Litherland give a formula in \cite{gl:sig} for calculating the signature $\sigma(K)$ of a knot $K$.  Here we review their construction briefly, but one can  find more details in \cite{gl:sig}.

Given a regular projection $D$ of a knot $K$ in $\mathbb{R}^2$, color  the components of $\mathbb{R}^2$---$D$ black and white in a checkerboard fashion, denoting the white regions by $X_{0}, \ldots, X_{n}$.  Denote by $c(X_{i},X_{j})$ the set of crossings of $D$ which are incident to $X_{i}$ and $X_{j}$.  Then assign an incidence number $\eta(C) = \pm 1$ to each crossing $C$ in the projection, following the convention in Figure \ref{fig:crossingeta}.

\begin{figure}[h]
\centering
\begin{minipage}[c]{.45\linewidth}
\labellist
\small
\pinlabel* {$\eta(C) = +1$} at 80 10
\pinlabel* {$\eta(C) = -1$} at 320 10
\endlabellist
\includegraphics[height = 25mm]{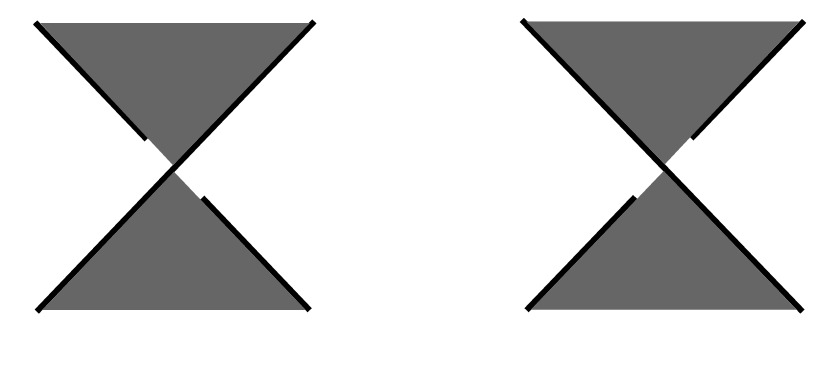}\end{minipage}
\begin{minipage}[c]{.45\linewidth}
\caption[Conventions for $\eta$ on crossings in colored diagrams]{Conventions for $\eta(C)$, where $C$ is a crossing in a colored diagram $D$}
\label{fig:crossingeta}
\end{minipage}
\end{figure}

\begin{df}\label{goeritzdef}
Let $D$ be a regular projection of a knot $K$ equipped with a checkerboard coloring and white regions $X_{0}, \ldots, X_{n}$.  Then let $G'(D)$ be the $(n+1) \times (n+1)$ matrix with
\begin{equation*}
G'(D) = \left(g_{ij}\right)_{i,j = 0}^{n}, \quad \text{where} \quad
g_{ij} = 
\begin{cases} 
-\displaystyle\sum_{c(X_{i},X_{j})}\eta(C)& \text{if $i \neq j$}\\ 
-\displaystyle\sum_{k \neq i}g_{ik} & \text{if $i = j$} .
\end{cases}
\end{equation*}
Then define $G(D)$, the \textit{Goeritz matrix of D}, to be the $n \times n$ symmetric integer matrix obtained from $G'(D)$ by deleting the $0^{th}$ row and $0^{th}$ column.
\end{df}

\begin{figure}[h]
\centering
\begin{minipage}[c]{.45\linewidth}
\labellist
\small
\pinlabel* {Type I} at 77 10
\pinlabel* {Type II} at 320 10
\endlabellist
\includegraphics[height = 25mm]{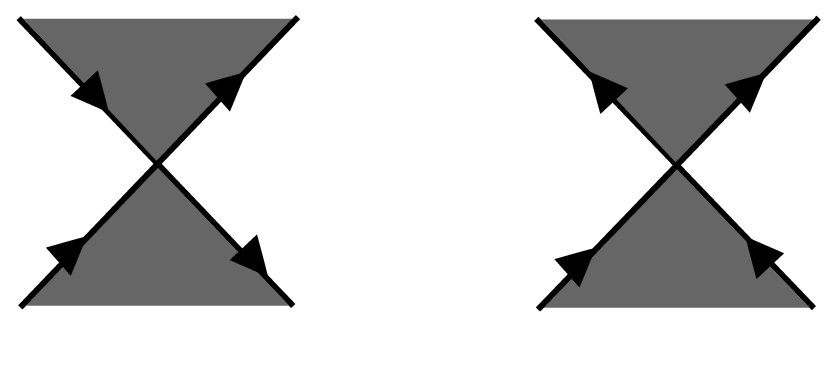}\end{minipage}
\begin{minipage}[c]{.45\linewidth}
\caption{Types of double point orientations counted by $\mu_{I}$ and $\mu_{II}$\label{fig:crossingorient}}
\end{minipage}
\end{figure}

Now fixing an orientation on $K$, we separate the double points of $D$ into types I and II, according to Figure \ref{fig:crossingorient}.  Note that when classifying a crossing in this way, we ignore which strand is passing over the other.

Now we define two integers $\mu_{I}$ and $\mu_{II}$ by
\begin{equation*}
\mu_{I}(D) = \displaystyle\sum_{C \text{ of type I}} \eta(C) \quad \text{and} \quad \mu_{II}(D) = \displaystyle\sum_{C \text{ of type II}} \eta(C).
\end{equation*}

It is shown in \cite{gl:sig} that the quantity $sign(G(D)) - \mu_{II}(D)$ is independent of the projection $D$, and thus an invariant of the knot $K$.  We'll make use of the following theorem:
\begin{thm}[Theorem $6$ from \cite{gl:sig}]\label{GLthm}
Let $K$ be a knot with regular projection $D$.  Then
\begin{equation*}
\sigma(K) = sign(G(D)) - \mu_{II}(D).
\end{equation*}
\end{thm}

We'll need the following fact, the proof of which will be left as an exercise.
\begin{lem}\label{writhelem}
Let $K$ be a knot with oriented regular projection $D$.  Then
\begin{equation*}
w(D) = \mu_{II}(D) - \mu_{I}(D).
\end{equation*}
\end{lem}

\subsubsection{Computations for two-bridge knots}

We'll perform our calculation for a two-bridge knot $K$ using the braid whose closure is the Conway form of $K$.  Further, we assume that the number of Bigelow generators cannot be reduced by an isotopy of the fork diagram.  Because Conway forms never involve the rightmost strand in the braiding, the induced fork diagram will always be reducible.  After reducing the diagram, $\red{\G}$ will be the set of $1$-tuples $\alpha_{1} \cap bE'_{1}$.  First we show that the function $\red{R}$ is very simple for such a reduced fork diagram:

\begin{prop}\label{prop:Rtildeprop}
Let $K$ be a two-bridge knot with Conway notation $[b_{1}, \ldots, b_{k}]$.  Then for any reduced Bigelow generator $\bx \in \red{\G}$ in the special reduced fork diagram above,

\begin{center}
$\red{R}(\bx) =
 \begin{cases} 
\frac{e - w - 2}{4}& \text{if $b_{1} > 0$}\\ 
\frac{e - w + 2}{4}& \text{if $b_{1} < 0$}.
\end{cases}$
\end{center}
where $e$ is the signed count of braid generators in the Conway form and $w$ is the writhe of that diagram.
\end{prop}

Before proving Proposition \ref{prop:Rtildeprop}, we'll have to define some new terminology.

\begin{df}
Let $K$ be an oriented two-bridge knot, and consider the special fork diagram acquired from the Conway form of $K$.  Augment the diagram by adding an extra horizontal tine edge $\alpha$ connecting $\mu_{2}$ and $\mu_{3}$, and give it a vertical handle $h$.  Then we'll call $x \in \left(\alpha - \mu_{2} \right) \cap \beta_{1}$ a \textit{central intersection}.
\end{df}

We then extend the $\red{\tld{R}}$ grading to central intersections in the natural way.  Of course our reduced fork diagrams don't include $\alpha$ or its handle; we will simply use these central intersections as an inductive tool for proving Proposition \ref{prop:Rtildeprop}.

\begin{proof}[Proof of Proposition \ref{prop:Rtildeprop}]

We first prove the proposition for the case where $b_{i} > 0$ for $i = 1, \ldots, k$.  The fork diagram corresponding to only $\sigma_{2}^{b_{1}}$ includes $b_{1}$ Bigelow generators with $\red{\tld{R}} = 0$.  Notice that in such a diagram, when $\beta_{1}$ is incident to $\mu_{i}$, it approaches from below when $i=1$ or $i=3$ and from above when $i=2$.  Furthermore, the actions of $\sigma_{2}$ and $\sigma_{1}^{-1}$ don't affect the directions from which $\beta_1$ approaches these punctures.  Let us develop inductive steps for applying $\sigma_{1}^{-1}$ and $\sigma_{2}$, the building blocks for diagrams of this type.

For the first inductive step, we examine the application of $\sigma_{1}^{-1}$ to an existing braid $b$.  Each existing element $g \in \Gtr$ on the interior of $\alpha_{1}$ spawns one new central intersection $c$ and is itself replaced by another interior $g' \in \Gtr$, with $\red{\tld{R}}(g') = \red{\tld{R}}(g)$ and $\red{\tld{R}}(c) = \red{\tld{R}}(g) + 1$.

We should also examine the effects on an arc terminating at $\mu_{1}$ (case I) or $\mu_{2}$ (case II), as in Figure \ref{fig:sigma1cases}.

\begin{figure}[h]
\centering
\subfigure[Case I]{
\includegraphics[height = 20mm]{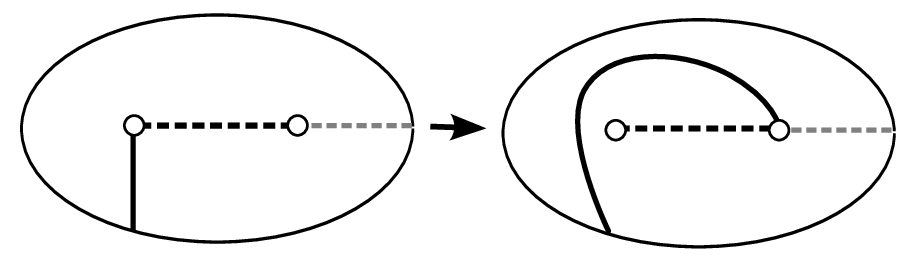}}
\subfigure[Case II]{
\includegraphics[height = 20mm]{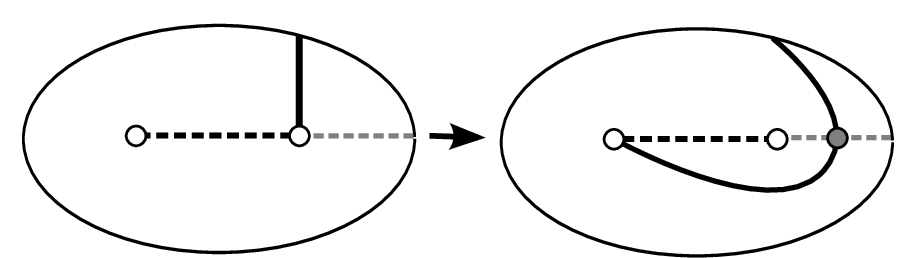}}
\caption[The endpoint cases in the $\sigma_{1}^{-1}$ inductive step in the proof of Proposition \ref{prop:Rtildeprop}]{The endpoint cases in the $\sigma_{1}^{-1}$ inductive step.  The arcs $\alpha_{1}$ and $\alpha$ are dotted black and dotted gray, respectively.}
\label{fig:sigma1cases}
\end{figure}
In both cases, the element $g \in \Gtr$ is replaced by $g' \in \Gtr$.  We see that $\red{\tld{R}}(g') = \red{\tld{R}}(g)$.

In case II, the application of $\sigma_{1}^{-1}$ also spawns a new interior central intersection point, $c$.  We see that $\red{\tld{R}}(c) = \red{\tld{R}}(g) + 1.$

The loops used for grading calculations above can be seen in Figure \ref{fig:sigma1loops}.

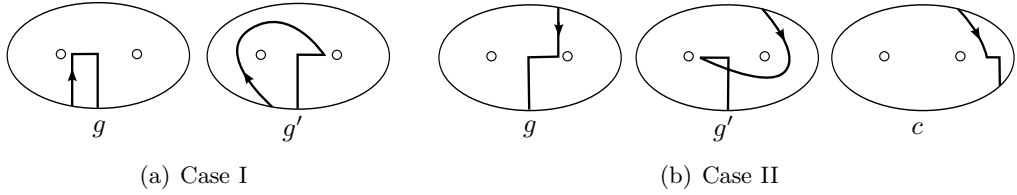
\begin{figure}[h]
\centering
\subfigure[Case I]{
\labellist 
\small
\pinlabel* {$g$} at 100 20
\pinlabel* {$g'$} at 300 20
\endlabellist
\includegraphics[height = 20mm]{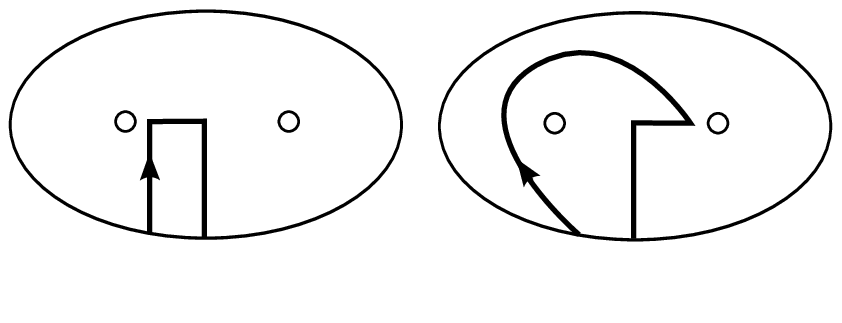}}\quad
\subfigure[Case II]{
\labellist 
\small
\pinlabel* {$g$} at 100 20
\pinlabel* {$g'$} at 300 20
\pinlabel* {$c$} at 500 20
\endlabellist
\includegraphics[height = 20mm]{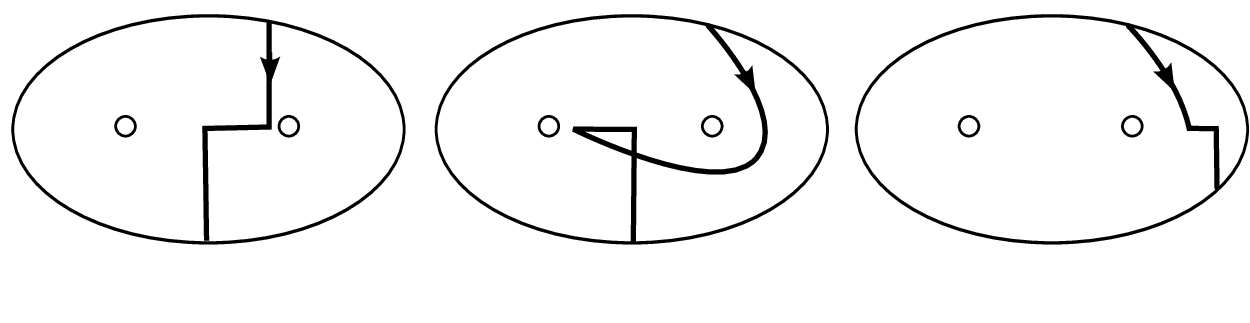}}
\caption[Grading loops associated to the $\sigma_{1}^{-1}$ inductive step in the proof of Proposition \ref{prop:Rtildeprop}]{Grading loops associated to the $\sigma_{1}^{-1}$ inductive step}
\label{fig:sigma1loops}
\end{figure}

For the second inductive step, we examine the application of $\sigma_{2}$ to an existing braid $b$.  Each existing central intersection $c$ spawns a new element $g \in \Gtr$ on the interior of $\alpha_{1}$, and is itself replaced by a new central intersection $c'$.  We see that $\red{\tld{R}}(c') = \red{\tld{R}}(c)$ and $\red{\tld{R}}(g) = \red{\tld{R}}(c) - 1$.

Some new elements of $\red{\Gtil}$ can also result from twisting a strand that originally terminated at $\mu_{2}$ (case I) or $\mu_{3}$ (case II), as shown in Figure \ref{fig:sigma2cases}.

\begin{figure}[h]
\centering
\subfigure[Case I]{
\includegraphics[height = 20mm]{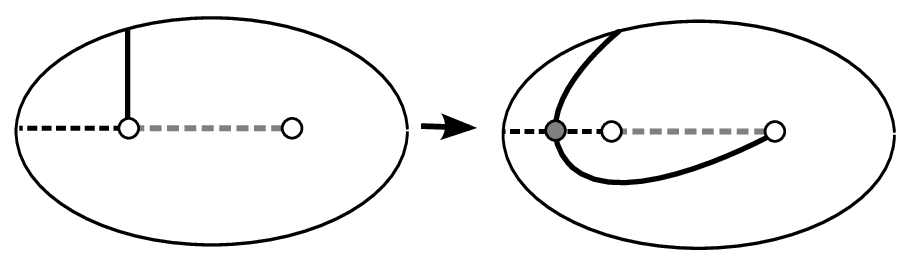}}
\subfigure[Case II]{
\includegraphics[height = 20mm]{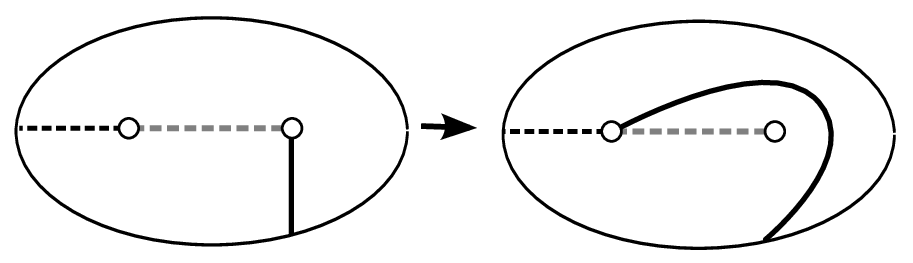}}
\caption[The endpoint cases in the $\sigma_{2}$ inductive step in the proof of Proposition \ref{prop:Rtildeprop}]{The endpoint cases in the $\sigma_{2}$ inductive step}
\label{fig:sigma2cases}
\end{figure}

In case I, we gain a central intersection $c$, and the element $g \in \Gtr$ is replaced by $h \in \Gtr$.  In case II, the central intersection $c$ is replaced by $g \in \Gtr$.  In either case, we have that $\red{\tld{R}}(c) = \red{\tld{R}}(g)+1$ and $\red{\tld{R}}(h) = \red{\tld{R}}(g)$.

The grading comparisons calculated above are demonstrated by the loops in Figure \ref{fig:sigma2loops}.

\begin{figure}[h]
\centering
\subfigure[Case I]{
\labellist 
\small
\pinlabel* {$g$} at 100 20
\pinlabel* {$c$} at 300 20
\pinlabel* {$h$} at 500 20
\endlabellist
\includegraphics[height = 20mm]{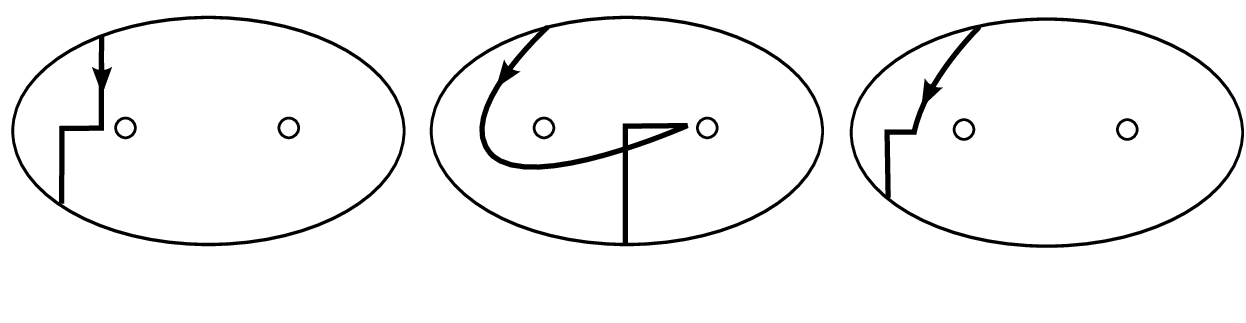}}\quad
\subfigure[Case II]{
\labellist 
\small
\pinlabel* {$c$} at 100 20
\pinlabel* {$g$} at 300 20
\endlabellist
\includegraphics[height = 20mm]{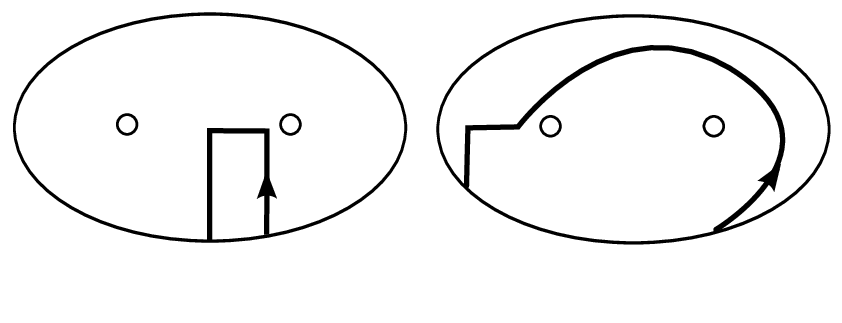}}
\caption[Grading loops associated to the $\sigma_{2}$ inductive step in the proof of Proposition \ref{prop:Rtildeprop}]{Grading loops associated to the $\sigma_{2}$ inductive step}
\label{fig:sigma2loops}
\end{figure}
Using the results of these two inductive steps, we see that all elements of $\red{\Gtil}$ for the $b_{1} > 0$ case lie in the $0$ level of the $\red{\tld{R}}$ grading.

The proof for the $b_{i} < 0$ case contains analogous inductive steps corresponding to applying the braid generators $\sigma_{1}$ and $\sigma_{2}^{-1}$.  The fork diagram corresponding to only $\sigma_{2}^{b_{1}}$ includes $-b_{1}$ Bigelow generators with $\red{\tld{R}} = 1$.  So, all elements of $\red{\Gtil}$ lie in level $1$ of the $\red{\tld{R}}$ grading when $b_{1} < 0$.

We then have that for any $\bx \in \red{\G}$,
\begin{center}
$\red{\tld{R}}(\bx) =
 \begin{cases} 
0& \text{if $b_{1} > 0$}\\ 
1& \text{if $b_{1} < 0$}.
\end{cases}$
\end{center}
and the result for $\red{R}$ follows.
\end{proof}

\begin{lem}\label{goeritzlem}
Let $k$ be odd and let $b_1, b_2, \ldots, b_k$ be nonzero integers which are all of the same sign.  If the knot projection $D$ is given by the plat closure of $\sigma_{2}^{b_{1}}\sigma_{1}^{-b_{2}} \ldots \sigma_{1}^{-b_{k-1}}\sigma_{2}^{b_{k}}$ (as in Figure \ref{conwayodd} or its mirror image),  then
\begin{equation*}
sign(G(D)) = 
\begin{cases} 
-\bigg(\displaystyle\sum_{i\text{ even}} b_{i}\bigg)- 1& \text{if $b_{1} > 0$}\\
-\bigg(\displaystyle\sum_{i\text{ even}} b_{i}\bigg) + 1& \text{if $b_{1} < 0$}.
\end{cases}
\end{equation*}
\end{lem}

When proving Lemma \ref{goeritzlem}, we'll make use of the following fact.

\begin{lem}\label{PDlem}
For $i = 1, 2, \ldots, m$, choose real numbers $a_{i} \geq 2$.  Then the matrix given by
$$
A_m:= \left(\begin{array}{ccccc}
a_1 & -1 & & &\\
-1 & a_2 & \ddots & &\\
& -1 & \ddots & -1 &\\
& & \ddots & a_{m-1} & -1\\
& & & -1 &a_m
\end{array}\right)
$$
is positive-definite (the entries specified are the only non-zero ones).
\end{lem}
\begin{proof}[Proof of Lemma \ref{PDlem}]
For any $X:=\left( x_1, x_2, \ldots, x_m \right)^{T} \in \mathbb{R}^m$,
\begin{align*}
X^T A_m X &= x_1 \left( a_1 x_1 - x_2 \right) + x_m \left( a_m x_m - x_{m-1}\right) + \sum_{i = 2}^{m-1} x_i \left( a_i x_i - x_{i-1}-x_{i+1}  \right)\\
&\geq x_1 \left( 2 x_1 - x_2 \right) + x_m \left( 2 x_m - x_{m-1}\right) + \sum_{i = 2}^{m-1} x_i \left( 2 x_i - x_{i-1}-x_{i+1} \right)\\
& = x_1^2 + x_m^2 + \sum_{i = 1}^{(n-1)} \left( x_i - x_{i+1}\right)^2.
\end{align*}
\end{proof}

\begin{proof}[Proof of Lemma \ref{goeritzlem}]
We first consider the case in which $b_{1} > 0$.  Color the components of $\mathbb{R}^{2}$---$K$ such that the exterior region is black.  Then there are $b+1$ white regions, where $ b:=\displaystyle\sum_{i\text{ even}}  b_{i} + 1$.
 Label the white regions such that $X_0$ is the region on the right side of the Conway form diagram and $X_1, \dots, X_{b}$ are the ones on the left, labelled from top to bottom.  The Goeritz matrix $G(D)$ is a $(b \times b)$ matrix given by
 $$
\left(\begin{array}{ccccc}
a_1 & 1 & & &\\
1 & a_2 & \ddots & &\\
& 1 & \ddots & 1 &\\
& & \ddots & a_{b-1} & 1 \\
& & & 1 &a_b
\end{array}\right)
\quad \text{where} \quad
a_{i} = \begin{cases}
	-(b_{j} + 1) & \text{if $i = b_1$ or $i = b_k$}\\
	-(b_{j} + 2) & \text{if $i = b_j$ with $ 1 \leq j \leq k$}\\
	-2 & \text{otherwise}
\end{cases}
$$
Again, the entries shown are the only nonzero ones.

%

In particular, $G(D)$ is negative-definite by Lemma \ref{PDlem}, and
\begin{equation*}
sign(G(D)) = -b = -\bigg(\displaystyle\sum_{i\text{ even}}b_{i} \bigg) - 1.
\end{equation*}
When $b_{1} < 0$, the Goeritz matrix is positive-definite and has signature
\begin{equation*}
sign(G(D)) = -\bigg(\displaystyle\sum_{i\text{ even}}b_{i} \bigg) + 1.
\end{equation*}
\end{proof}

\begin{proof}[Proof of Theorem \ref{thm:Rsignthm}]
By Proposition \ref{prop:Rtildeprop}, all generators have the same $\red{R}$-filtration value.  However, Equation \ref{eqn:diff} implies that nonzero components of the differential $\widehat{\partial}$ on the filtered complex must decrease the $\red{R}$-level by at least one; consequently, $\widehat{\partial} \equiv 0$.  The knot $K$ is thus $\rho$-degenerate, since $\text{rk}\left( \widehat{HF}\left( \DBC{K}, \mathfrak{s}\right)\right) = 1$ for each $\mathfrak{s} \in \text{Spin}^{c}(\DBC{K})$.
For $b_{1}>0$, we have that
\begin{equation*}
\frac{e - w-2}{4} = \frac{1}{4}\Bigg( 
	\bigg(
		\displaystyle\sum_{i=1}^{k} (-1)^{i+1}b_{i}
	\bigg) 
	+ \mu_{I} - \mu_{II} - 2
\Bigg),
\end{equation*}
and
\begin{equation*}
\sigma(K) = sign(G(D)) - \mu_{II}(D) = -\Bigg(
	\displaystyle\sum_{i\text{ even}} b_{i}
\Bigg)- 1 - \mu_{II}(D).
\end{equation*}
Furthermore, it is easy to see that
\begin{equation*}
-(\mu_{I}(D) + \mu_{II}(D)) = \displaystyle\sum_{i = 1}^{k} b_{i}.
\end{equation*}
So, for each generator $\bx$, we then have that
\begin{align*}
 2\red{R}(\bx) - \sigma(K)& = \frac{1}{2} \Bigg( 
 	\bigg(
		\displaystyle\sum_{i=1}^{k} (-1)^{i+1}b_{i}
	\bigg)
	+2\bigg(
		\displaystyle\sum_{i \text{ even}}b_{i}
	\bigg)
	+ \mu_{I} + \mu_{II}
\Bigg)\\
&=  \frac{1}{2} \Bigg( 
 	\bigg(
		\displaystyle\sum_{i=1}^{k} b_{i}
	\bigg)
	+ \mu_{I} + \mu_{II}
\Bigg ) = \frac{1}{2}(0) = 0.
\end{align*}
The calculation for the $b_{1} <0$ case is similar.
\end{proof}

\section{LINKS}\label{sec:links}

As mentioned in Remark \ref{rmk:links}, the constructions appearing here (and indeed, the unreduced constructions appearing in \cite{et:R}) can be partially extended to the case of a link of several components.  The proof of invariance of filtered chain homotopy type in the unreduced case (Theorem 1 of \cite{et:R}) doesn't rely on the number of components, and so carries through to the case of a link.  A version of Proposition \ref{prop:RredR} holds for links, and so an analogue of Theorem \ref{thm:redRthm} does as well (as long as one restricts to torsion $\text{Spin}^c$-structures).  In particular, although it can be arranged (by choosing the right braid representative) that the deleted pair $\alpha_n$, $\beta_n$ belongs to any link component we like, the $\red{\rho}$-filtered chain homotopy type of $\widehat{CF}$ doesn't depend on this choice.

If $L$ is a two-bridge link, then $b_1(\DBC{L}) = 0$, and so all $\text{Spin}^c$-structures are torsion.  This allows one to define the filtrations $\red{R}$ and $\red{\rho}$ on the entire complex $\widehat{CF}(\DBC{L})$, and indeed Theorem \ref{thm:Rsignthm} holds for links as well (with the proof unchanged).

\section{FUTURE DIRECTIONS}

\subsection{$\rho$-degeneracy}
We have seen that when the knot $K$ is $\rho$-degenerate, then the function $\red{R}$ provides an absolute Maslov grading on the group $\widehat{HF}(\DBC{K})$.  It was established above that this occurs when $K$ is a two-bridge knot.  It is natural to ask whether a larger class of knots is $\rho$-degenerate, such as alternating knots or quasi-alternating knots.  Following Seidel and Smith in \cite{ss:R2}, we have conjectured above that all knots are $\rho$-degenerate.

\subsection{A possible concordance invariant}\label{futureconc}

Knot homology theories and Floer homology theories have been known to produce invariants of the knot concordance class.  Examples include Rasmussen's $s$-invariant in Khovanov homology \cite{ras:slice}, Manolescu and Owens's $\delta$-invariant in Heegaard Floer homology \cite{cmo:delta}, and Ozsv\'ath and Szab\'o's $\tau$-invariant in Knot Floer homology \cite{os:4ball}.  Conjecture \ref{conj:conc} speculates that $r$ also provides one.

Recall that the concordance invariants $s$, $\delta$, and $\tau$ have all been shown to be equal to a constant multiple of $\sigma$ (the classical signature) when restricted to the set of alternating knots.  We saw in Section \ref{sec:r} that $r = 3\sigma/4$ for two-bridge knots; one should ask whether this relationship extends to a larger class of knots, such as alternating of quasi-alternating knots.  In any case, one should compare $r$ to $s$, $\delta$, and $\tau$.

\subsection{Knot mutation}
Bloom showed in \cite{bloom:odd} that odd Khovanov homology is invariant under Conway mutation, while \OS proved in \cite{os:mutant} that knot Floer homology can distinguish a Kinoshita-Terasaka knot from one of its mutants.  Viro noted in \cite{viro:cover} that mutant links have homeomorphic double branched covers, and thus can't be distinguished by the Heegaard Floer homology groups discussed here.  However, one could ask whether this extra filtration structure on the chain complex can distinguish mutant knots.  Whether or not the mutants are $\rho$-degenerate, it is possible that the invariant $r$ could be used to distinguish them.

\subsection*{Acknowledgements}

It is my pleasure to thank Ciprian Manolescu for suggesting this problem to me and for his invaluable guidance as an advisor.  I would also like to thank Robert Lipshitz, Liam Watson, and Tye Lidman for some instructive discussions, Stephen Bigelow for some helpful email related to his paper \cite{big:jones}, Yi Ni for some useful comments regarding relative Maslov gradings, and Sucharit Sarkar for suggesting a definition for $r_{K}$.  Many thanks to Shelly Harvey for suggesting the proof of Lemma \ref{lem:cancel} given here.

I am also greatly indebted to the anonymous Referee, whose corrections and suggestions led to countless improvements to this article.

\bibliography{references}
\bibliographystyle{alpha}
\end{document}